\documentclass[a4paper,11pt]{article}
\usepackage[utf8]{inputenc}
\usepackage{amsmath, amsfonts, amsthm,amssymb,  bbm, dsfont, lmodern, upgreek,euscript,mathrsfs,graphics,latexsym,marginnote}
\usepackage{hyperref,mathtools} 
\mathtoolsset{showonlyrefs}
\usepackage{stmaryrd}
\usepackage[usenames, dvipsnames]{color}
\usepackage[left=2.2cm,right=2.2cm,top=2cm,bottom=2cm]{geometry}
\usepackage{float}

\newtheorem{lemma}{Lemma}[section]
\newtheorem{theorem}[lemma]{Theorem}
\newtheorem{proposition}[lemma]{Proposition}

\newtheorem{definition}[lemma]{Definition}
\theoremstyle{remark}
\newtheorem{remark}[lemma]{Remark}

 \usepackage{mathrsfs}

\def\beq{\begin{equation}}   \def\eeq{\end{equation}}
\def\bea{\begin{eqnarray}}  \def\eea{\end{eqnarray}}
\newcommand{\ii}{{\rm i}}

\newcommand{\comm}[2]{\left\llbracket #1 \,  , \,   #2 \right\rrbracket}
\newcommand{\tm}{\mathtt{m}}
\newcommand{\mS}{\mathcal{S}}
\newcommand{\mM}{\mathcal{M}}
\newcommand{\mR}{\mathcal{R}}
\newcommand{\fA}{\mathsf{A}}
\newcommand{\fB}{\mathsf{B}}

\newcommand{\fR}{\mathsf{R}}
\newcommand{\defeq}{:=}
\newcommand{\fv}{\mathfrak{v}}
\renewcommand{\bar}{\overline}
\newcommand{\cA}{{\mathcal A}}

\newcommand{\cF}{{\mathcal F}}
\newcommand{\cG}{{\mathcal G}}

\newcommand{\cL}{{\mathcal L}}
\newcommand{\cM}{{\mathcal M}}

\newcommand{\cO}{{\mathcal O}}
\newcommand{\cP}{{\mathcal P}}

\newcommand{\cR}{{\mathcal R}}
\newcommand{\cS}{{\mathcal S}}

\newcommand{\wh}[1]{{\widehat{#1}}}

\newcommand{\bF}{{\bf F}}
\newcommand{\bG}{{\bf G}}

\newcommand{\bQ}{{\bf Q}}

\newcommand{\bU}{{U}}

\renewcommand{\Re}{{\rm Re }}
\renewcommand{\Im}{{\rm Im }}

\newcommand{\fa}{{\mathfrak{a}}}

\newcommand{\fP}{{\mathcal{P}}}

\newcommand{\fV}{\mathsf{V}}
\newcommand{\sH}{\mathscr{H}}

\newcommand{\sM}{\mathscr{M}}
\newcommand{\sP}{\mathscr{P}}

\newcommand{\be}{\begin{equation}}
\newcommand{\ee}{\end{equation}}
\newcommand{\vr}{\varrho}
\newcommand{\tr}{{\mathtt r}}
\newcommand{\uno}{{\rm Id}}

\newcommand{\Opw}[1]{{\rm Op}^{\scriptscriptstyle W}\!\left(#1\right)}

\newcommand{\Opbw}[1]{{\rm Op}^{\!\scriptscriptstyle BW}\!\left(#1\right)}
\newcommand{\vOpbw}[1]{{\rm Op}^{\scriptscriptstyle BW}_{{\tt vec}}\!\left(#1\right)}
\newcommand{\zOpbw}[1]{{\rm Op}^{\scriptscriptstyle BW}_{{\tt out}}\!\left(#1\right)}
\newcommand{\vOmega}{{\mathbf{\Omega}}}

\newcommand{\pare}[1]{( #1)}
\newcommand{\bra}[1]{[ #1]}
\newcommand{\mA}{\mathcal{A}}

\newcommand{\mG}{\mathcal{G}}

\newcommand{\mF}{\mathcal{F}}
\newcommand{\mC}{\mathcal{C}}

\newcommand{\mP}{\mathcal{P}}

\newcommand{\opbw}{{{\rm Op}^{{\scriptscriptstyle{\mathrm BW}}}}}

\newcommand{\pa}{\partial}
\newcommand{\im}{{\rm i}}

\newcommand{\C}{{\mathbb C}}

\newcommand{\N}{{\mathbb N}}

\newcommand{\R}{{\mathbb R}}

\newcommand{\T}{{\mathbb T}}
\newcommand{\Z}{{\mathbb Z}}

\newcommand{\norm}[1]{\| #1 \|}
\newcommand{\abs}[1]{\left|#1 \right|}

\newcommand{\la}{\langle}
\newcommand{\ra}{\rangle}
\newcommand{\di}{{\rm d}}
\newcommand{\und}[1]{\underline{#1}}

\newcommand{\ta}{{\mathtt a}}
\newcommand{\tb}{{\mathtt b}}

\newcommand{\td}{{\mathtt d}}

\newcommand{\tf}{{\mathtt f}}
\newcommand{\tg}{{\mathtt g}}

\newcommand{\tk}{{\mathtt k}}

\newcommand{\tq}{{\mathtt q}}

\newcommand{\ts}{{\mathtt s}}

\newcommand{\tv}{{\mathtt v}}

\newcommand{\tz}{{\mathtt z}}

\newcommand{\tF}{{\mathtt F}}

\newcommand{\tI}{{\mathtt I}}
\newcommand{\tJ}{{\mathtt J}}

\newcommand{\tN}{{\mathtt N}}

\newcommand{\tQ}{{\mathtt Q}}
\newcommand{\tR}{{\mathtt R}}

\newcommand{\tV}{{\mathtt V}}

\newcommand{\e}{{\epsilon}}
\newcommand{\s}{{\sigma}}

\newcommand{\lvect}[2]{\begin{pmatrix}#1 \\#2\end{pmatrix}}
\newcommand{\vect}[2]{\Big(\begin{smallmatrix}#1 \\#2\end{smallmatrix}\Big)}
\newcommand{\ov}{\overline}

\newcommand{\x}{\xi }

\newcommand{\wt}[1]{{\widetilde{#1}}}

\newcommand{\fr}{{\mathfrak{r}}}
\newcommand{\fs}{{\mathfrak{s}}}

\numberwithin{equation}{section}

\title{One dimensional energy cascades  \\  in a fractional quasilinear  NLS}
\author{Alberto Maspero\footnote{
	International School for Advanced Studies (SISSA), Via Bonomea 265, 34136, Trieste, Italy. \newline
	\textit{Emails:} \texttt{amaspero@sissa.it}}, \  Federico Murgante\footnote{	 Università degli studi di Milano, Department of Math. Federigo Enriques,
	Via Saldini 50, 20133 Milano, Italy \newline 
	 \textit{Email:} \texttt{federico.murgante@unimi.it}}}

\begin{document}
\date{}
\maketitle 

\begin{abstract}
We consider the problem of transfer of energy to high frequencies in  a quasilinear  Schr\"odinger equation with sublinear dispersion, on the one dimensional torus. 
We exhibit  initial data undergoing finite but arbitrary large Sobolev norm explosion: their initial norm is arbitrary small in Sobolev spaces of high regularity, but at a later time becomes arbitrary large. 
We develop  a novel mechanism producing instability, which is based on extracting, via paradifferential normal forms, an effective equation driving the dynamics whose leading term is a non-trivial transport operator with non-constant coefficients. We prove that such operator is responsible for energy cascades via a positive commutator estimate inspired by  Mourre's commutator  theory. 
\end{abstract}

\tableofcontents

\section{Introduction}

A fundamental question in physics and mathematical analysis is to study  how energy is transferred and redistributed from macro to micro scales in deterministic systems, being
central to  understand the emergence of turbulent dynamics, 
 specially in fluids. 
Formal computations of energy transfers have been performed since the 1960s by  Hasselmann for the 
pure gravity water waves \cite{Hasselmann1, Hasselmann2},  by Longuet-Higgins and Gill for the  $\beta$-plane
equation \cite{LH}, and more recently for the dispersive surface quasi-geostrophic equation (SQG) \cite{SS}, but still lack rigorous mathematical  justification.

A rigorous way to effectively capture energy transfers is to
construct solutions exhibiting growth of Sobolev norms, as pointed out for example by Bourgain \cite{bou2000} 
in the context of nonlinear Hamiltonian PDEs.
Whereas an active line of research -- starting from the breakthrough  work by Colliander-Keel-Staffilani-Takaoka-Tao \cite{CKSTT}-- has rigorously proved growth of Sobolev norms  for certain  {\em semilinear} Schr\"odinger equations  \cite{hani14, guardia_kaloshin, haus_procesi15, guardia_haus_procesi16, GHMMZ, GG, Giu}, there are  no rigorous results for {\em quasilinear dispersive} equations, even though the most relevant dispersive models in fluid dynamics -- such as those mentioned at the very beginning-- are of quasilinear type.  

This is due to several difficulties. The first one,  common for all dispersive equations, is that  the linearized waves merely oscillate over time and consequently any growth in Sobolev norms is a purely  nonlinear mechanism, making  the analysis particularly challenging.
A further   difficulty, specific to  quasilinear PDEs on compact manifolds, is that  global well posedness is (usually)  not known,  in contrast with the (subcritical) semilinear setting.
In addition,   growth of Sobolev norms happens on time scales longer than those predicted by the long-time Cauchy theory (obtained  via modern quasi-linear normal forms and modified energy methods), posing the problem of constructing solutions  with a  lifespan longer than the expected one.

This paper aims to initiate a rigorous study of energy transfers in {\em quasilinear dispersive PDEs} by proposing a new paradigm for constructing solutions that exhibit growth of Sobolev norms, and which we believe could  serve as a foundational framework to rigorously study energy transfers in  dispersive fluid equations, such as those mentioned at the beginning.
Note that the pure gravity water waves, the $\beta$-plane equation and the dispersive SQG share  two common   features: 
  a nonlinear transport term and a sublinear dispersion relation.
We propose a simplified model retaining  exactly these features, and employ it as a theoretical test-bed to explore our  new mechanism.

Specifically, we consider the   fractional quasilinear NLS (nonlinear Schrödinger) equation 
\begin{equation}
\label{eq:main}
 \pa_t u    =  - \im |D|^\alpha  u +  |u|^2 u_x  , \quad  x \in \T:= \R/2\pi\Z  \ , \quad \alpha \in (0,1)  \ ,
\end{equation}
with $|D|^\alpha$ the Fourier multiplier defined by  $|D|^\alpha e^{\im k x} = |k|^\alpha e^{\im k x}$, $k \in \Z$.
 Note that, by  energy methods and  in view of the hyperbolic structure of the nonlinearity,  equation \eqref{eq:main} is locally wellposed\footnote{In particular, ill-posedness phenomena \`a la Christ \cite{Christ}, which require non-hyperbolic nonlinearities like  $u^{p-1}u_x$, do not happen for \eqref{eq:main}} in $H^s(\T, \C)$ for any  $s>\frac32$, see  Remark \ref{rem:LWP}. 
Here
$H^s:=H^s(\T, \C)$, $s \in \R$, is the Sobolev space  with norm 
$$
\norm{u(t)}_s^2:= \sum_{k \in \Z} \la k \ra^{2s} \, | u_k(t)|^2 \ , \quad   \la k \ra:=\max(1, |k|)   \ , 
$$
and $u_k(t):= \frac{1}{2\pi}\int_{\mathbb{T}}u(x) e^{- \im k x }\,\di x $ is the $k$-th Fourier coefficient. 

Equation \eqref{eq:main} is also gauge invariant, so the $L^2$-norm is constant in time. Therefore, a growth in time of the $H^s$ norm, $s \gg 1$, indicates a transfer of  energy   to high frequencies.
Our  main result  is the construction of a solution with   Sobolev norm  
 arbitrary small at  initial time, but  arbitrarily large at a later one. Precisely we prove:
\begin{theorem}\label{thm:main}
There exists $s_0> \frac32$ such that  given   any
$s >  3 s_0$,   $0< \delta \leq 1$ and $K \geq 1$, 
 there exists a solution $u(t) \in H^s(\T, \C)$ of \eqref{eq:main}  and a time $T>0$ such that 
 \begin{equation}\label{init}
 \norm{u(0)}_s \leq \delta 
 \quad
\mbox{ and }
\quad 
 \norm{u(T)}_s \geq K  \ .  
 \end{equation}
 Moreover $$\sup\limits_{0 \leq t \leq T}\norm{u(t)}_{s_0} \leq 2 \delta \ . $$
\end{theorem}
Theorem \ref{thm:main} guarantees the existence of a solution of \eqref{eq:main} with smooth and arbitrary small initial datum  undergoing  finite but arbitrary large Sobolev norm explosion.
Such solution has constant $L^2$-norm and  stays small in the ``low'' $H^{s_0}$-norm. 
Local Cauchy theory, given by energy methods, implies that $\norm{u(t)}_s \leq 2 \delta$ for all times $|t| \leq C\delta^{-2}$, see Remark \ref{rem:LWP}; we show that Sobolev norm explosion happens on the just longer timescale  $T \sim \delta^{-2} \log(\delta^{-1})$.
Of course, one of the crucial difficulties is to ensure  existence of the solution over this longer timescale.\\
We do not know the fate of such  solution after time $T$, 
and since global existence for \eqref{eq:main} is not established, we cannot exclude the possibility that, after time $T$,  energy cascades trigger a finite-time singularity formation. 
We remark that, in similar models  such as the fractional KdV equation,  solutions with large initial data can develop shocks \cite{CE, KS, H1, H2, Y,OP, KSW}, resulting in the $H^1$ norm exploding while the  $L^\infty$ one stays bounded.
 However, these shock solutions appear distinct from those described in our Theorem\ref{thm:main}, for which we ensure that  low Sobolev norms stay small.
 
On the other end, not every  initial data gives rise to turbulent solutions of \eqref{eq:main}: consider for example the  plane waves 
$a e^{\im (k x - \omega t)}$ with $\omega= |k|^\alpha - a^2 k$, which can be made of arbitrary small size.
We also expect that KAM methods, like those developed in \cite{BBHM, BFM2, FG}, would enable the construction of globally defined, small-amplitude, time quasi-periodic solutions, demonstrating the coexistence of stable and unstable dynamics.

\smallskip

As mentioned earlier, the primary novelty of this paper is the introduction of a new mechanism for generating energy cascades,  tailored to quasilinear dispersive PDEs with a sublinear dispersion relation and a nonlinear transport term.
{In brief, such structure allows us to  extract, via a novel  quasilinear normal form,  a transport operator with {\em absolutely continuous spectrum}, that drives the dynamics of \eqref{eq:main}, inducing dispersive effects in frequency space and resulting in the growth of Sobolev norms.}

Such  mechanism is entirely distinct from the {\em only two} existing ones developed for semilinear Hamiltonian PDEs: 
the first one, pioneered by Colliander-Keel-Staffilani-Takaoka-Tao \cite{CKSTT}, exploits the dynamics of the so-called ``toy model'' and works for semilinear NLS on $\T^d$, $d \geq 2$, and some related models  \cite{CKSTT,hani14, guardia_kaloshin, haus_procesi15, guardia_haus_procesi16, GHMMZ, GG, Giu}. 
The second one, discovered by Gérard-Grellier \cite{gerard_grellier}, leverages the peculiar integrable structure of the  Szeg\H{o} equation.
We stress again that, in all these models, the nonlinearity is semilinear, in contrast to all relevant dispersive PDEs coming from fluids  which are  {quasilinear}.

Let us now describe better  our  mechanism. 
After a paradifferential normal form \`a-la Berti-Delort \cite{BD}, we conjugate equation \eqref{eq:main} to 
\be\label{eq.intro00}
\pa_t w = -\im |D|^\alpha w + \Opbw{ \im \la \und \tV \ra(u(t);x) \xi } w  {+ \mbox{ quasilinear remainders} }
\ee
where  $\Opbw{\cdot}$ is a Bony-Weyl paradifferential operator (see \eqref{BW}) of order one, coming from the nonlinearity of \eqref{eq:main}, and with the  transport term having {\em non-constant  coefficient}
\begin{equation}\label{und.V.intro}
    \langle \,\und \tV\, \rangle(u(t);x) := 
		2\,  \Re \Big(\sum_{n \in \N} u_{n}(t) \, \bar{u_{-n}(t)} \, e^{\im 2n x} \Big)  \ .
\end{equation}
This normal form is  significantly different from the one of Berti-Delort \cite{BD} and of \cite{FI, BFP, BFF, BMM, MMS},
where the symbol of the paradifferential operator has  {\em constant} coefficients (at least at low homogeneity). 
It is also very different from the normal form of \cite{CKSTT}: indeed 
 the nonlinear vector field in 
\eqref{eq.intro00} is {\em not Birkhoff-resonant}, since the main term 
$\Opbw{ \im \la \und \tV \ra(u(t);x) \xi } w$ has phases of oscillations  given by
$$
|n|^\alpha - |-n|^\alpha + |j+2n|^\alpha - |j|^\alpha \neq 0 \ , \quad \forall  n \in \N \ , j \in \Z  \ ;
$$  
 in principle it might be eliminated by a (formal) Birkhoff normal form procedure, but the required transformation is unbounded and not well defined in $H^s$, due to the quasi-linear nature of the problem. 
 Actually, it will be exactly this  term to drive the instability: 
 energy cascades are due to quasi-resonant interactions rather than exact resonances; this is 
 reminiscent, in wave turbulence,  to the fact that are quasi-resonances (rather than resonances) to play a fundamental role in the rigorous derivation of the  wave kinetic equation \cite{DH}.

Note that the normal form \eqref{eq.intro00} guarantees only a cubic lifespan $\sim \delta^{-2}$ for  initial data of size $\delta\ll 1$,  which is  too short  to observe any energy transfers phenomena. 
Here come the first novelty of our method. 
We give up the  control of  {\em any} solution for times longer than $\sim \delta^{-2}$, and   restrict to particular solutions whose initial data is mostly concentrated on the two Fourier modes $\Lambda:=\{-1, 1\}$. 
Via an ad-hoc normal form, we  decouple the dynamics of the modes in $\Lambda$ and in $\Lambda^c$, and prove that such special solutions are  long-time controlled: with this we mean that, 
on the enhanced  timescale
 $\delta^{-2}\log \delta^{-1}$, 
the modes in $\Lambda$ evolve essentially as rotations, whereas the modes on $\Lambda^c$ remain of very small size in a low $H^{s_0}$ norm. 
In addition, we prove that  long-time controlled solutions fulfill an effective system of the form 
\be\label{intro.eff}
\pa_t \zeta  =   - \im |D|^\alpha \zeta   + \im \Opbw{{(\tJ_1+\fv(x ))} \xi }\zeta  {+ \mbox{ quasilinear remainders} }
\ee
Here  $\tJ_1$ is a real number and $\fv(x)$ a real valued function, both depending nonlinearly on the initial data $u(0)$ (see \eqref{J1} and \eqref{V0}).
We  develop a new robust way to prove that    \eqref{intro.eff} has solutions undergoing  growth of Sobolev norms. 
To do so, we  extend to the nonlinear setting a  positive commutator method, inspired by  Mourre's  theory \cite{Mourre}.
Precisely, we construct a paradifferential operator $\fA$, see \eqref{fA}, such that the commutator 
$$\im [\fA, \Opbw{{(\tJ_1+\fv(x ))} \xi }]$$
is strictly positive on large frequencies up to a small remainder. 
This is possible provided the function $\tJ_1 + \fv(x)$ does not have sign, a condition that we force    by tuning the  initial datum.
This condition carries significant meaning: it ensures that the operator $\Opbw{{(\tJ_1+\fv(x ))} \xi }$ has non-trivial {\em absolutely continuous spectrum.} 
This feature is the key factor driving energy transport to high frequencies: it induces   a dispersive effect in the energy space that is directly analogous, in frequency variables, to the classical mechanism of spatial mass transport to infinity in Schr\"odinger equations on Euclidean spaces.


A further benefit of our method  is that it allows us to prove that $\zeta(t)$ grows at an exponentially fast rate.
This is  due to the quasilinear nature of equation \eqref{eq:main}: for semilinear NLS, polynomial upper bounds in time are known (see e.g. \cite{bou96, Staffilani, Sohinger, PTV}), which become subpolynomial in time  for linear  time-dependent Schr\"odinger equations (see e.g. \cite{bourgain99, Delort10, MaRo, BGMR2, BLM, BL}).



\smallskip

\noindent{\bf Related literature:}
Whereas for linear time dependent equations several results are known \cite{bou99,del, Mas19, BGMR1,LZZ, LLZ, LLZ2, BGMRV, FaouRaph, HausMaspero, Mas21, Mas22}, 
for nonlinear systems, as we already mentioned,   the results are scarce and limited to essentially two models: the semilinear Schr\"odinger equation (NLS) and certain integrable equations.
Regarding the first, after the seminal works by 
Kuksin \cite{Kuk96,Kuk97}, the
 breakthrough result by Colliander-Keel-Staffilani-Takaoka-Tao \cite{CKSTT} for the NLS on $\T^d$, $d \geq 2$,  
identified the first mechanism of growth, based on the toy-model construction. 
Such mechanism was further exploited by  
Guardia-Kaloshin \cite{guardia_kaloshin}, Haus-Procesi \cite{haus_procesi15}, Guardia-Haus-Procesi \cite{guardia_haus_procesi16}, 
 Guardia-Giuliani \cite{GG} and Giuliani \cite{Giu}. 
All these results construct solutions starting with norm arbitrally small and becoming arbitrarily large at a later time. 
We also mention 
Hani \cite{hani14} and Guardia-Haus-Hani-Maspero-Procesi \cite{GHMMZ} that construct  solutions undergoing  Sobolev norm inflation and starting arbitrary close to periodic or quasi-periodic  orbits. 
Solutions with unbounded paths have been constructed by 
 Hani-Pausader-Tzvetkov-Visciglia \cite{hani15}  for the NLS on $\R \times \T^2$,  combining dispersive effects and the resonant toy-model construction.

The second known mechanism ensuring  growth of Sobolev norms 
was pioneered by  
Gérard-Grellier  \cite{gerard_grellier} for the Szeg\H{o} equation, exploiting its peculiar  integrable structure \cite{gerard_grellier0}. 
We also mention 
Biasi-Evnin \cite{BE} for a truncated  Szeg\H{o} systems,  
 Gérard-Lenzmann \cite{GL} for the integrable Calogero-Moser derivative NLS, and long time instability results  for the cubic half-wave equation obtained by  Gérard-Grellier \cite{gerard_grellier3} on $\T$ and Gérard-Lenzmann-Popovnicu-Raphael \cite{GLPR} on $\R$ (exploiting    resonant approximations with the Szeg\H{o} equation). \\
Furthermore we mention Guardia-Giuliani \cite{GG2} for chains of infinite pendula, the recent numerical result by 
Gallone-Marian-Ponno-Ruffo \cite{Gallone}
for the FPUT chain and Elgindi- Shikh Khalil \cite{EK2022} for a completely different norm inflation mechanism in $L^\infty$.

\subsection{Scheme of the proof}
We shall now describe in more details the methods of the proof and the plan of the paper. 

\smallskip
\noindent{\bf Step 1: paradifferential normal form.} 
The first step  is to transform equation \eqref{eq:main} via  the paradifferential normal form pioneered by Berti-Delort \cite{BD},  further developed and extended in  \cite{FI, BFP, FGI, BFF, BMM, BCGS, MMS}. 
While previous applications of the Berti-Delort method aimed primarily at constructing a modified energy to establish {\em upper bounds} on the Sobolev norms of solutions,
our approach leverages the method to extract an effective equation that has unstable  solutions.

In Section \ref{sec:nf}, we perform  two paradifferential transformations to  conjugate the original equation \eqref{eq:main} to the normal form  system \eqref{teo62}, whose cubic component has the form
\be\label{eq.intro1}
\pa_t w = -\im |D|^\alpha w + \Opbw{ \im \la \und \tV \ra(u(t);x) \xi + \im a_2^{(\alpha)}(u(t);x, \xi) } w + R_2(u(t))w + h.o.t.
\ee
with   $\langle \,\und \tV\, \rangle(u(t);x)$ in \eqref{und.V.intro}, $a_2^{(\alpha)}$ a  symbol of order $\alpha$ and quadratic in $u(t)$, and $R_2(u(t))$ a smoothing operator  again quadratic in $u$. 
This normal form is  significantly different from the one of \cite{BD} and of \cite{FI, BFP, BFF, BMM, MMS},
where the symbol of the paradifferential operator has  {\em constant} coefficients (at least at low homogeneity). 
On the contrary, in \eqref{eq.intro1}, $\langle \,\und \tV\, \rangle(u;x)$ has  {\em non-constant coefficients}, and additionally it   depends on time through $u(t)$. 
This is the term who will give rise to the paradifferential operator in \eqref{intro.eff}. 
To do so, 
 we need to remove (or at least simplify) such time dependence.   
The first natural attempt, i.e.  replace
in $\langle \,\und \tV\, \rangle(u(t);x)$ the function 
$u(t)$ with its linear evolution    $e^{- \im |D|^\alpha t}u(0)$, fails because it produces an error that we cannot bound on the long time scales  needed to see growth. 
Therefore, we need to study  the nonlinear dynamics of at least two modes $u_n(t)$, $u_{-n}(t)$.
So we fix the modes in $\Lambda := \{-1, 1 \}$ and study the  nonlinear dynamics of $u_1(t)$, $u_{-1}(t)$.

\smallskip 

\noindent {\bf Step 2: the $\Lambda$-normal form.} We decompose 
the solution as follows: 
\begin{align*} u(t) = u^\top(t) + u^\perp&(t)\quad \text{where}\\     u^\top(t):= u_1(t) e^{\im x} + u_{-1}(t) e^{- \im x} \ , 
&\quad u^\perp(t) := \sum_{k \neq \pm 1} u_k(t) e^{\im k x}. \end{align*}  This decomposition separates the  {\em tangential modes} $u^\top(t)$ from the {\em normal modes} $u^\perp(t)$. To decouple the dynamics of these modes, we use a weak-normal form. The paradifferential operator in equation \eqref{eq.intro1} vanishes when restricted to $\Lambda$ (see \eqref{proj.op}). Therefore, the dynamics of $ u^\top(t) $ is governed by the smoothing operator $ R_2(u)w$.

We  decouple the dynamics of the tangential and normal modes in $R_2(u)w$ by removing from this term two types of monomials $u_{j_1}^{\sigma_1} u_{j_2}^{\sigma_2} u_{j_3}^{\sigma_3} e^{\im k x}$: 
\begin{itemize}
    \item[(i)] \textbf{Monomials with $ (j_1, j_2, j_3) \in \Lambda $ and $ k \in \Lambda^c $}:

    \noindent $\bullet$ This ensures that the set $\Lambda$  remains invariant under the cubic part dynamics of \eqref{eq.intro1};
    
    \noindent $\bullet$ It requires first-order Melnikov conditions:
    \be\label{mel}
    |j_1|^\alpha - |j_2|^\alpha + |j_3|^\alpha - |k|^\alpha \neq 0  \  , \quad j_1-j_2+j_3-k=0 \ ,
    \ee
    that actually {we verify whenever one and only one  among $(j_1, j_2, j_3, k)$ lies in $\Lambda^{{c}}$.}

    \item[(ii)] \textbf{Monomials with exactly two indexes among $(j_1, j_2, j_3)$ in $\Lambda$ and the remaining one and $k$ in $\Lambda^c$}:
    
    \noindent $\bullet$ This is needed so that the leading term in equation \eqref{eq.intro1} 
    is given by the skewadjoint paradifferential term
    $\Opbw{\im \la \und \tV \ra(u_1 \bar{u_{-1}}; x) \xi }w$ (whose monomials have exactly 2 indexes inside $\Lambda$ and 2 outside); 

    \noindent $\bullet$ It requires  second-order Melnikov conditions: 
    $$
    |j_1|^\alpha - |j_2|^\alpha + |j_3|^\alpha - |k|^\alpha \neq 0  \  , \quad j_1-j_2+j_3-k=0 \ ,
    $$  when two indexes among 
    $(j_1, j_2, j_3, k)$
    are in $\Lambda$ and the other two in $\Lambda^c$, provided 
    $j_1 \neq j_2$ or $j_1 \neq k$.
\end{itemize}
As a result, only integrable monomials of the form $|u_{j_1}|^2 u_{j_3} e^{\im j_3 x}$, with either $j_1, j_3 \in \Lambda$ or $j_1 \in \Lambda,\, j_3 \in \Lambda^c$ or viceversa are left in the smoothing operator $R_2(u)w$.
Finally, in Proposition \ref{prop.megliodemax2}, 
we identify the remaining resonant integrable 
monomials via an a-posteriori identification argument \`a la Berti-Feola-Pusateri \cite{BFP} (see also \cite{BFF}), obtaining the explicit form \eqref{Y.str}.

\smallskip

\noindent {\bf Step 3: The effective equation.}
The  
variables $z^\top(t)$ and $ z^\perp(t)$ solve   system 
\eqref{eq.ztop}--\eqref{eq.zperp}, which has roughly the form 
\be\label{intro.sys}
\begin{cases}
    \pa_t z^\top = - \im |D|^\alpha z^\top + Y^{(\Lambda)}_3(z^\top(t)) + \cO(\norm{z^\perp}_{s_0}^3, \norm{z}_{s_0}^5)\\
    \pa_t z^\perp = \- \im |D|^\alpha z^\perp + \Opbw{\im \la \und{\tV} \ra(z^\top(t);x) \xi + \im a^{(\alpha)}_2(z(t); x, \xi)}z^\perp  + \cO(\norm{z^\top}_{s_0} \norm{z^\perp}_{s_0} \norm{z^\perp}_s)
\end{cases}
\ee
where $Y^{(\Lambda)}_3(z^\top)$ is the explicit {\em integrable} vector field \eqref{Y.int}, and the symbol of the transport operator in the equation for $z^\perp$ is evaluated only on the tangential modes $z^\top(t)$.

To further understand the dynamics of system \eqref{intro.sys} and  to extract from it the effective equation \eqref{intro.eff}, we introduce a small parameter $ \e\ll \delta \leq 1$ and  we consider 
special solutions of system \eqref{intro.sys}, that we call {\em long-time controlled} (see Definition \ref{A}).
They are characterized by two properties: 
\begin{itemize}
    \item[(i)] Their initial data 
 are small in $L^2$, with most mass on the modes $z_1(0), z_{-1}(0)$:
 $$
 \norm{ z^\top(0, \cdot)}_{L^2} \leq \e, \quad \norm{ z^\perp(0, \cdot)}_{L^2} \leq \e^{3};
 $$
 \item[(ii)] Their  high $H^s$-norms have 
 large a-priori bounds:
  $$\norm{z(t)}_s \leq \e^{-\theta} \quad \text{with } 0<\theta \ll 1\ . $$
\end{itemize}
Note that  the large a-priori bound above  is not restrictive for our problem: if it  fails, it means the solution has already grown.
We then prove that {\em any} long-time controlled solution, on the enhanced timescale   $| t| \lesssim \e^{-2} \log(\e^{-1})$, has:

 $ \bullet$ The 
 modes $z_1(t)$ and $ z_{-1}(t)$ evolving very close to the rotations: 
 $$z_{\pm 1}(t)=  e^{- \im t(1\pm |z_{\pm 1}(0)|^2)} z_{\pm 1}(0)+ \cO( \e^{3-\theta});$$
 
 $\bullet$ The  ``low'' $H^{s_0}$-norm of $z^\perp(t)$ staying  very small,  i.e.  $\norm{z^\perp(t)}_{s_0} \leq \e^2$.
One key idea to obtain this is to estimate $z^\perp(t)$ in  $L^2$, exploiting the cancellation coming from the skewadjointness of the paradifferential operator,   then deducing a  bound for  $ \| z^\perp(t)\|_{s_0}$ by interpolation with the large a-priori  bound for 
$ \| z(t)\|_{s}$. 

Finally, we  approximate the evolution of $z^\top(t)$ with the rotations $ e^{- \im t(1\pm |z_{\pm 1}(0)|^2)} z_{\pm 1}(0)$  in the symbol $\la \und{\tV} \ra(z^\top(t);x) $ obtaining a negligible remainder, and, after a space translation, we arrive to an effective system of the form \eqref{intro.eff}, see Proposition \ref{prop:eff}.

\smallskip

\noindent {\bf Step 4: Growth of Sobolev norms.}
After this analysis, we have essentially reduced the problem  to 
construct solutions of the effective equation \eqref{intro.eff} undergoing  growth of Sobolev norms.
We construct a paradifferential  operator $\fA$, of order $2s$ and supported on high-frequencies, see \eqref{fA}, fulfilling the positive commutator estimate (Lemma \ref{lem:pos.comm})
\begin{equation}\label{intro.comm}
    \im [\fA, \Opbw{(\tJ_1 + \fv(x))\xi}] \geq \tI_1 \Opbw{|\xi|^{2s} \eta_\tR^2(\xi)} + h.o.t.
\end{equation}
Here $\tI_1$ is a strictly positive real number   depending  on the initial data, see  
\eqref{tI1},  and $\eta_\tR$ a cut-off function on  high frequencies.
To obtain such positive commutator estimate, the main ingredient is to find a symbol $\fa(x, \xi)$ which is an escape-function for the dynamics of $(\tJ_1 + \fv(x))\xi$, namely such that  the Poisson bracket $\{\fa(x, \xi), (\tJ_1 + \fv(x))\xi \} $
is strictly  positive. 
This is possible provided the function $\tJ_1 + \fv(x)$ does not have sign; since 
$$
\tJ_1 + \fv(x)=   \frac{|z_1(0)|^2 + |z_{-1}(0)|^2}{2}  + 2\Re\big( z_1(0) \bar{z_{-1}(0)} e^{\im 2x } \big) , 
$$
it is enough to select the values of the initial modes $z_{\pm 1}(0)$ so that 
$\frac{|z_1(0)|^2 + |z_{-1}(0)|^2}{2} < 2  |z_1(0) | \, |{z_{-1}(0)} |$. 
The same condition yields the strict positivity of the  number $\tI_1$ in \eqref{intro.comm}.
An important point is that the operator $\fA$ is chosen to be supported on very large $|\xi| \geq \tR \geq \e^{-\frac{3+\theta}{1-\alpha}}$. 
This is required so that the   dispersive term $-\ii |D|^\alpha$ and all the other lower order operators becomes perturbative with respect to the leading transport. 
To conclude, we define the functional $\cA(t):= \la \fA z^\perp, z^\perp \ra$ and show that  \eqref{intro.comm}
leads to a lower bound for the dynamics of $\frac{\di}{\di t} \cA(t)$, forcing  $\cA(t)$ to grow exponentially fast provided $\cA(0)$ is not too small, a condition that can be imposed by  well-preparing the  initial data. Being $ \mA(t)\lesssim\| z^\perp(t)\|_s^2$,  growth of Sobolev norms follows.

 \section{Functional setting}

 In this section we introduce the  paradifferential operators and smoothing remainders,  following \cite{BD, BMM}.
 We also introduce a new class of transformations, that we call {\em admissible transformations}, see Definition \ref{admtra}.
 They are maps $U \mapsto \bF(U)$ whose main property is to be of regularity $C^1$ with respect to the internal variable. 
Consequently, the nonlinear map $U \mapsto  \bF(U)U$ results invertible.
We shall prove that all the transformation generated along the normal form reduction of Section 
\ref{sec:nf} are admissible.
\smallskip 
 
\noindent{\bf Function spaces.} Along the paper we deal with real parameters 
$
s\geq s_0  \gg \vr$. {We use the following conventions for the set of natural numbers 
\begin{align*}
\N:= \{ 1, 2, \ldots \}, \quad   \N_0:= \N\cup \{0\}.
\end{align*}}
For $s\in \R$ we shall denote with $H^s(\T;\C^2)$ the space of couples of complex valued Sobolev functions  in $H^s(\T, \C)$  and with 
$$
H^s_\R(\T;\C^2):=\Big\{ U= \vect{ u^+}{ u^-}\in H^s(\T;\C^2)\colon  \ u^-= \bar{u^+}\Big\} \ .
$$
Given $r>0$ we set $B_{s}(r)$ {the ball of radius $r$} in ${H}^s\left(\mathbb{T},\mathbb{C}^2\right)$ and $B_{s,\R}(r)$ the ball of radius $r$ in ${H}^s_\R\left(\mathbb{T},\mathbb{C}^2\right)$.
Given an interval $ I\subset \R$ symmetric with respect to $ t = 0 $ and a Banach space $X$,  we use the standard notation $C(I,X)$ to denote the space of continuous functions with values in $X$. 
Given $r>0$ we set $B_s(I;r)$ the ball of radius $r$ in $C(I,{H}^s\left(\mathbb{T},\mathbb{C}^2)\right)$ and by 
 $B_{s,\R}(I;r)$ the ball of radius $r$ in $C(I,{H}^s_{\R}\left(\mathbb{T},\mathbb{C}^2)\right)$. 
We  denote $ L^2(\T,\C):=  H^0(\T,\C)$ and 
 we define
\be\label{scpr12hom}
\la u, v \ra_{L^2} := \frac{1}{2\pi} \int_\T  u(x)\, \bar {  v(x)} \, \di x \, .
\ee
	Given $N \in \N_0$, we denote by $ W^{N, \infty}  (\T) $ the space of 
	continuous functions $ u : \T \to \C $, $ 2 \pi$-periodic, 
	whose derivatives up to order  $N$ are in $L^\infty$, equipped with the norm 
	$$
	\| u \|_{W^{N,\infty}} := 
	\sum_{\ell=0}^{N} \| \pa_x^\ell u \|_{L^\infty}.
	$$
	For $ N = 0 $ the norm $ \| \cdot  \|_{W^{N,\infty}} = \| \cdot \|_{L^{\infty}}  $.
 
We denote by  $\tau_\varsigma $, $\varsigma  \in \R$, 
and by $\tg_\theta$, $\theta \in \T$,
the translation operator respectively the phase rotation given by 
\be\label{tra}
[\tau_\varsigma u](x) :=u(x + \varsigma)  \ , 
\qquad 
\big[\tg_\theta \vect{u}{\bar u}\big](x) := \lvect{e^{\im \theta }u(x) }{e^{-\im \theta }\bar u(x)}  \  . 
\ee

  \noindent{\bf Symmetries of operators and vector fields.} 
 Given a linear operator  $ A(U)$ acting on $L^2(\T;\C)$ 
we associate the linear  operator  defined by the relation 
\begin{equation}\label{opeBarrato}
\ov{A}(U)[v] := \ov{A(U)[\ov{v}]} \, ,   \quad \forall v: \T \rightarrow \C \, .
\end{equation}
An operator $A$ is {\em real } if $A = \bar A$. 
We say that a matrix of operators acting on $L^2(\T;\C^2)$  is \emph{real-to-real}, if it has the form 
\begin{equation}\label{vinello}
R(U) =
\left(\begin{matrix} R_{1}(U) & R_{2}(U) \\
\ov{R_{2}}(U) & \ov{R_{1}}(U)
\end{matrix}
\right) \, , \quad \forall U \in L^2_\R (\T, \C^2)  \ . 
\end{equation}
A real-to-real matrix of operators $R(U)$  acts in the subspace   
$  L^2_\R (\T, \C^2) $.
If $R(U)$ and $R'(U)$ are  real-to-real operators then also $R(U)\circ R'(U)$ is real-to-real.\\
A  matrix  $R(U)$ as in \eqref{vinello}  is  translation resp.
gauge  
invariant  if 
\begin{equation}
\label{mat.gauge.inv}
\tau_\varsigma \circ R(U) = R(\tau_\varsigma U )\circ \tau_\varsigma \ , 
\ \forall \varsigma \in \R 
\quad
\mbox{ resp. } 
\quad
\tg_\theta \circ R(U) = R(\tg_\theta U )\circ \tg_\theta  \ , \ \forall \theta \in \T \ . 
\end{equation}
Similarly we will say that a vector field  
\be\label{rtr}
X(U): =  \vect{X(U)^+}{X(U)^-}  \quad \text{is real-to-real if} \quad
\bar{X(U)^+}=X(U)^- \, , \quad \forall  U \in  L^2_\R (\T, \C^2) \, , 
\ee
and translation resp.
gauge   invariant if 
 \begin{equation}\label{X.gauge}
 \tau_\varsigma\circ X = X \circ  \tau_\varsigma  \ , 
\quad
\forall \varsigma \in \R \ , 
\qquad
\tg_\theta \circ X = X\circ \tg_\theta , \quad \forall \theta \in \T 
 \ .
\end{equation}
If $ R(U)$ in \eqref{vinello} is translation resp. gauge  invariant, then the  vector field
$X(U):= R(U)U$ is translation resp.
gauge invariant as well.

\smallskip
\noindent{\bf Fourier expansion.}
Given a $2\pi$-periodic function $u(x)$ in $  L^2 (\T,\C)$, we expand it in Fourier series as 
\begin{equation}\label{fourierseries}
u(x)= \sum_{j \in \mathbb{Z}} u_j \,  e^{\im j x}, \quad u_j:= \frac{1}{2\pi}\int_{\mathbb{T}}u(x) e^{- \im j x }\,\di x \, .
\end{equation}
We shall expand a function $ U\in L^2(\T;\C^2)$ as 
\be \label{u+-}
U=\vect{u^+}{u^-}= \sum_{\sigma\in \pm} \sum_{j \in \Z} \tq^\sigma u_j^\sigma e^{\ii\sigma j x},
 \quad 
   u^\sigma_j:=\frac{1}{2\pi}\int_{\mathbb{T}}u^\sigma(x) e^{- \im \sigma j x }\,\di x 
\ee
 where $ \tq^+:= \vect{1}{0}, \quad \tq^-:= \vect{0}{1} $.
 
For $ \vec{\jmath} = (j_1,\dots, j_p) \in \Z^p$, $p \geq 1$, 
and $\vec{\sigma} = (\sigma_1,\dots,\sigma_p)\in \{\pm\}^p$ we  denote 
$ |\vec{\jmath}| := \max(|j_1|.\dots, |j_p| ) $ and 
\be\label{notationuvecjvecs}
u_{\vec{\jmath}}^{\vec{\sigma}}:= u_{j_1}^{\sigma_1}\cdots u_{j_p}^{\sigma_p} \, , 
\qquad \vec{\sigma} \cdot \vec{\jmath} := \sigma_1 j_1  + \dots+ \sigma_p j_p \, , 
\quad
\vec \sigma \cdot \vec 1 := \sigma_1 + \cdots + \sigma_p  \ . 
\ee 
 We also denote by $\fP_p$ the set of indexes 
 \begin{equation}
 \label{mom1}
 \fP_p:= \left\{  ( \vec{\jmath} , \vec \sigma) \in \Z^p \times \{\pm \}^p \colon \quad   \vec{\jmath} \cdot \vec    \sigma = 0   \ , \ \ \vec \sigma \cdot \vec 1 = 0  \right\} \ . 
 \end{equation}

  \noindent{\bf  Fourier representation of homogeneous operators and  vector fields.} 
In the sequel we shall encounter  matrices of linear operators, gauge and translational invariant, of the form 
\be\label{Mupm}
M(U)= \begin{pmatrix}M^+_+(U)&M^-_+(U)\\M^+_-(U)&M^-_-(U)\end{pmatrix} , 
\ee
 depending on $U$ in a homogeneous way. We shall call them $p$-homogeneous if they are polynomials in $U$ of order $p$. 
We write them in  Fourier as
\be \label{smoocara0}
\begin{aligned}
M(U)V= \vect{(M(U)V)^+}{(M(U)V)^-}, \quad (M(U)V)^\sigma= 
  \sum_{ \sigma k = \vec \sigma_{p} \cdot \vec \jmath_{p}+ \sigma' j \atop 
\sigma = \vec \sigma_p \cdot \vec 1  + \sigma' 
   } M_{\vec \jmath_{p}, j,k}^{\vec \sigma_{p}, \sigma',\sigma} u_{\vec \jmath_{p}}^{\vec \sigma_{p}} v^{\sigma'}_{j}  {e^{\ii \sigma k x}}\, ,
\end{aligned}
 \ee 
where  the coefficients $M_{\vec \jmath_{p}, j,k}^{\vec \sigma_{p}, \sigma', \sigma} \in \C$  fulfill the  
the following symmetric property:
  for any permutation $ \pi $ of $ \{1, \ldots, p \} $, it results
\begin{equation}
\label{M.coeff.p}
M_{j_{\pi(1)}, \ldots,j_{\pi(p)}, j,k}^{ \sigma_{\pi(1)}, \ldots, \sigma_{\pi(p)},\sigma',\sigma} 
=  
M_{j_{1}, \ldots, j_{p}, j,k}^{ \sigma_{1}, \ldots, \sigma_{p},\sigma',\sigma} \, . 
\end{equation}
The operator $M(U)$ is real-to-real, according to definition  \eqref{vinello}, if and only if its coefficients fulfill
\be\label{M.realtoreal}
\bar{M_{\vec \jmath_{p}, j,k}^{\vec\sigma_{p}, \sigma', \sigma}} = M_{\vec \jmath_{p}, j,k}^{-\vec \sigma_{p}, -\sigma', -\sigma} \ .
\ee
A $ (p+1)$-homogeneous vector field, which is   gauge and translation invariant (see \eqref{X.gauge}), can be expressed in Fourier as: for any $\sigma = \pm$,  
\be\label{polvect}
X(U)^\sigma
= \sum_{k\in \Z} X(U)^\sigma_k 	\, e^{\ii\sigma k x} , 
\qquad X(U)^\sigma_k= 
\!\!\!\sum_{
	k \sigma = \vec{\sigma}_{p+1} \cdot \vec{\jmath}_{p+1} \atop
	\sigma = \vec{\sigma}_{p+1} \cdot \vec 1
	}\!\!\!\!\!\!\!\!\!
X_{\vec{\jmath}_{p+1}, k}^{ \vec{\sigma}_{p+1}, \sigma} \, u^{\vec{\sigma}_{p+1}}_{\vec{\jmath}_{p+1}} \,,
\ee
the last sum being in $(\vec \jmath_{p+1}, \vec \sigma_{p+1})$, 
and with  coefficients  $ X_{\vec{\jmath}_{p+1}, k}^{ \vec{\sigma}_{p+1}, \sigma}\in \C$  satisfying 
the symmetry condition: for any permutation $ \pi $ of $ \{1, \ldots, p+1 \} $, 
\be\label{symmetric}
X_{\ j_{\pi(1)}, \ldots,j_{\pi(p+1)}, k}^{ \sigma_{\pi(1)}, \ldots, \sigma_{\pi(p+1)},\sigma} 
=  
X_{\ j_{1}, \ldots, j_{p+1}, k}^{ \sigma_{1}, \ldots, \sigma_{p+1},\sigma} \, . 
\ee
The constraint of the indexes in \eqref{polvect} can also be  written as $(\vec \jmath_{p+1}, k, \vec \sigma_{p+1}, - \sigma) \in \cP_{p+2}$ (recall \eqref{mom1}), and we shall often  use this notation.

If $X(U)$ is real-to-real, see \eqref{rtr}, then 
\be\label{X.real}
\bar{X(U)^+_k}=X(U)^-_k\quad \mbox{i.e.}  \quad\bar{X_{\vec{\jmath}_{p+1}, k}^{ \vec{ \sigma}_{p+1}, +}}=X_{\vec{\jmath}_{p+1}, k}^{- \vec{\sigma}_{p+1}, -} \, . 
\ee

 \subsection{Paradifferential calculus}
In this section  we introduce paradifferential and smoothing operators, following \cite{BD, BMM}.

\smallskip

 \noindent{\bf Symbols.}
We define  the class of symbols which we will use along the paper. 
They correspond to the autonomous symbols of Definition 3.3 in \cite{BD}, where the dependence on time enters only through the function $U=U(t)$. In view of this, we do not need to keep track on the regularity indexes in time and we fix $K = K' =0$ with respect to Definition 3.3 of \cite{BD}.
\begin{definition}[Symbols]\label{def:sfr}
Let $ m  \in \R $, $N \in \N_0$, $p \in \N$, $s_0, r>0$.
\begin{enumerate}
\item{\bf H\"older symbols.} We denote by $\Gamma^m_{W^{N,\infty}}$ the space of 
functions $ a : \T\times \R\to \C $,  $a(x, \xi)$, 
which are $C^\infty$ with respect to $\xi$ and such that, for any  $ \beta \in \N_0 $, 
there exists a constant $C_\beta >0$ such that
\begin{equation}\label{simb-pro}
\big\| \pa_\xi^\beta \, a(\cdot, \xi) \big\|_{W^{N,\infty}} \leq C_\beta \, \la \xi \ra^{m - |\beta|}  , \quad \forall \xi \in \R \,  . 
\end{equation}
We endow $\Gamma^m_{W^{N,\infty}}$ with the  family of norms defined, for any $n \in \N_0$, by
\begin{equation}
\label{seminorm}
\abs{a}_{m, {W^{N,\infty}}, n}:=  \max_{\beta\in \{0, \dots , n\} }\, 
\sup_{\xi \in \R} 
\, \big\| \la \xi \ra^{-m+|\beta|} \, \pa_\xi^{\beta} a(\cdot,  \xi) \big\|_{W^{N,\infty}} \ . 
\end{equation} 

{   \item {\bf $p$-Homogeneous symbols.} We denote by $\widetilde{\Gamma}^m_p$ the space of  $p$-linear symmetric maps from $\left( C^{\infty}\left(\mathbb{T};\mathbb{\C}^2\right)\right)^p$ to
        $ C^\infty(\mathbb{T}\times \R; \C) $ ,
		$ (x, \xi) \mapsto a_p(U_1, \dots, U_p ;x,\xi)$ defined by
		  \be\label{espr.hom.sym}
a_p(U_1, \dots, U_p; x, \xi):=  \sum_{\substack{\vec{\jmath}\in \Z^p\\ \vec{\sigma}\in \{ \pm\}^p}} 
a_{\vec \jmath}^{\vec\sigma}(\xi)
(u_1)_{j_1}^{\sigma_1} \cdots (u_p)_{j_p}^{\sigma_p}\, e^{\ii (\vec{\sigma}\cdot \vec{\jmath}) x},
        \ee
		where $a_{\vec{\jmath}}^{\vec{\sigma}}(\xi):= a_{j_1, \dots, j_p}^{\sigma_1, \dots, \sigma_p}(\xi) $ are complex valued Fourier multipliers,  satisfying
        $$
       a_{j_1, \dots, j_p}^{\sigma_1, \dots, \sigma_p}(\xi) = a_{j_{\pi(1)}, \dots, j_{\pi(p)}}^{\sigma_{\pi(1)}, \dots, \sigma_{\pi(p)}} (\xi)\quad \text{ for any } \pi \text{ permutation of } \{1, \dots, p\}\,,
        $$
        and
        for some $	\mu\geq 0$, 
		\be \label{homosymbo}
		| \partial_\xi^\beta a_{\vec{\jmath}}^{\vec{\sigma}}(\xi) | \leq C_\beta \la \vec{\jmath}\ra^{\mu}  \langle \xi\rangle^{m-\beta},  \quad \forall \vec{\jmath}\in \Z^p,\, \vec{\sigma}\in\{\pm\}^p, \, \beta\in \N_0.
		\ee
        We shall denote by 
        $$
        a_p(U;x, \xi):= a_p(U, \cdots, U; x, \xi)
        $$
        the polynomial symbol associated to the multilinear symmetric symbol. \\
		We  denote by $\widetilde{\Gamma}^m_0 $ the space of constant coefficients symbols $ \xi \mapsto a(\xi)$ which satisfy \eqref{homosymbo} with $\mu= 0 $.}   
		
		\item  {\bf Non-homogeneous symbols. }   We denote by $\Gamma_{\geq p}^m[r]$ the space of functions  $ (U;x,\xi)\mapsto a(U;x,\xi) $,
		defined for $ U \in B_{s_0}(r)$ for some $s_0$ large enough, with complex values, such that  for any $\ts\geq s_0$,  there are $C>0$, $r':=r'(\ts)\in(0,r)$ and for any $ U \in B_{s_0}\left(r' \right)\cap H^{\ts}\left(\mathbb{T};\mathbb{C}^2\right)$,  any $\beta \in \N_0$ and  $N \leq \ts - s_0$, one has the estimate
		\begin{equation}\label{nonhomosymbo}
			\left\|\partial_\xi^\beta a\left({ U;\cdot,\xi }\right)\right\|_{W^{N,\infty}}  \leq C \langle \xi \rangle^{m-\beta} \| U\|_{{s_0}}^{p-1}\|U\|_{\ts} \, .
		\end{equation}
		  In addition we require also  the translation invariance property
\begin{equation} \label{mome}
a\left( \tau_{\varsigma} U; x,\xi\right)= a\left( U; x+\varsigma, \xi\right),\quad \forall 
\varsigma\in \R \, , 
\end{equation}
where $\tau_\varsigma$ is the translation operator in \eqref{tra}.
\item {\bf Symbols.} We denote by $\Sigma\Gamma^m_0[r]$ the class of symbols of the form
\begin{equation}
\label{symbols}
a(U; x, \xi) = a_0(\xi)  + a_2(U; x, \xi)  + a_{\geq 4}(U;x, \xi)
\end{equation}
where $ a_0(\xi)\in \wt\Gamma^m_0$ is a Fourier multiplier, $a_2(U) \in \wt \Gamma_2^m$ and  $ a_{\geq 4}(U) \in \Gamma^m_{\geq 4}[r]$.
We denote by $\Sigma\Gamma^m_2[r]$ the class of symbols of the form \eqref{symbols} with $a_0(\xi) = 0$. Finally sometimes we shall write  $\Sigma\Gamma^m_4[r]\equiv \Gamma_{\geq 4}^m[r]$.
	\end{enumerate}
	We say that a symbol  $a(U;x,\xi) $ is \emph{real} if it is real valued for any
		$ U \in B_{s_0,\R}(I;r)$.
		
		We also denote by $\wt \cF_p$ (respectively $\cF_{\geq p}[r]$)  the subspace of $\wt \Gamma^0_p$ (respectively $\Gamma^0_{\geq p}[r]$)   made of those
symbols which are independent of $\xi$, and by $\wt \cF^\R_p$ (respectively $\cF^\R_{\geq p}[r]$)  to denote functions in $\wt \cF_p$ (respectively $\cF^\R_{\geq p}[r]$)  which are real valued.
\end{definition}

\begin{remark}
{
Sometimes we shall write a symbol $a_p(U; x, \xi)$ only in polynomial form
\be \label{not-symm}
		a_p(U;x,\xi):=\sum_{\substack{\vec{\jmath}\in \Z^p\\ \vec{\sigma}\in \{ \pm\}^p}} \wt a_{\vec{\jmath}}^{\vec{\sigma}}(\xi) \, u_{\vec{\jmath}}^{\vec{\sigma}} \, e^{\ii (\vec{\sigma}\cdot \vec{\jmath}) x}
		\ee
        with some Fourier multiplier coefficients $\wt a_{\vec{\jmath}}^{\vec{\sigma}}(\xi) $ not necessarily symmetric, but fulfilling the estimates
\eqref{homosymbo}. 
One obtains the symmetric coefficients $a_{j_1, \dots, j_p}^{\sigma_1, \dots, \sigma_p} $ in the expression \eqref{espr.hom.sym} by symmetrizing, i.e., denoting by $\cS_p$ the symmetric group of permutations of $\{1, \ldots, p\}$, 
$$
a_{j_1, \dots, j_p}^{\sigma_1, \dots, \sigma_p} = 
\frac{1}{p!}\sum_{\pi \in \cS_p} \wt a_{j_{\pi(1)}, \dots, j_{\pi(p)}}^{\sigma_{\pi(1)}, \dots, \sigma_{\pi(p)}}  \ . 
$$
We shall use the  notation \eqref{not-symm} for example  in formulas \eqref{vuoneunder} and  for the resonant  transport term in \eqref{VresZ}; the reason is that the transport term \eqref{VresZ} is perhaps the most important object, being the term responsible for the growth, and we prefer to express it in the simplest possible form.
}
\end{remark}

\noindent
$ \bullet $  If $a$ is a symbol in $  \Gamma^m_{W^{N,\infty}} $ 
then $ \partial_x a  \in  \Gamma^{m}_{W^{N-1,\infty}}$ and  
$ \partial_\xi a \in   \Gamma^{m-1}_{W^{N,\infty}}$.
If  $b $ is a symbol in $  \Gamma^{m'}_{W^{N,\infty}}$  then 
$a b \in  \Gamma^{m+m'}_{W^{N,\infty}}$.
If $a \in \Gamma^{m}_{\geq p}[r]$ and $b \in \Gamma^{m'}_{\geq q}[r]$, then  $ab  \in  \Gamma^{m+m'}_{\geq p+q}[r]$ .

\noindent
$\bullet$  $p$-homogeneous symbols in $ \widetilde \Gamma^m_p$ and non-homogeneous symbols in $\Gamma^m_{\geq p}[r]$  are actually functions with values in $\Gamma^m_{W^{N,\infty}}$ for some $N \in \N$, whose seminorms \eqref{seminorm} are bounded by
\begin{equation}
|a_p|_{m, W^{N,\infty}, n} \leq C_n \, \| U\|_{{1}}^{p-1}\|U\|_{{N+\mu+1}}  \ , 
\quad
|a|_{m, W^{N,\infty}, n} \leq C_n \, \norm{U}_{{s_0}}^{p-1} \norm{U}_{\ts} \ , \ \ N \leq \ts - s_0  \ .
\end{equation}

\noindent
$\bullet$ A $p$-homogeneous symbol $a_p(U,x,\xi)$ is a non-homogeneous symbol, since \eqref{espr.hom.sym}--\eqref{homosymbo} imply
\begin{equation}\label{nonhomosymbo.homo}
			\left\|\partial_\xi^\beta a_p\left({ U;\cdot,\xi }\right)\right\|_{W^{N,\infty}}  \leq C \langle \xi \rangle^{m-\beta} \| U\|_{{1}}^{p-1}\|U\|_{{N+\mu+1}} \,  \ ,
\end{equation}
and \eqref{espr.hom.sym} implies the translation invariance property \eqref{mome}.

\smallskip

\noindent{\bf Paradifferential quantization.}
Given $p\in \N_0$ we consider   functions
  $\chi_{p}\in C^{\infty}(\R^{p}\times \R;\R)$ and $\chi\in C^{\infty}(\R\times\R;\R)$, 
  even with respect to each of their arguments, satisfying, for $0<\delta_0\leq \tfrac{1}{10}$,
\begin{align*}
&{\rm{supp}}\, \chi_{p} \subset\{(\xi',\xi)\in\R^{p}\times\R; |\xi'|\leq\delta_0 \langle\xi\rangle\} \, ,\qquad \chi_p (\xi',\xi)\equiv 1\,\,\, \rm{ for } \,\,\, |\xi'|\leq \delta_0 \langle\xi\rangle / 2 \, ,
\\
&\rm{supp}\, \chi \subset\{(\xi',\xi)\in\R\times\R; |\xi'|\leq\delta_0 \langle\xi\rangle\} \, ,\qquad \quad
 \chi(\xi',\xi) \equiv 1\,\,\, \rm{ for } \,\,\, |\xi'|\leq \delta_0   \langle\xi\rangle / 2 \, . 
\end{align*}
For $p=0$ we set $\chi_0\equiv1$. 
We assume moreover that 
$$ 
|\partial_{\xi}^{\ell}\partial_{\xi'}^{\beta}\chi_p(\xi',\xi)|\leq C_{\ell,\beta}\langle\xi\rangle^{-\ell-|\beta|} \, , \  \forall \ell\in \N_0, \,\beta\in\N_0^{p} \, ,  
\ \ 
|\partial_{\xi}^{\ell}\partial_{\xi'}^{\beta}\chi(\xi',\xi)|\leq C_{\ell,\beta}\langle\xi\rangle^{-\ell-\beta}, \  \forall \ell, \,\beta\in\N_0 \, .
$$ 
If $ a (x, \xi) $ is a smooth symbol 
we define its Weyl quantization  as the operator
acting on a
$ 2 \pi $-periodic function
$u(x)$ (written as in \eqref{fourierseries})
 as
$$
{\rm Op}^{W}(a)u= \sum_{k\in \Z}
\Big(\sum_{j\in\Z}\hat{a}\big(k-j, \tfrac{k+j}{2}\big) u_j \Big){e^{\im k x}}
$$
where $\hat{a}(k,\xi)$ is the $k^{th}-$Fourier coefficient of the $2\pi-$periodic function $x\mapsto a(x,\xi)$.

\begin{definition}{\bf (Bony-Weyl quantization)}\label{quantizationtotale}
If $a(U;x,\xi)$ is a symbol in $\widetilde{\Gamma}^{m}_{p}$, 
respectively in $\Gamma^m_{W^{N,\infty}}$ or $\Gamma^{m}_{\geq p}[r]$,
we set
\be\label{regula12}
a_{\chi_{p}}(U;x,\xi) := \!\!\!\!\! \sum_{\substack{\vec{\jmath}\in \Z^p\\ \vec{\sigma}\in \{ \pm\}^p}}\chi_p (\vec{\jmath},\xi) a_{\vec{\jmath}}^{\vec{\sigma}}(\xi) u_{\vec{\jmath}}^{\vec{\sigma}} e^{\ii (\vec{\sigma}\cdot \vec{\jmath}) x},
\quad 
 a_{\chi}(U;x,\xi) :=\sum_{j\in \Z} 
\chi (j,\xi )\hat{a}(U;j,\xi)e^{\im j x}  \, 
\ee
where in the last equality $  \hat a(U; j,\xi) $ stands for $j^{th}$ Fourier coefficient of $a(U;x, \xi)$ with respect to the $ x $ variable, and 
we define the \emph{Bony-Weyl} quantization of $ a(U; \cdot)  $ as 
\begin{align}\label{BW}
&\opbw(a({U};\cdot))v= {\rm Op}^{W} (a_{\chi_{p}}({U};\cdot))v= \!\!\!\!\!\!\!\!\!\sum_{\substack{(\vec{\jmath},j,k)\in \Z^{p+2}\\ \vec{\sigma}\in \{ \pm\}^p\\ \vec{\sigma}\cdot \vec{\jmath}+j=k}}\chi_p 
\left(\vec{\jmath},\frac{j+k}{2}\right)
 a_{\vec{\jmath}}^{\vec{\sigma}}\left(\frac{j+k}{2}\right) 
 u_{\vec{\jmath}}^{\vec{\sigma}} v_j
{e^{\im k x}} \, ,\\
&\opbw(a(U;\cdot))v= {\rm Op}^{W} (a_{\chi}(U;\cdot))v=
\!\!\!\!\!\! \sum_{(j,k)\in \Z^2} \!\!\!  \chi \left(k-j,\frac{j+k}{2} \right)
\hat{a}\left(U; k-j, \frac{k+j}{2}\right)v_j {e^{\im k x}} \, .\label{BWnon}
\end{align}
\end{definition}
Note that if 
$ \chi \Big( k-j, \tfrac{ k + j }{2}\Big) \neq 0 $
then $ |k-j| \leq \delta_0 \langle \frac{j + k}{2} \rangle  $
and therefore, for $ \delta_0 \in (0,1)$, 
\begin{equation}\label{rela:para}
\frac{1-\delta_0}{1+\delta_0} |k| \leq 
|j| \leq \frac{1+\delta_0}{1-\delta_0}|k| \, , \quad \forall j, k \in \Z\, . 
\end{equation}
This relation shows that the action of a paradifferential  operator does not spread much the Fourier support of functions.

\noindent
$ \bullet $  If $ a$ is a homogeneous  symbol, the two definitions  of quantization in \eqref{BW} and \eqref{BWnon} differ by a  smoothing operator according to 
Definition \ref{omosmoothing} below. 

\noindent
$\bullet$
Definition \ref{quantizationtotale} 
is  independent of the cut-off functions $\chi_{p}$, $\chi$,  
up to smoothing operators that we define below (see Definition \ref{omosmoothing}), see the remark at page 50  of  \cite{BD}. 

\noindent
$\bullet$
Given  a paradifferential  operator
$ A = \Opbw{a(x,\xi)} $ it results
\be\label{A1b}
\bar A = \Opbw{\bar{a(x, - \xi)}} \, , \quad 
A^\top = \Opbw{a(x, - \xi)} \, , \quad
A^*= \Opbw{\bar{a(x,  \xi)}} \, , 
\ee
where $ A^\top $  and $ A^* $ denote respectively the transposed and  adjoint operator with respect to the complex, respectively real,  scalar product 
of $  L^2(\T, \C) $ in \eqref{scpr12hom}. It results $ A^* = \bar A^\top $. 

\noindent
$\bullet$
 A paradifferential operator $A= \Opbw{a(x,\xi)}$ is {\it real} (i.e. $A = \bar A$) if 
\be \label{realetoreale}
 \bar{a(x,\xi)}= a^\vee(x,\xi) \quad \text{ where} \quad a^{\vee}(x,\xi) := a(x,-\xi) \, .
 \ee
\noindent
$ \bullet $
A matrix of paradifferential operators $ \Opbw{A( x,\x)}$ is real-to-real, i.e. \eqref{vinello} holds, if and only if 
the  matrix of symbols $A(x,\x)$ has the form 
\begin{equation}\label{prodotto}
{\footnotesize A(x,\x) =
\left(\begin{matrix} {a}(x,\x) & {b}(x,\x)\\
{\ov{b^\vee(x,\x)}} & {\ov{a^\vee(x,\x)}}
\end{matrix} 
\right)=\left(\begin{matrix} {a}(x,\x) & 0\\
0 & {\ov{a^\vee(x,\x)}}
\end{matrix} 
\right)+\left(\begin{matrix} 0& {b}(x,\x)\\
{\ov{b^\vee( x,\x)}} & 0
\end{matrix} 
\right) \,   \, . }
\end{equation}

\noindent
$ \bullet $
A  real-to-real matrix of $U$-dependent paradifferential operators $ \Opbw{A(U; x,\x)}$ is gauge invariant, i.e. \eqref{mat.gauge.inv} holds, if and only if 
the symbols in  \eqref{prodotto} fulfill, with $\tg_\theta$ in \eqref{tra},
\begin{equation}\label{sym.gauge}
a(U; x, \xi) = a(\tg_\theta U; x, \xi) \ , 
\quad
e^{\im 2 \theta } \, b(U; x, \xi) = b(\tg_\theta U; x, \xi)  \ , \quad \forall \theta \in \T \ , 
\end{equation}
If, in addition,  $a, b\in \wt \Gamma^m_p$, then
$\Opbw{a}$ in  \eqref{BW} have indexes restricted to $\vec \sigma \cdot 1=0$, whereas $\Opbw{b}$ to $\vec\sigma \cdot 1 = 2$.

\smallskip
We will use also the notations
\begin{equation}\label{vecop}
\begin{aligned} 
&\vOpbw{a(x,\xi)}:= \Opbw{\footnotesize \begingroup\setlength\arraycolsep{1 pt}\begin{bmatrix} a(x,\xi)&0\\0& \bar{ a^{\vee}(x,\xi)}\end{bmatrix}\endgroup} \, , \ 
\zOpbw{b(x,\xi)}:= \Opbw{\footnotesize\begingroup\setlength\arraycolsep{1 pt}\begin{bmatrix} 0&b(x,\xi)\\ \bar{ b^{\vee}(x,\xi)}&0\end{bmatrix}\endgroup}
\end{aligned}
\end{equation}

Along the paper 
we shall use the following  results concerning the action of a paradifferential operator 
in Sobolev spaces. 
 We refer to 
\cite[Theorem A.7]{BMM} for the proof of $(i)$ and to  
 \cite[Proposition 3.8]{BD}   for the proof of $(ii)$, $(iii)$.

\smallskip

\begin{theorem}{\bf (Continuity of Bony-Weyl operators)}
\label{thm:contS}
Let $m\in \R$, $p\in \N$, $r>0$. Then:

\smallskip
$(i)$ Let  $ a \in \Gamma^m_{L^\infty} $.
Then $\Opbw{a}$ extends to a bounded operator 
$ H^{s} \to  H^{s-m}$ for any $ s \in \R $  satisfying the estimate, for any $ u \in H^s $,  
\begin{align} \label{cont00}
& \norm{\Opbw{a}u}_{{s-m}} \lesssim \, \abs{a}_{m, L^\infty, 4} \, \norm{u}_{{s}} \ .
\end{align}

\smallskip
$(ii)$ Let  $a\in \widetilde{\Gamma}_{p}^{m}$. 
There is $ s_0 > 0 $ such that for any $s \in \R$,  
there is a constant $C>0$, depending only on $s$ and on \eqref{homosymbo} with $\ell=\beta=0$,
such that for any $U_1,\ldots,U_{p} \in H^{s_0}(\T, \C^2)$ and $v \in H^s(\T, \C)$, one has 
\begin{equation}\label{stimapar}
\|\Opbw{a(U_1,\dots, U_p;\cdot)}v\|_{{s-m}}\leq C\prod_{j=1}^{p}\|U_{j}\|_{{s_0}}\|v\|_{{s}} \, ,
\end{equation}
for $p\geq 1$, while for $p=0$ the \eqref{stimapar} holds by replacing the right hand side with $C\|v\|_{{s}}$.

\smallskip
$(iii)$ Let $a\in \Gamma^{m}_{\geq p}[r]$.
There is $ s_0 > 0 $ such that  for any $s \in \R$
there is a constant $C>0$  such that
for any $U \in B_{s_0}(r)$  one has 
\be\label{stimapar2}
\|\Opbw{a(U;\cdot)}\|_{\mathcal{L}({H}^{s},{H}^{s-m})}\leq C\|U\|_{{s_0}}^{p} \, .
\ee
\end{theorem}

\smallskip 

\noindent{\bf Classes of $m$-operators and smoothing operators.} 
We introduce $m$-operators and smoothing operators.  This is  a small adaptation of  \cite{BD, BMM} where we consider only autonomous maps, where again the time dependence is only  through $U(t)$. 
In particular we put $K,K' = 0$ with respect to the notation in \cite{BD,BMM}.
 Given integers $(n_1,\ldots,n_{p+1})\in \N^{p+1}$, we denote by $\max_{2}(n_1 ,\ldots, n_{p+1})$ 
the second largest among  $ n_1,\ldots, n_{p+1}$.

\begin{definition}[Classes of $m$-operators]
\label{def.m-op}
 Let $m \in \R$, $p \in \N_0$ and $r >0$.
\begin{enumerate}
\item  {\bf $p$-homogeneous $m$-operators}. We denote by
$\wt \cM^{m}_p$ the class of 
 $(p+1)$-linear operators
 from $(C^{\infty}(\T;\C^{2}))^{p}\times C^{\infty}(\T;\C)$ to 
 $C^{\infty}(\T;\C)$ of the form
$ (U_{1},\ldots,U_{p}, v)\to M_p(U_1,\ldots, U_p)v$,   symmetric
 in $(U_{1},\ldots,U_{p})$,
with Fourier expansion
\begin{equation}
 \label{Mp}
M_p(U)v:=
M_p(U, \ldots, U)v= \!\!\!\!\!\!\!\!
 \sum_{\substack{  \vec\sigma_{p}\in \{\pm\}^{p} \\  k -j = \vec \sigma_{p} \cdot \vec \jmath_{p}} }
\!\!\!\!\!\!
  M_{\vec \jmath_{p}, j,k}^{\vec \sigma_{p}} \, u_{\vec \jmath_{p}}^{\vec \sigma_{p}}\,  v_{j}  \, {e^{\ii k x} }
 \end{equation}  
that satisfy the following. There are $\mu\geq0$, $C>0$ such that 
 for any $(\vec \jmath_{p},j,k) \in \Z^{p+2} $,   $ \vec \sigma_p \in \{ \pm \}^{p} $,  one has 
\be\label{smoocara}
 |M_{\vec \jmath_p, j, k}^{\vec \sigma_p} |\leq C \, 
{\rm max}_2\{ \la j_1\ra,\dots, \la j_p\ra,\la j\ra \}^{\mu}\, \max\{ \la j_1\ra,\dots, \la j_p\ra,\la j\ra \}^{m}   \, . 
 \ee

\item  {\bf Non-homogeneous $m$-operators.} We denote by $\cM^{m}_{\geq p}[r]$ the class of   
 operators  $(U, v)\to M(U) v $ defined on $B_{s_0}(r)\times{H}^{s_0}(\T,\C)$ for some $ s_0 >0  $, 
  which are  linear in the variable $v $ and such that the following holds true. 
  For any $s\geq s_0$ there are $C>0$ and 
  $r'=r'(s)\in]0,r[$ such that for any 
  $U\in B_{s_0}(r')\cap {H}^{s}(\T,\C^2)$, 
  any $ v \in  {H}^{s}(\T,\C)$, we have that
{
\begin{equation}
\label{piove}
\begin{aligned} 
& \|{ M(U)v }\|_{{s - m}} 
\leq C \left( \|{v}\|_{s}\|{U}\|_{s_0}^{p} 
 +\|{v}\|_{{s_0}}\|U\|_{{s_0}}^{p-1}\|{U}\|_{{s}} \right) \quad && \text{if} \quad p \geq 1\, ,\\
 & \|{ M(U)v }\|_{{s - m}} 
\leq C \left( \|{v}\|_{s}
 +\|{v}\|_{{s_0}} \|{U}\|_{{s}}\| U\|_{s_0} \right) \quad &&\text{if} \quad p =0\,.
 \end{aligned}
\end{equation}}

 In addition we require the translation invariance property
\begin{equation}\label{def:R-trin}
M( \tau_\varsigma  U) [\tau_\varsigma v]  =  
\tau_\varsigma \big( M(U)v \big) \, , \quad \forall \varsigma \in \R \, . 
\end{equation}
where $\tau_\varsigma$ is the translation operator in \eqref{tra}.

\item {\bf $m$-Operators.} 
We denote by $\Sigma\cM^m_0[r]$ the space of operators {$(U,v) \to M(U)v$} of the form 
\begin{equation}\label{sum.maps}
M(U) = M_0 + M_2(U)  + M_{\geq 4}(U) \ 
\end{equation}
where  $M_p(U)$ in $\wt \cM_p^m$, $p\in\{0,2\}$, and  $M_{\geq 4}(U)$ in $\cM^m_{\geq 4}[r]$.\\
We  denote  by $\Sigma\cM^m_2[r]$ the operators of the form \eqref{sum.maps} with $M_0 = 0$. Finally  sometimes we shall write $\Sigma\cM^m_4[r]\equiv \cM^m_{\geq 4}[r]$. 
\end{enumerate}
\end{definition}

\noindent
$\bullet$ A $p$-homogeneous $m$-operator $M_p$ is a non-homogeneous $m$-operator. Indeed \eqref{smoocara} implies the  quantitative estimate: 
for  $ s_0 \geq  \mu+1>0$, for any $s \geq s_0 $, 
any $U \in H^s(\T;\C^2)$,  any $v \in   H^s(\T;\C)$
	\be \label{smoothing}
	\| M_p(U)v \|_{{s-m}}\lesssim_s \| U \|_{{s_0}}^{p} 
	\| V \|_{ s}+ \| U \|_{{s_0}}^{p-1}\| U \|_{s} \|V \|_{{s_0}} 
	\ee
	which is \eqref{piove} (see Lemma 2.8  and 2.9 in \cite{BMM} for a proof).
Moreover \eqref{def:R-trin} follows from the Fourier restriction $ k -j = \vec \sigma_{p} \cdot \vec \jmath_{p}$ in \eqref{Mp}.\\
$\bullet$ 
{\bf (Paradifferential operators as $m$-operators)}
If $a(U;x,\xi)$ is a symbol in $ \Sigma \Gamma^{m}_{0}[r]$  
then the paradifferential operator  $ \opbw(a(U;x,\xi))$ is an $m$-operator
$ \Sigma  \mM_{0}^m[r]$. This is a consequence of Theorem \ref{thm:contS}--$(ii)$\&$(iii)$. \\
$\bullet$ We will meet vector fields of the form $X(U) = M(U)U$ where $M(U)$ is a matrix of $p$-homogeneous $m$-operators as in \eqref{Mupm}.
In this case the relation between the Fourier coefficients  of the vector field in \eqref{polvect} and those of the $m$-operator in  \eqref{Mp} is given by 
\begin{equation}
	\label{simmetrizzata}
	X_{\ j_1, \ldots, j_{p}, j_{p+1}, k}^{\sigma_1, \ldots, \sigma_{p}, \sigma_{p+1},  \sigma} = 
	\frac{1}{p+1} \left(
	M_{\ j_1, \ldots, j_{p}, j_{p+1}, k}^{\sigma_1, \ldots, \sigma_{p}, \sigma_{p+1},  \sigma} + 
	M_{\ j_{p+1}, \ldots, j_{p}, j_1, k}^{\sigma_{p+1}, \ldots, \sigma_{p}, \sigma_1,  \sigma} +
	\cdots
	+
	M_{\ j_1, \ldots, j_{p+1}, j_{p}, k}^{\sigma_1, \ldots, \sigma_{p+1}, \sigma_{p},  \sigma} 
	\right) \, , 
\end{equation} 
namely they are obtained symmetrizing with respect to the second last index $(j,\sigma')$  the coefficients  $M_{\ \vec \jmath_{p}, j, k}^{\vec \sigma_{p}, \sigma',\sigma}$ of $M(U)$.

\smallskip

If $m  \leq 0 $ the  $m$-operators  are referred to as smoothing operators.
 \begin{definition}{\bf (Smoothing operators)} \label{omosmoothing}
Let $ \vr\geq0$,   $p\in \N_0$ and $q \in \{0,2\}$.
We define the  $\varrho$-smoothing operators 
\be\label{smoothingoper}
\begin{aligned}
& \widetilde{\mathcal{R}}^{-\vr}_{p}:= \widetilde{\mathcal{M}}^{-\vr}_{p} \, ,
\quad
   \mathcal{R}^{-\vr}_{\geq p}[r]:= \mathcal{M}^{-\vr}_{ \geq p}[r]  \ , 
   \quad
 \Sigma\mathcal{R}^{-\vr}_q[r]:= \Sigma\mathcal{M}^{-\vr}_q[r]  \ .
 \end{aligned}
\ee
\end{definition}

\noindent
$ \bullet $ In view of \eqref{smoocara} a homogeneous $ m $-operator in $ \widetilde{\mathcal{M}}^{m}_{p} $ 
with the property that, on its support,  
$ {\rm max}_2\{ \la j_1\ra,\dots, \la j_p\ra,\la j\ra \}\sim  {\rm max}\{ \la j_1\ra,\dots, \la j_p\ra,\la j\ra \} $  is actually a 
{\it smoothing} operator in $ \widetilde{\mathcal{R}}^{-\vr}_{p} $ for {\it any} $ \varrho \geq 0 $ satisfying \eqref{smoocara} with $\mu \leadsto \mu+m+\varrho$ and $m \leadsto -\varrho$. 

\noindent
$ \bullet $ The Definition \ref{omosmoothing} 
of smoothing operators is modeled to gather remainders which 
satisfy  either the property $ \max_2 (n_1, \ldots, n_{p+1}) \sim  \max (n_1, \ldots, n_{p+1}) $
or arise as remainders of  compositions of paradifferential operators, 
see Proposition \ref{teoremadicomposizione} below, and thus have a fixed order $ \varrho $
of regularization.  
 
\smallskip

 \noindent {\bf Composition theorems.}
Let 
$D_{x}:=\frac{1}{\ii}\pa_{x}$. 
The following is   Definition 3.11 in \cite{BD}.

\begin{definition}{\bf (Asymptotic expansion of composition symbol)}
\label{def:as.ex}
Let   $ \varrho  \geq 0 $, $m,m'\in \R$, $r>0$.
Consider symbols $a \in \Sigma\Gamma^{m}_{p}[r]$   and $b\in  \Sigma\Gamma^{m'}_{p'}[r]$, $p,p' \in \{0,2\}$. For $U$ in $B_{\ts}(I;r)$ 
we define, for $\varrho< \ts- s_0$, the symbol
\begin{equation}\label{espansione2}
(a\#_{\varrho} b)(U;x,\x):=  \sum_{k=0}^{\varrho}  \frac{1}{2^k} \sum_{\ell+\beta=k}
\frac{(-1)^{\beta}}{\ell! \beta!} 
(\pa_\xi^\ell D_x^\beta a ) \cdot
(\pa_\xi^\beta D_x^\ell b )(U;x,\xi) \ .
\end{equation}
\end{definition}

\noindent
$ \bullet $
The symbol $ a\#_{\varrho} b $ belongs  to $\Sigma\Gamma^{m+m'}_{p+p'}[r]$.

\noindent
$ \bullet $
We have  that 
$ a\#_{\varrho}b =ab+\frac{1}{2 \ii }\{a,b\} $ 
up to a symbol in $\Sigma\Gamma^{m+m'-2}_{p+p'}[r]$,     
where 
\be\label{poisson}
\{a,b\}  :=  \pa_{\xi}a\pa_{x}b -\pa_{x}a\pa_{\xi}b  \in \Sigma \Gamma^{m+m'-1}_{p+p'}[r]
\ee
denotes the Poisson bracket. Moreover if $a\in \Gamma^m_{W^{N,\infty}}$ and  $b\in \Gamma^{m'}_{W^{N,\infty}}$ then $\{a,b\}\in \Gamma^{m+m'-1}_{W^{N-1,\infty}} $ with estimate
\be \label{poi.est}
| \{a,b\} |_{m+m'-1, W^{N-1,\infty}, n} \lesssim | a |_{m, W^{N,\infty}, n+1}| b |_{m', W^{N,\infty}, n+1}.
\ee

\noindent
$\bullet$  Due to  \eqref{symbols}, the symbol $a\#_\varrho b$ does not contain symbols of odd homogeneity.

\noindent
$ \bullet $
$ \overline{a^\vee} \#_{\varrho} 
\overline{b^\vee}  = \overline{a \#_{\varrho} b}^\vee  $ where $ a^\vee $ is defined in \eqref{realetoreale}.

 \smallskip

 The following proposition is proved in  \cite[Theorem A.8]{BMM} and  \cite[Proposition 3.12]{BD}.

\begin{proposition}{\bf (Composition of Bony-Weyl operators)} \label{teoremadicomposizione}
Let $m, m' \in \R$, $p,p' \in \{0,2\}$, $\varrho \geq 0$ and  $r >0$.

\smallskip
\noindent
$(i)$ 
Let $a \in \Gamma^m_{W^{\varrho, \infty}}$, $b \in \Gamma^{m'}_{W^{\varrho, \infty}}$. Then  
\begin{align}
\label{comp01A}
\Opbw{a}\Opbw{b} 
& =  \Opbw{a\#_\varrho b} + R(a,b)
\end{align}
where the linear operator $R(a,b)\colon {H}^s \to {H}^{s-(m+m')+\varrho}$, 
$\forall s \in \R$,  satisfies, for some $N=N(\varrho) >0$,
\begin{align}
\label{comp020}
\norm{R(a,b)u}_{{s -(m+m') +\varrho}} \lesssim  \left(\abs{a}_{m, W^{\varrho, \infty}, N} \, \abs{b}_{m', L^\infty, N} + \abs{a}_{m, L^\infty, N} \, \abs{b}_{m', W^{\varrho, \infty}, N}  \right) \norm{u}_{s} \ .
\end{align}
One can take $N(2) = 7$.\\
\smallskip
\noindent
$(ii)$ Let
$a\in  \Sigma{\Gamma}^{m}_p[r] $,  $b\in  \Sigma{\Gamma}^{m'}_{p'}[r]$. Then 
\begin{equation}
 \Opbw{a(U;x,\x)}\circ\Opbw {b(U;x,\xi)} = \Opbw{(a\#_{\varrho} b)(U;x,\x)}
 + R(U) 
\end{equation}
where $R(U) $
are smoothing operators in $ \Sigma {\mathcal{R}}^{-\varrho+m+m'}_{p+p'}[r]$.
\end{proposition}

\smallskip

\noindent 
$\bullet$ Let $a(U) \in \Sigma\Gamma^m_{p}[r]$ and $b(U) \in \Sigma \Gamma^{m'}_{p'}[r]$,  with the notation in \eqref{vecop}, one has 
  \be\label{commurule}
 \begin{aligned}
 	&\left[\zOpbw{b},\vOpbw{a}\right]= \zOpbw{b\#_\vr \overline{a^\vee}- a\#_\vr b}+ R(U)\\
 	&\left[\zOpbw{a},\zOpbw{b}\right]= \vOpbw{a\#_\vr \overline{b^\vee}- b\#_\vr \overline{a^\vee}}+ R(U)\\
 	&\left[\vOpbw{a},\vOpbw{b}\right]= \vOpbw{a\#_\vr {b}- b\#_\vr {a}}+ R(U)
 \end{aligned}
 \ee
where  $R(U)$ are real-to-real matrices of  smoothing operators in $\Sigma \cR_{p+p'}^{-\varrho + m'+m}[r]$.

\smallskip
\noindent We conclude this section with 
the  paralinearization of the product (see  \cite[Lemma 7.2]{BD}).
\begin{lemma}{\bf (Bony paraproduct decomposition)}
\label{bony}
Let  $f,g,h$ be functions in $H^\s(\T;\C)$   with  
$\s >\frac12$. Then 
\begin{equation}\label{bonyeq}
fgh = \Opbw{f  g }h + \Opbw{f  h }g+\Opbw{g  h }f + 
R_1(f, g )h+  R_2(f ,h )g+R_3(g, h )f
\end{equation}
where for $j=1, 2, 3 $, $R_j$ is a homogeneous smoothing operator in $ \widetilde \mR^{-\vr}_{1}$ for any $ \vr \geq 0$. 
\end{lemma} 
 
 \smallskip
 
\noindent {\bf Composition of $m$-operators.}  The following lemma, which is a consequence of Proposition 2.15 (items $(ii)$ and $(iv)$) in \cite{BMM}, shall be used below.
 \begin{lemma}\label{nuovetto}
 	Let $m,m', m_0\in \R$, $\vr\geq 0$, $r>0$, $p \in \{0,2\}$. Let   $M(U) $ be a real-to-real matrix of  $m$-operators in 
 	$ \Sigma\mM_2^{m}[r] $, $ \bF(U) $ be a real-to-real matrix of  $0$-operators  $ \mM^0_{\geq 0}[r]$ and $ \mathtt{p}(\xi)$ be a matrix of Fourier multipliers in $ \wt \Gamma_0^{m_0}$.
 	Then:
 	\begin{enumerate}
 		\item If  $c(U)$ is a $2$-homogeneous symbol in $\wt{\Gamma}_2^{m'}$ and $c_{\geq 4}(U)$ is a non-homogeneous symbol in $ \Gamma_{\geq 4}^m[r]$, 
 		$$
 		b_2(U;x,\xi):= c(-\ii  \mathtt{p}(D)U;x,\xi), \quad \text{and}  \quad \begin{cases}
 		b_{\geq 4}(U;x,\xi):= c(M(U)U,U;x,\xi)\\
 		b'_{\geq 4}(U;x,\xi):= c_{\geq 4}(\bF(U)U;x,\xi)
 		\end{cases}
 		$$
 		are symbols respectively in $\wt \Gamma^{m'}_{2}$ and  $ \Gamma_{\geq 4}^{m'}[r']$ for some $r'>0$;
 		\item If $Q(U)$ is a  $2$-homogeneous smoothing operator in $\widetilde{\mR}_2^{-\vr}$, 
 		$$
 		\widetilde R_2(U):= Q(-\ii  \mathtt{p}(D)U, U )\in \widetilde{\mR}_2^{-\vr+\max\{0,m_0\}} \quad \text{and} \quad 
 		R_{\geq 4}(U):= Q(M(U)U,U)\in \mR^{-\vr+\max\{0,m\}}_{\geq 4}[r];$$
 		\item  If $ R({U}) \in \Sigma\cR^{-\vr}_{ 2}\bra{r} $ and $ \mathsf{a}\pare{U;x, \xi}\in \Sigma\Gamma^m_{ 2}\bra{r} $, $ \vr \geq m $, then
 		\begin{align*}
 			R({U})\circ \Opbw{\mathsf{a}({U;x, \xi}) } \in \cR^{-\vr + m}_{\geq 4}\bra{r},
 			&&
 			\Opbw{\mathsf{a}\pare{U;x, \xi} } \circ R\pare{U} \in \cR^{-\vr + m}_{\geq 4}\bra{r}. 
 		\end{align*}
 		\item  If  $M $ is  in 
$ \Sigma\mathcal{M}^{m}_{p}[r]$ and $M'$ is  in
$ \Sigma\mathcal{M}^{m'}_{p'}[r] $ then the composition 
$ M\circ M'$  
is  in $\Sigma\mathcal{M}^{m+\max(m',0)}_{p+p'}[r]$.
\item 
If   $M(U) $ is in $ {\mathcal{M}}_{\geq 4}^{m}[r]$, 
then  $M(\bF(U)U)$ is in  $  \mM^{m}_{\geq 4}[r']$ for some $r'>0$.
 	\end{enumerate}
 	
 \end{lemma}

 \subsection{Admissible transformations}
{
 In this section we introduce a class of $U$-dependent transformations, that we call {\em admissible},  that have three properties:
 $(i)$ they are bounded as maps on Sobolev spaces of sufficiently high regularity, 
 $(ii)$ they are differentiable with respect to the internal variable $U$ and  $(iii)$ their differential 
may  lose $m$-derivatives in the external variable, but gain $\varrho$-derivatives in the internal one. 
Examples are flows of paradifferential and smoothing operators, see Lemma \ref{lem:flow.ad} and Lemma \ref{flow.s.ad}.
 \begin{definition}[Admissible transformations]\label{admtra}
Let $r>0$, {$m, \vr \geq 0$}. We say that a
real-to-real matrix  $\bF(U) $ of  non-homogeneous   $0$-operators in $\cM_{\geq 0}^0[r]$
   is an $m$-{\em admissible transformation} of {\em gain} $\vr$ if the following holds:
\begin{enumerate}
\item[(i)] {\bf Linear invertibility:}  $\bF(U)$ is linearly invertible  and its inverse  $\bF(U)^{-1}$ is a  real-to-real matrix of  non-homogeneous  ${0}$-operators in $\cM_{\geq 0}^{{0}}[r]$
satisfying the following: there exists $s_0>0$ such that
for any $s\geq s_0 + \vr $ there is a constant $C:=C_{s}>0$  and $r=r_s >0$ such that for any $ U\in B_{s_0,\R}(r)\cap {H^{s-\vr}_\R(\T;\C^2)}$ and $V \in H^s_\R(\T;\C^2)$ one has 
	\begin{align}
		&\| \bF(U)V\|_{s}+ \| \bF^{-1}(U)V\|_{s}\leq C (\| V \|_s + \|U\|_{s-\vr}\| U\|_{s_0}\|V\|_{s_0}) \ .  \label{lin.est.F}
		\end{align}
	\item[(ii)] {\bf Expansion:}  $\bF(U) - \uno$ 
    is a matrix of $m$-operators in $\Sigma\cM^m_2[r]$ expanding as
	\be\label{esp:F}
	\bF(U) = \uno+ \bF_{2}(U)+ \bF_{\geq 4}(U),
 \quad \bF_{2}(U)\in \wt \mM^m_2, \quad  \bF_{\geq 4}(U) \in \mM_{\geq 4}^m[r].
	\ee
		\item[(iii)]{\bf Derivative:} there is $ s_0\geq 0$ such that for any $\sigma \geq s_0 + \vr $, 
        the map 
		$$
		B_{\sigma-\vr, \R}(r)  \ni U \mapsto \bF(U)  \in \cL\big(H^{\sigma+m}_\R(\T,\C^2), \, H^{\sigma}_\R(\T,\C^2)\big) =: X^{\sigma,m}
		$$
        is differentiable.
        Moreover its differential $\di_U \bF (U)$ satisfies  the quantitative bound: there are $C=C_\s>0$, $r'=r'(\s)>0$ such that for any $U \in B_{\sigma-\vr,\R}(r')  $ and $\hat U \in H^{\sigma-\vr}_\R(\T;\C^2)$     \begin{equation}\label{stima.dF}
          \norm{\di_U \bF (U)[\widehat U] }_{X^{\s,m}} \leq C \| U \|_{\sigma-\vr} \| \hat U\|_{\sigma-\vr}   \ .
        \end{equation}
		Moreover, for any $s\geq s_0+m$, there is $C:=C_{s}>0$  such that 	for any $U\in B_{s_0,\R}(r)\cap H^{s}_\R(\T;\C^2)$,  $Z, \hat U \in H^{s}_\R(\T;\C^2)$ one has 
		\be
		\begin{aligned}\label{stima.d.adm}
		\| &  \left(  \di_U  \bF(U)[\hat U]  - \di_U \bF_2(U)[\hat U]\right) Z\|_{s-m}=\|   \di_U  \bF_{\geq 4}(U)[\hat U] Z \|_{s-m}\\
		& \leq  C
		\left(  \| U\|_{s_0}^3\| \hat U\|_{s_0}\| Z\|_s 
		+
		 \| U\|_{s_0}^3\| \hat U\|_{s}\| Z\|_{s_0} 
		 +
		  \| U\|_{s} \| U\|_{s_0}^2 \| \hat U\|_{s_0}\| Z\|_{s_0}
		\right).
		\end{aligned}
		\ee
\end{enumerate} 
\end{definition}
\begin{remark}\label{rem.b.ad}
$(1)$ Compared to $ m $-operators in $ \mathcal{M}_{\geq 0}^0[r] $, admissible transformations exhibit a gain of $ \varrho $ derivatives in the internal variable $ U$; see the second term in estimate \eqref{lin.est.F} and compare it with \eqref{piove} for $p = 0$. 
This additional gain will be verified since   the admissible transformations we consider are linear flows generated by either paradifferential operators or smoothing operators. In both cases, the internal variable gains derivatives with respect to the external one.\\
  $(2)$ Thanks to the  bound in \eqref{lin.est.F}, $\bF(U)$  conjugates  any matrix $B_{\geq 4}(U)$ of $0$-operators in $\mM^0_{\geq 4}[r]$ into another matrix of $0$-operators in $\mM^0_{\geq 4}[r]$, namely
  $
\bF(U) B_{\geq 4}(U) \bF(U)^{-1}
  $
  is a matrix of $0$-operators  in $\mM^0_{\geq 4}[r]$.\\
$(3)$ Property $(ii)$ implies that 
\begin{align}
		&\| \left[\bF(U)-\uno \right]V\|_{s-m}+ \| \left[\bF^{-1}(U)-\uno \right]V\|_{s-m}\leq C \| U\|_{s_0}^2 \| V \|_{s}+{\| U\|_{s_0}\| U\|_s \|V\|_{s_0}} \  \label{lin.est.F1} 
	\end{align}
	and that
\begin{equation}\label{tri.est.F2}
\|\di_U \bF_2(U)[\hat U]V \|_{{s-m}}
\lesssim_s \| U \|_{{s_0}} \| \hat U \|_{{s_0}} 
\| V \|_{ s}+ \| U \|_{{s_0} }\| \hat U \|_{s} \|V \|_{{s_0}}  + 
	\| U \|_{{s}} \| \hat U \|_{{s_0}} \|V \|_{s_0}	 \ .
\end{equation}
$(4)$ The expansion \eqref{esp:F} for $\bF(U)$ implies the corresponding expansion for $\bF(U)^{-1}$:
\be\label{esp:F_inv}
\bF(U)^{-1}=\uno - 	\bF_{2}(U)+ \breve \bF_{\geq 4}(U),
\ee
where 
$
   \breve \bF_{\geq 4}(U):= -\bF(U)^{-1}\bF_{\geq 4}(U) + \bF(U)^{-1}[\bF(U)-\uno] \bF_2(U)  $
   is a real-to-real matrix of $2m$-operators in $\mM_{\geq 4}^{2m}[r]$.
 \end{remark} 
We now prove that admissible transformations are closed by composition.
\begin{lemma}\label{lem:comp}
Let $\bF^{(1)}(U)$ be $m_1$-admissible with gain $\vr_1$ and  $\bF^{(2)}(U)$ be $m_2$-admissible with gain $\vr_2$.
If 
$m_1 \leq \vr_2$, 
then  the composition 
$\bF^{(1)}(U)\bF^{(2)}(U)$ is a $m_1+m_2$-admissible transformation
with gain $\vr:= \min(\vr_2 - m_1, \vr_1)$.
\end{lemma}
\begin{proof}
 We set $m:=m_1+m_2$. 
$(i)$ and $(ii)$ follows by the composition properties of $m$-operators, see Lemma \ref{nuovetto}-4,  and by applying twice  estimate \eqref{lin.est.F} and using also $ \vr\leq \min\{ \vr_1,\vr_2\}$.
Moreover we have the expansion 
\be \label{esp:compo}
\bF^{(1)}(U)\bF^{(2)}(U)=\uno +\bF^{(1)}_2(U)+\bF^{(2)}_2(U)+\bF^{(1,2)}_{\geq 4 }(U)
\ee
where
$\bF^{(1,2)}_{\geq 4 }(U)= \bF^{(1)}_{\geq 4 }(U)+\bF^{(2)}_{\geq 4 }(U)+ \left(\bF^{(1)}_2(U)+\bF^{(1)}_{\geq 4 }(U)\right)\left(\bF^{(2)}_2(U)+\bF^{(2)}_{\geq 4 }(U)\right)\in \mM^{m}_{\geq 4}[r]$.\\
 $(iii)$  
Set $s_0:= s_0^{(1)}+s_0^{(2)}$ with $s_0^{(j)}$, $j=1,2$,   the regularity threshold in property $(iii)$ for $ \bF^{(j)}$.
We first prove that, for any $ \sigma \geq s_0+\vr$, $U \mapsto \bF^{(1)}(U) \bF^{(2)}(U)$ is differentiable at $U \in B_{\sigma-\vr, \R}(r) $, $r>0$ sufficiently small,   and its differential is given by
 \be\label{dFG}
 \di_U \big( \bF^{(1)}(U) \bF^{(2)}(U)  \big)[\widehat U]  = (\di_U \bF^{(1)}(U)[\widehat U]) \, \bF^{(2)}(U)   
 +
\bF^{(1)}(U) \, (\di_U \bF^{(2)}(U)[\widehat U])   \ . 
 \ee
 Indeed fix $U\in B_{\sigma-\vr, \R}(r)$, take $\wh U$ with  $\norm{\hat U}_{\s-\vr} \ll   r$ and put
  \begin{align*}
 \bQ(U,\hat U):=&\bF^{(1)}(U+ \hat U)\bF^{(2)}(U+ \hat U)\\
 &-\bF^{(1)}(U)\bF^{(2)}(U)- \left((\di_U \bF^{(1)}(U)[\widehat U]) \, \bF^{(2)}(U)+\bF^{(1)}(U) \, (\di_U \bF^{(2)}(U)[\widehat U]) \right)\\   
 =&\left( \bF^{(1)}(U+ \hat U)- \bF^{(1)}(U)- \di_U \bF^{(1)}(U)[\hat U]\right)\bF^{(2)}(U+ \hat U)\\
 &+ \bF^{(1)}(U) \left( \bF^{(2)}(U+\hat U)-\bF^{(2)}(U)- \di_U \bF^{(2)}(U)[\hat U]\right)+\di_U \bF^{(1)}(U)[\hat U]\left( \bF^{(2)}(U+\hat U )-\bF^{(2)}(U)\right) \\
 & = \bQ_1(U,\hat U) +\bQ_2(U,\hat U) + \bQ_3(U,\hat U).
\end{align*}
We show that for $j=1,2,3$
\be\label{est.Q}
\norm{\bQ_j(U,\hat U)}_{X^{\s,m}}\lesssim \norm{\hat U}_{\s-\vr}^2
\ee
proving formula \eqref{dFG}.
Consider first $\bQ_1(U,\hat U) V$ with $V\in H^{\s + m}_{\R}(\T, \C^2)$.
Using the differentiability of $ \bF^{(1)}(U)$,  estimate  \eqref{lin.est.F} for $ \bF^{(2)}(U+ \hat U)$ and that $\vr =\min( \vr_2 - m_1, \vr_1)$
we get that 
\begin{align*}
\norm{\bQ_1(U,\hat U) V}_{\s} & \lesssim \norm{\bF^{(1)}(U+ \hat U) - \bF^{(1)}(U)- \di_U \bF^{(1)}(U)[\hat U]}_{X^{\s,m_1}} \| \bF^{(2)}(U+ \hat U)V \|_{\sigma+m_1}\\
  & \lesssim \| \hat U\|_{\s-\vr_1}^2 (\|V\|_{\s+m_1}+ \| U+ \hat U\|_{\s-(\vr_2-m_1)}\| U+ \hat U\|_{s_0}\| V\|_{s_0}) {\lesssim} \| \hat U\|_{\s-\vr}^2 \|V\|_{\s+m_1},
\end{align*}
proving \eqref{est.Q} for $j=1$ as $m \geq m_1$. 
We now prove the estimate for $j=2$. Using \eqref{lin.est.F} and  the differentiability of $ \bF^{(2)}$, we get 
\begin{align*}
\norm{\bQ_2(U,\hat U) V}_{\s}  \lesssim &\norm{\left( \bF^{(2)}(U+\hat U)-\bF^{(2)}(U)- \di_U \bF^{(2)}(U)[\hat U]\right)V}_{\s}\\
&+ \| U \|_{\s -\vr_1}\| U\|_{s_0} \| \left( \bF^{(2)}(U+\hat U)-\bF^{(2)}(U)- \di_U \bF^{(2)}(U)[\hat U]\right)V\|_{s_0} \\
  \lesssim & \| \hat U\|_{\s-\vr_2}^2 \|V\|_{\s+m_2}+ \| U\|_{\s-\vr_1}\|U\|_{s_0}\| \hat U\|_{\s-\vr_2}^2\| V\|_{\s+m_2} {\lesssim} \| \hat U\|_{\s-\vr}^2 \|V\|_{\s+m}
\end{align*}
proving  also \eqref{est.Q} for $j=2$
Consider now $j=3$. 
Applying first \eqref{stima.dF} for $\di_U \bF^{(1)}(U)[\hat U]$  with $ m \leadsto m_1 $, then 
writing $\bF^{(2)}(U + \hat U) - \bF^{(2)}(U) = \int_0^1 \di_U \bF^{(2)}(U + \tau \hat U)[\hat U] \di \tau $
and using  \eqref{stima.dF} for $\di_U \bF^{(2)}(U+\tau \hat U)[\hat U]$, $\tau \in [0,1]$ with $m\leadsto m_2$ and $\s\leadsto \s+m_1$ we get
\begin{align*}
    \norm{\bQ_3(U,\hat U)V}_{\s} 
    {\lesssim}&   \norm{ U}_{\s-\vr_1}\norm{\hat U}_{\s-\vr_1}\norm{ \bF^{(2)}(U+\hat U)V-\bF^{(2)}(U)V}_{\s+m_1}\\
    \lesssim&  \norm{ U}_{\s-\vr_1} \norm{\hat U}_{\s-\vr_1}\,  
  \norm{ U}_{\s+m_1-\vr_2}\,   \| \hat U\|_{\s+m_1-\vr_2}\,  \| V\|_{\s+m} 
    {\lesssim} \norm{ U}_{\s-\vr}^2  \| \hat U \|_{\s-\vr}^2 \| V\|_{\s+m},
\end{align*}
proving  also \eqref{est.Q} for $j=3$. We conclude that  \eqref{dFG} holds.\\
Next we show that 
$\di_U \big( \bF^{(1)}(U) \bF^{(2)}(U)  \big) $
 fulfills estimate \eqref{stima.dF}. So fix $\widehat U \in H^{\s-\vr}_\R(\T, \C^2)$ and 
 $V \in H^{\s+m}_\R(\T, \C^2)$ and 
 consider the first term in the right hand side of \eqref{dFG}. We have
 \begin{align*}
    &  \norm{(\di_U \bF^{(1)}(U)[\widehat U]) \, \bF^{(2)}(U)   V  }_{\s} \stackrel{\eqref{stima.dF}}{\lesssim} \norm{U}_{\s-\vr_1} \norm{\wh U}_{\s-\vr_1} \norm{\bF^{(2)}(U)   V }_{\s+m_1} \\
     \stackrel{\eqref{lin.est.F}}{\lesssim}  & 
     \norm{U}_{\s-\vr_1} \norm{\wh U}_{\s-\vr_1} \left( \norm{V}_{\s  +m_1} + \norm{U}_{\s + m_1-\vr_2}\norm{U}_{s_0}\norm{V}_{s_0} \right)
     {\lesssim }
     \norm{U}_{\s-\vr}  \norm{\wh U}_{\s-\vr}  \norm{V}_{\s +m} 
 \end{align*}
 The second term in \eqref{dFG} has an  analogous estimate, proving \eqref{stima.dF}.\\ 
Finally we prove the estimate \eqref{stima.d.adm}. First we compute the differential 
\be
\begin{aligned}
\di_U \bF^{(1,2)}_{\geq 4 }(U)[\hat U]Z=& \di_U \bF^{(1)}_{\geq 4 }(U)[\hat U]Z+\di_U \bF^{(2)}_{\geq 4 }(U)[\hat U]Z\\
&+\left(\di_U\bF^{(1)}_2(U)[\hat U ]+\di_U \bF^{(1)}_{\geq 4 }(U)[\hat U]\right)\left(\bF^{(2)}_2(U)+\bF^{(2)}_{\geq 4 }(U)\right)Z\\
&+\left(\bF^{(1)}_2(U)+\bF^{(1)}_{\geq 4 }(U)\right)\left(\di_U\bF^{(2)}_2(U)[\hat U ]+\di_U \bF^{(2)}_{\geq 4 }(U)[\hat U]\right)Z.
\end{aligned}
\ee
Estimate \eqref{stima.d.adm} for $\di_U \bF^{(1,2)}_{\geq 4 }(U)[\hat U]Z$ follows from the corresponding estimates for $\di_U \bF^{(1)}_{\geq 4 }(U)[\hat U]Z$, $\di_U \bF^{(2)}_{\geq 4 }(U)[\hat U]Z$ in \eqref{tri.est.F2} and \eqref{piove}--\eqref{smoothing} for $\bF_2^{(1)}(U)$, $\bF_2^{(2)}(U)$, $\bF_{\geq 4}^{(1)}(U)$ and $\bF_{\geq 4}^{(2)}(U)$.
\end{proof}
Next we prove a local invertibility property of the nonlinear map $U \mapsto \bF(U)U$ when $\bF(U)$ is an admissible transformation.  
\begin{lemma}\label{loc.inv}
	Let  $\bF(U)$ be  a $m$-admissible transformation with gain $\vr\geq {\max\{ m, 1\}}$. 
    Consider the  nonlinear map 
	$\mF(U):= \bF(U)U$. 
	The following holds true:
	\begin{itemize}
	\item[(i)] There exists $s_0' \geq 0$ such that for any $s \geq s_0'$, the  map $\cF^{-1}$ is locally invertible: namely there is $r'>0$ and $\mF^{-1}: B_{s_0',\R}(r')\cap H^{s}_\R(\T;\C^2)\to H^{s}_\R(\T;\C^2)$ such that 
	$$
	\mF\circ \mF^{-1}(V)=V, \quad \mF^{-1}\circ \mF(U)=U, \quad \forall U,V\in B_{s_0',\R}(r') \ .
	$$
	\item[(ii)] One has  $\cF^{-1}(V) = \bG(V)V$ with $\bG(V) $ a matrix of non-homogeneous $0$-operators in $\cM_{\geq 0}^0[r']$ such that $\bG(V) - \uno \in \Sigma \cM_2^{2m}[r']$ for some $r' >0$ and  expands  as
	\begin{equation}\label{exp.G}
	\bG(V) = \uno  - \bF_2(V) + \bG_{\geq 4}(V) \ , 
	\quad \bG_{\geq 4}(V) \in \cM^{2m}_{\geq 4}[r'] \ .
	\end{equation}
	\end{itemize}
\end{lemma}
\begin{proof}
Let $s_0, r >0$  the parameters given by Definition \ref{admtra} associated to  $\bF(U)$.\\
$(i)$  
Let $\sigma_0:= s_0 + \vr$.
We prove that there exists $r_1>0$ such that for any  $V \in B_{\sigma_0 +m, \R}(r_1)$ there is a unique solution  $U=\mF^{-1}(V) \in B_{\s_0, \R}(r)$  of the equation 
	$	V= \mF(U)= \bF(U)U$. Then we show that if $V \in H^{s}_\R(\T, \C^2)$, $s > \s_0$, also $U \in H^{s}_\R(\T, \C^2)$.\\
Exploiting the linear  invertibility  of $\bF(U)$, we recast 
$V=\bF(U)U$ 
 as the fixed point problem 
\begin{equation}\label{ift}
	\mG(U;V):=\bF(U)^{-1} V = U \ .
	\end{equation}
First we  show that 
for any $V \in  B_{\s_0+m, \R}(r_1)$, the map $U \mapsto \mG(U;V)$ is a contraction on the ball  $B_{\s_0, \R}(r)$ provided $r_1>0$ is small enough. \\
{\underline{$\mG(U;V)$ maps the ball into itself.}} 
Let $V \in  B_{\s_0+m, \R}(r_1)$ and $U \in  B_{\s_0, \R}(r)$. It follows from
\eqref{lin.est.F} that
\be
\norm{\mG(U;V)}_{\s_0} \leq C \left( \norm{V}_{\s_0} +  \norm{U}_{s_0}^2 \norm{V}_{s_0} \right) \leq C r_1 (1+r^2) \leq r
\ee
which is verified provided $r_1$ is sufficiently small.\\
{\underline{$\mG(U;V)$ is a contraction.}}
Again let $V \in  B_{\s_0+m, \R}(r_1)$ and $U_1, U_2 \in  B_{\s_0, \R}(r)$.
By $(iii)$ one has  
\begin{align}\bF(U_1) - \bF(U_2) = \int_0^1 \di_U \bF(\tau U_1 + (1-\tau)U_2)[U_1 - U_2]  \di \tau  \  , 
\label{differenza}
\end{align}  
which  applying $\bF(U_1)^{-1}$ to the left and $\bF(U_2)^{-1}$ to the right  yields
\begin{equation}\label{Fuinv.diff}
    \bF(U_1)^{-1} - \bF(U_2)^{-1} = -\int_0^1 \bF(U_1)^{-1} \, \di_U \bF(\tau U_1 + (1-\tau)U_2)[U_1 - U_2]\,  \bF(U_2)^{-1}  \di \tau  \ .
\end{equation}
Exploiting such formula we get 
\begin{align*}
\norm{\mG(U_1;V) & - \mG(U_2;V)}_{\s_0}  \stackrel{\eqref{lin.est.F}}{\leq }
\sup_{\tau \in [0,1]}\norm{\di_U \bF(\tau U_1 + (1-\tau)U_2)[U_1 - U_2]\,  \bF(U_2)^{-1} V}_{\s_0} \\
& \ \ \ +  \norm{U_1}_{s_0}^2 
\sup_{\tau \in [0,1]}\norm{\di_U \bF(\tau U_1 + (1-\tau)U_2)[U_1 - U_2]\,  \bF(U_2)^{-1} V}_{s_0} 
\\
& \stackrel{\eqref{stima.dF}}{\leq } C \left( \norm{U_1}_{s_0} + \norm{U_2}_{s_0} \right) \, \norm{U_1- U_2}_{s_0} \, \norm{F(U_2)^{-1}V}_{\s_0 + m} \\
& \stackrel{\eqref{lin.est.F}}{\leq}
C \left( \norm{U_1}_{\s_0} + \norm{U_2}_{\s_0} \right) \, \norm{U_1- U_2}_{\s_0} \left( \norm{V}_{\s_0+m} + \norm{U_2}_{s_0} \norm{U_2}_{s_0 + m } \norm{V}_{s_0} \right) \\
& \leq C \left( \norm{U_1}_{\s_0} + \norm{U_2}_{\s_0} \right) \, \norm{U_1- U_2}_{\s_0} \,  \norm{V}_{\s_0+m} \leq \frac12 \norm{U_1- U_2}_{\s_0} 
\end{align*}
where in the last step we chose $r_1>0$ small enough.  
By Banach fixed point theorem, there is a unique 
$U \in B_{\s_0,\R}(r)$ solving the fixed point problem \eqref{ift}, and so we put 
\be\label{sol.imp}
\cF^{-1}(V) := U \ , \mbox{ so that } \cG(\cF^{-1}(V); V) =  \cF^{-1}(V) \  . 
\ee
\underline{Upgraded regularity.} We now show that for any $s\geq \s_0+m$,
if $V\in B_{\s_0+m,\R}(r_1) \cap  H^s_\R(\T, \C^2)$,  then  $\mF^{-1}(V)$ belongs to
$ H^s_\R(\T, \C^2)$   and
\begin{equation}\label{vs.us}
 \|\cF^{-1}(V)\|_s \leq 2 C_s \|V\|_s\,.
\end{equation}
First, from the fixed point, $U:=\cF^{-1}(V) \in B_{\s_0,\R}(r)$. 
Now fix  $\underline{n}\in \N$ so that  $s \in(\s_0+\underline{n}\vr , \s_0+\underline{n}\vr+1]$. 
Then, from equation \eqref{ift} and estimate \eqref{lin.est.F}, 
we get
$$
\norm{U}_{\s_0+n\vr} \leq C \big( \norm{V}_{\s_0+n\vr} + \norm{U}_{\s_0+(n-1)\vr} \norm{U}_{s_0}\norm{V}_{s_0} \big)  \ , \quad n =1 , \ldots, \underline{n} \ . 
$$
This shows that  $U\in H^{\s_0+\underline{n}\vr}_{\R}(\T;\C^2)$ and, using also that $U\in B_{s_0,\R}(r)$, $V\in B_{s_0+m,\R}(r_1)$, we get
\begin{align}
\norm{U}_{\s_0+n\vr} \leq C \norm{V}_{\s_0+n\vr}, \quad n=1,\ldots, \underline{n}. 
 \label{junatarella}   
\end{align}
Finally, using $ s-\vr \leq \s_0+\underline{n}\vr<s$ and again that
 $U\in B_{s_0,\R}(r)$, $V\in B_{s_0+m,\R}(r_1)$, we deduce
$$
\| U\|_s 
= 
\| \bF^{-1}(U) V \|_s 
\stackrel{\eqref{lin.est.F}}{\leq}C\left(\| V \|_s+ \| U\|_{s-\vr}\|U\|_{s_0}\|V\|_{s_0}\right)\stackrel{\eqref{junatarella}}{\leq} C \| V\|_{s}.
$$
So far we have shown that
$\cF\circ \cF^{-1}(V) = V$ for any $V \in B_{s_0'}(r_1)$, 
where  $s_0' = s_0 + \vr + m$. Now we show that
$\cF^{-1}\circ \cF(U) = U$ provided $U \in B_{s_0',\R}(r')$, with a smaller $r'$. 
First of all, note that $ \cF^{-1} \circ \cF(U) $ solves the fixed point equation \eqref{ift} with $ V \leadsto \cF(U) $ and $ U \leadsto \cF^{-1} \circ \cF(U) $.
When $ \mF(U) \in B_{s_0', \R}(r_1) $, the map $ \cG(\,\cdot\,;\cF(U)) $ is a contraction. As a result, the associated fixed point problem admits a unique solution, which must therefore coincide with $ U $.
We prove now that $ \mF(U) \in B_{s_0', \R}(r_1) $. Indeed 
estimate
 \eqref{lin.est.F}, for some $C>1$, gives
$$
\| \mF(U)\|_{s_0'}=\| \bF\left(U\right)U\|_{s_0'}\leq C\|U\|_{s_0'}\leq r_1
$$
for any $U\in B_{s_0' ,\R}(r_1/C)$. The thesis of item $(i)$ follows by choosing $r':= r_1/C$.\\
$(ii)$  It follows from \eqref{sol.imp}  and \eqref{ift}
\begin{equation}\label{cF-1}
\cF^{-1}(V) = \bG(V) V \ , \quad  \bG(V):= \bF(\cF^{-1}(V))^{-1}\in \cM^0_{\geq 0}[r'] .  
\end{equation}
Since by definition $r'=r_1/C\leq r_1$, by the fixed point theorem $ \mF^{-1}(V)\in B_{s_0+\vr,\R}(r)$ for any $V \in B_{s_0',\R}(r')$. 
Then, since $\bF(U)^{-1}$ is a a real-to-real matrix of non-homogeneous $0$-operators in $\cM^0_{\geq 0}[r]$, it follows that $ \bG(V)$ is a real-to-real matrix of non-homogeneous $0$-operators in $\cM^0_{\geq 0}[r']$ (with $s_0\leadsto s_0'$).\\
Next we  show that $\bG(V)$ expands as in \eqref{exp.G}. 
Put  $\wt \cF^{-1}(V) := V - \bF_2(V)V$.  Then, using the expansion 
$\cF(U) = U + \bF_2(U)U + \bF_{\geq 4}(U)U$ and Lemma \ref{nuovetto}, we get 
$$
(\wt \cF^{-1}\circ \cF)(U) = U  + \bF'_{\geq 4}(U)U \ , \quad \mbox{ with }\bF'_{\geq 4}(U) \in \cM^{2m}_{\geq 4}[r]. \ 
$$
Substituting $U = \cF^{-1}(V)$ and using \eqref{cF-1}  and Lemma \ref{nuovetto},  we obtain
$$
\cF^{-1}(V) =  V - \bF_2(V) V + \bG_{\geq 4}(V) V, \quad 
 \bG_{\geq 4}(V):= -
\bF'_{\geq 4}(\cF^{-1}(V)) \bG(V)  \in \cM^{2m}_{\geq 4}[r'] \ . 
$$
This proves the expansion in 
\eqref{exp.G}.
\end{proof}
An immediate consequence of the above lemma is that the inverse $\mF^{-1}$ of an admissible transformation $\mF$ fulfills the estimate  
\begin{equation}\label{stima.inv.adm}
\|\mF^{-1}(V)\|_s\leq C_s  \|V\|_s, \quad \text{for any} \ V\in B_{s_0',\R}(r')\cap H^s(\T;\C^2)  \ .
\end{equation}
 We now show that the linear flows generated by two types of  paradifferential operators are admissible transformations. 
 Consider the flows 
 \be\label{flussoG}
\begin{cases}
\pa_\tau \Phi^\tau(U)= G(\tau,U)\Phi^\tau(U)\\
\Phi^0(U)=\uno
\end{cases} \quad \text{where}\quad
 G(\tau,U)=
 \begin{cases}
 \vOpbw{ \frac{\beta(U;x)}{1+\tau \beta_x(U;x)}\ii \xi},&\beta\in \wt \mF_2^\R \ \  \text{or}\\
\zOpbw{ g(U;x,\xi)},& g\in \wt \Gamma_2^0.
 \end{cases}
 \ee
 \begin{remark}\label{rem:gauge.inv}
The map $\Phi^\tau(U)$	is gauge invariant if 
the generator $G(U;\tau)$ is gauge invariant.
Indeed 
$\Phi^\tau(\tg_\theta U)\tg_\theta$ and 
$\tg_\theta \Phi(U) $ solve the same equation, thus coinciding. 
\end{remark}
 The following lemma ensures that the  flow map $\Phi^\tau(U)$  generated by $G(\tau, U)$ is an admissible transformation for any $\tau\in [0,1]$.
\begin{lemma}\label{lem:flow.ad}
Let $\Phi^\tau(U)$ be the flow map in \eqref{flussoG}. Fix an arbitrary $\vr \geq 0$. Then\\
  $(i)$ if $G(\tau,U)= \vOpbw{ \frac{\beta(U;x)}{1+\tau \beta_x(U;x)}\ii \xi}$ then 
  $\Phi^\tau(U)$  is a $2$-admissible transformation with gain $\vr$;\\
 $(ii)$ if  $G(\tau,U)= \zOpbw{ g(U;x,\xi)}$ then 
 $\Phi^\tau(U)$ is a $0$-admissible transformation with gain $\vr$.
\end{lemma} 
 \begin{proof}
 Along the proof we put  $m=2$ if $G(\tau, U)$ is as in $(i)$, and $m=0$ in case $(ii)$.\\
It is classical that $\Phi^\tau(U)$ is a matrix of $0$-operators in $ \cM^0_{\geq 0}[r]$ as well as its linear inverse,  see e.g. Lemma 3.16 of \cite{BMM2}.
In particular, estimate (3.53) in  \cite{BMM2} (with $k = K' =K=0$) gives that  for any 
$ U\in B_{s_0,\R}(r)\cap {H^{s-\vr}_\R(\T;\C^2)}$ and $V \in H^s_\R(\T;\C^2)$,  $\sup_{\tau \in [0,1]}\norm{\Phi^\tau(U)^{\pm 1}V}_s \leq C \norm{V}_s$,  which clearly implies both 
\eqref{lin.est.F} and the second of \eqref{piove}.

We prove now the  expansion \eqref{esp:F}; first expand 
 \begin{equation}\label{G:exp}
 G(\tau,U)=G_2(U)+  G_{\geq 4}(\tau,U)=  \begin{cases}
 	\vOpbw{ {\beta(U;x)}\ii \xi}+G_{\geq 4}(\tau,U)\\
 	\zOpbw{ g(U;x,\xi)} \, ,
 \end{cases} 
 \end{equation}
so  the expansion of $\Phi^\tau(U)$ reads 
$$
\Phi^\tau(U)= \uno + \tau G_2(U) + \bF_{\geq 4}(\tau, U), \quad \bF_{\geq 4}(\tau, U)\in \mM^m_{\geq 4}[r].
$$
We prove now $(iii)$. 
First we claim that, for both choices of $G(\tau, U)$ in \eqref{flussoG}, there is $s_0 >0$ such that   for any  $s\in \R$ 
\begin{align} 
	&\sup_{\tau \in [0,1]}\norm{\di_U G(\tau,U)[\widehat{U}]W }_{s-\frac{m}{2}} \lesssim \norm{U}_{s_0}\norm{\widehat{U}}_{s_0}\norm{W}_{s}, \label{claimata} \\
&\sup_{\tau \in [0,1]}\norm{\di_U G_{\geq 4}(\tau,U)[\widehat{U}]W }_{s-\frac{m}{2}} \lesssim \norm{U}_{s_0}^3\norm{\widehat{U}}_{s_0}\norm{W}_{s}.
\label{claimata1}
\end{align}
Assuming for the moment such properties, consider the differential 
 ${\rm d}_U \Phi^\varsigma(U)[\widehat U]$. 
 It   fulfills the variational equation 
\be 
\begin{cases}
\partial_\varsigma {\rm d}_U \Phi^\varsigma (U)[\widehat U] =G(\varsigma,U) {\rm d}_U \Phi^\varsigma(U)[\widehat U]
+\di_U G(\varsigma,U)[\widehat{U}] \Phi^\varsigma(U) \\
{\rm d}_U \Phi^0(U)[\widehat U]=0 \, ,
\end{cases}   
\ee
 whose solution is given by the  Duhamel formula 
\footnotesize
\begin{align}
{\rm d}_U \Phi^\varsigma(U)[\widehat U] =& \Phi^\varsigma(U) \int_0^\varsigma  \Phi^\tau (U)^{-1} \ \di_U G(\tau,U)[\widehat{U}]\   \Phi^\tau (U)\, {\rm d}\tau \label{prima_minale} \\
\stackrel{\eqref{G:exp}, \eqref{flussoG}}{=} &\varsigma \di_U G_2(U)[\widehat{U}]+\varsigma \int_0^\varsigma G(\theta,U)\Phi^\theta(U)\di_U G_2(U)[\widehat{U}]\, \di \theta + \Phi^\varsigma(U)\int_0^\varsigma \di_U G_{\geq 4}(\tau,U)[\widehat{U}]\, \di \tau \\
&+\Phi^\varsigma(U) \int_0^\varsigma \int_0^\tau  \Phi^\theta  (U)^{-1} \  \left[ \di_U G(\tau,U)[\widehat{U}], G(\theta,U)\right]\   \Phi^\theta (U)\,\di \theta  \,{\rm d}\tau\label{minale}
\end{align}
\normalsize
where in the second equality we also used the expansion
\footnotesize
$$\Phi^\theta (U)^{-1} \ \di_U G(\tau,U)[\widehat{U}]\   \Phi^\theta (U)_{|\theta=\tau}= \di_U G(\tau,U)[\widehat{U}]+ \int_0^\tau  \Phi^\theta  (U)^{-1} \  \left[ \di_U G(\tau,U)[\widehat{U}], G(\theta,U)\right]\   \Phi^\theta (U)\,\di \theta.$$
\normalsize
Inserting  estimates \eqref{claimata}--\eqref{claimata1} in  \eqref{minale} and using \eqref{piove} for $\Phi^\varsigma(U)$ and   \eqref{stimapar2} for $G(\tau, U)$,  one checks that  for any $\s \geq s_0 + \vr$
\be
\norm{\di_U \Phi^\varsigma(U)[\widehat U] V }_\s \leq C \norm{U}_{s_0} \, \norm{\wh U}_{s_0} \norm{W}_{\s+2}
\leq \norm{U}_{\s-\vr} \, \norm{\wh U}_{\s - \vr} \norm{W}_{\s+2}
\ee
showing the validity of  \eqref{stima.dF}. \\
Similarly one checks that 
the term  $\left(\di_U\Phi^\tau(U)[\hat U]- \tau\di_U G_2(U)[\hat U]\right)W$ fulfills \eqref{stima.d.adm}.
We now prove {\eqref{claimata}--\eqref{claimata1}}. Consider  first $ G(\tau,U)= \zOpbw{ g(U;x,\xi)}$, for which \eqref{claimata1} is trivial (being $G_{\geq 4 }(\tau,U)\equiv 0$).
 Since $g(U;\cdot)$ is homogeneous of degree 2, 
$$
\di_U G(\tau,U)[\widehat{U}]=\zOpbw{ 2 g(\widehat{U}, U;x,\xi)}=\di_U G_2(U)[\hat U] \ ,
$$
and  \eqref{claimata} follows from  Theorem \ref{thm:contS}.\\
Next we analyze  the case  $ G(\tau,U)=  \vOpbw{ \frac{\beta(U;x)}{1+\tau \beta_x(U;x)}\ii \xi}$. Its differential is given by 
\footnotesize
\begin{align*}
\di_U G(\tau,U)[\widehat{U}]
& = 
2\vOpbw{ 
\beta(\widehat{U},U;x)\ii \xi}  -{2\tau}  \vOpbw{\frac{\beta(\widehat{U},U;x) \beta_x(U;x)}{1+\tau \beta_x(U;x)}
+ \frac{\beta(U;x)\beta_x(\widehat{U}, U;x)}{(1+\tau \beta_x(U;x))^2}\ii \xi}\\
&= \di_U G_2(U)[\hat U]+ \di_U G_{\geq 4}(\tau, U)[\hat U]
\end{align*}
\normalsize
Now notice that $\beta(\widehat{U},U;x)  \in \wt \cF^\R_{2}$ and
$$
\tb(\tau, \widehat U, U):= \frac{\beta(\widehat{U},U;x) \beta_x(U;x)}{1+\tau \beta_x(U;x)}
+ \frac{\beta(U;x)\beta_x(\widehat{U}, U;x)}{(1+\tau \beta_x(U;x))^2}
\in L^\infty(\T; \R)$$ with bound  
$ \sup_{\tau \in [0,1]} \norm{\tb(\tau, \widehat U, U)}_{L^\infty}
\lesssim \norm{\widehat{U}}_{s_0}\norm{U}_{s_0}^3.$ 
Then Theorem  \ref{thm:contS} gives 
 \eqref{claimata} and \eqref{claimata1}.
 \end{proof}
Next we consider the flow map  generated by a matrix of  smoothing operators:
 \be\label{flussoR}
\begin{cases}
\pa_\tau \Phi^\tau(U)= R(U)\Phi^\tau(U)\\
\Phi^0(U)=\uno
\end{cases} \quad \text{where}\quad
 R(U) \in \wt \cR^{-\varrho}_2 \ . 
 \ee
\begin{lemma}\label{flow.s.ad}
Let $\vr>0$. The flow $\Phi^\tau(U)$ in \eqref{flussoR} is a $0$-admissible transformation with gain $\vr$. 
\end{lemma} 
 \begin{proof}
Since  $R(U) \in \wt\cR^{-\varrho}_2$,
 	for any 
	$ U \in { B_{s_0, \R}(r)}$  with $s_0>0$ sufficiently large  
	the problem \eqref{flussoR} admits a unique solution  
	$\Phi^{\tau}(U)$ fulfilling
    $\norm{\Phi^\tau(U)V}_{s_0} \leq C \norm{V}_{s_0}$ uniformly for $\tau \in [-1,1]$. 
    We now prove that
    $\Phi^\tau(U)$ fulfills
    \eqref{lin.est.F}. 
Let $s \geq s_0+ \vr$, 
  $ U \in { B_{s_0, \R}(r_s)} \cap H^{s-\vr}_{\R}(\T, \C^2)$ with a sufficiently small $r_s>0$  {and} $V \in H^{s}_\R(\mathbb{T};\C^2)$. 
   Then   the integral formula  $\Phi^\tau(U) = \uno  + \int_0^\tau R(U) \Phi^{\tau'}(U) \di \tau'$ and estimate  \eqref{smoothing} (with $m = -\vr$ and $s \leadsto s-\vr$) yield
	\begin{equation}\label{flowsmoothestimate}
		\begin{aligned}
			 \norm{ \Phi^\tau(U)V }_{s}\leq & C_s  \left(\norm{ V}_{s }
            +
             \norm{V}_{s_0} \norm{U}_{s-\vr} \norm{U}_{s_0}\right)
            +
             C_s\sup_{\tau \in [-1,1]}\norm{\Phi^\tau(U)V}_{s-\vr} \norm{U}_{s_0}^2 \ .
		\end{aligned}
	\end{equation}
Then, possibly shrinking $r_s$ so that $C_s r_s^2 < \frac{1}{2}$, we obtain
$$ \sup_{\tau \in [-1,1]} \norm{ \Phi^\tau(U)V }_{s}\leq 2  C_s  \left(\norm{ V}_{s } + \norm{V}_{s_0} \norm{U}_{s-\vr} \norm{U}_{s_0}\right),
             $$
            proving \eqref{lin.est.F}.
The rest of the proof follows along the same lines as the previous one.
 The algebraic expansion \eqref{minale} holds with $ G(\tau,U)\leadsto R(U)$ and, since  
 $ \di_U R(U)[\widehat U] = 2 R(U, \widehat U)$,
 we replace  \eqref{claimata} and \eqref{claimata1} with  the bound 
 $$
\norm{ \di_U R(U)[\widehat U]W}_\sigma \leq C \norm{U}_{\s-\vr} \norm{\wh U}_{\s-\vr} \norm{W}_{\s-\vr}
 $$
 obtained from \eqref{smoothing} with $m =- \vr$ and $s-m \leadsto \sigma$.
 Then both \eqref{stima.dF}
and 
 \eqref{stima.d.adm} follow.
 \end{proof} 
}
 \section{Analysis of weak resonances}\label{sec:Lnormal}
Equation \eqref{eq:main} is Hamiltonian, with Hamiltonian function given by 
\begin{equation}
\label{hamiltonian}
\sH(u) := 
\int_\T  (|D|^\alpha u ) \bar u + \frac{\im}{4} \int_\T |u|^2 \, (\bar u \, u_x - u \bar u_x) \, \di x  \ . 
\end{equation}
Due to the gauge and translation invariance of equation \eqref{eq:main}, any sufficiently regular solution $u(t)$ of  \eqref{eq:main}  conserves the   total mass and momentum, namely 
\begin{equation}
\label{MP}
\begin{aligned}
& \sM(u(t)) :=\frac{1}{2\pi}  \norm{u(t)}_{L^2}^2 \equiv  \frac{1}{2\pi} \int_\T |u(t,x)|^2 \di x = \sum_{k \in \Z} | u_k(t)|^2 = \sM(u(0)) ,   \\
& \sP(u(t)) := \frac{1}{2\pi} \int_\T  \ii  (\pa_x u(t,x))  \, \bar{u(t,x)}  \, \di x = -\sum_{k \in \Z} k | u_k(t)|^2 = \sP(u(0))  \ . 
\end{aligned}
\end{equation}
In view of this  
we introduce the new  variable
\be 
v(t,x):= e^{\ii t \sP(u(t))}u\big(t,x-\sM(u(t))t \big) \ .
\ee
Clearly $v(t,x)$ and $u(t,x)$ have same Sobolev norms,
same magnitude,  mass and the momentum, 
 i.e.
$$
\norm{v(t, \cdot)}_{s} = \norm{u(t,\cdot)}_s \ , \quad \forall s \in \R 
$$
and 
$$
|v(t,x)|= |u(t,x-\sM(u(t))t )|, \quad \sM(v(t))=\sM(u(t)), \quad \sP(v(t)) = \sP(u(t)) \ , $$
and one readily checks that $v(t,\cdot)$ fulfills the re-normalized equation 
 \be \label{rinormalizzata}
  \pa_t v   =  - \im |D|^\alpha  v+  |v|^2 v_x  - \sM(v)v_x+ \ii \sP(v)v \  . 
 \ee
 This is the equation that we shall consider from now on, and we will relabel $v\leadsto u$.
Also \eqref{rinormalizzata} is a Hamiltonian PDE with  Hamiltonian function 
 \be
 \widetilde \sH(v):= \sH(v)- \sM(v)\sP(v) \ .
 \ee
 
\begin{remark}
The reason we renormalize  equation  \eqref{eq:main} is that the vector field of  \eqref{rinormalizzata} does not contain   integrable resonant monomials of the form $|u_k|^2 u_\ell e^{\im \ell x}$  with $k \neq \ell$.
Although not strictly necessary, it simplifies the analysis of the resonant part of  \eqref{rinormalizzata} in Lemma \ref{lem:wX}.
\end{remark}

\smallskip

\noindent{\bf Analysis of 4-waves interactions.}
Denote by $ \fR$ the subset of $\fP_4$ (recall \eqref{mom1})  consisting in 4-waves resonant indexes, namely
\begin{equation}
 \label{res1}
 \fR := \left\{
 ( \vec{\jmath} , \vec \sigma) \in \fP_4 
  \colon 
  \quad 
 \sigma_1 |j_1|^\alpha + \sigma_2  |j_2|^\alpha + \sigma_3 |j_3|^\alpha  +\sigma_4 |j_4|^\alpha = 0
   \right\}  \ .
 \end{equation} 
When  $\alpha \in (0,1)$ is irrational, one can expect  the  set $\fR$ to contain only integrable  resonances, namely indexes of the form 
$\big( (k, k, \ell, \ell), (+, -, +, -)  \big)$ with $k, \ell \in \Z$
and their permutations.
For $\alpha$  rational, instead,  nonintegrable resonances do exist in general:  for example, when $\alpha = \frac12$,  one has the  non-integrable  Zakharov-Dyachenko resonances \cite{ZakD}.
We do not care if such non-integrable resonances exist or not, since, as we discussed in the introduction,  our energy cascades will be due to {\em quasi-resonances}, rather than exact resonances. 
What we really are  interested in, is to study the resonances between frequencies in a fixed set $\Lambda$ and those in its complementary set, with  at most two frequencies in $\Lambda^c$.

We shall now study resonant sets with indexes  constrained to belong to certain subsets. 
\begin{definition}Given a set $\Lambda \subseteq \Z$ and $n \in \{0, \ldots, 4\}$,  we  denote by 
$\fP_\Lambda^{(n)}$ the elements of $\cP_4$ (see \eqref{mom1}) having  exactly $n$ indexes {\em outside} the set $\Lambda$: 
\begin{equation}\label{momj}
\fP_\Lambda^{(n)} := \{ ( j_1, j_2, j_3, j_4 , \vec \sigma) \in \fP_4 \colon   \mbox{exactly } n\mbox{ indexes  among  } j_1, j_2, j_3, j_4 \mbox{ are  outside } \Lambda  \} \ . 
\end{equation}
We denote by $\fR_\Lambda^{(n)} $ the subset of $\fP_\Lambda^{(n)}$ made of $4$-waves resonances: with $\tR$ in \eqref{res1}, 
 \begin{equation}\label{resj}
\fR_\Lambda^{(n)} := \{ (  j_1, j_2, j_3, j_4 , \vec \sigma) \in \fR\colon  \mbox{exactly } n\mbox{ indexes  among }  j_1, j_2, j_3, j_4  \mbox{ are  outside } \Lambda  \} \ .
\end{equation} 
\end{definition}
We shall now study in detail the sets $\fR_\Lambda^{(n)}$, $n = 0,1,2$, when $\Lambda$ is given by 
\begin{equation}
  \label{Lambda}
  \Lambda := \{ -1, +1\} \ . 
  \end{equation}

\begin{lemma}\label{lem:wres}
Let $\Lambda$ in \eqref{Lambda} and $\cP^{(n)}_\Lambda$, $\fR^{(n)}_\Lambda$  defined in \eqref{momj} and \eqref{resj}.
\begin{itemize}
\item[(i)] The set $\fP_\Lambda^{(0)} \equiv \fR_\Lambda^{(0)}$ and it contains only integrable resonances: 
\be\label{R0}
\begin{aligned}
\fR_\Lambda^{(0)} & = \left\{ 
\big( \pi (\mathtt{k}, \mathtt{k}, \ell, \ell), \  \pi(+,-, +, -) \big), \colon \mathtt{k}, \ell\in \Lambda  \ , 
\ \  \pi \in \cS_4 \right\} 
\end{aligned}
\ee
and  $\cS_4$ is the symmetric group of  permutations of  four symbols.
\item[(ii)] The set $\fR_\Lambda^{(1)} = \emptyset$. Moreover
$\fP_\Lambda^{(1)}$ has finite cardinality and  there exists $c >0$ such that 
\begin{equation}
\label{lower.R1}
(\vec{\jmath}, \vec \sigma)  \in \fP_\Lambda^{(1)} \quad \Rightarrow
\quad
\abs{ \sigma_1 |j_1|^\alpha + \sigma_2  |j_2|^\alpha + \sigma_3 |j_3|^\alpha  +\sigma_4 |j_4|^\alpha } \geq c \ .  
\end{equation}
\item[(iii)] The set 
\be
\begin{aligned}\label{fR2Lambda}
\fR_\Lambda^{(2)} & = \{ 
\big( \pi (\mathtt{k}, \mathtt{k}, \ell, \ell), \  \pi (+,-, +, -) \big)
 \colon  \ \mathtt k \in \Lambda, \  \ell  \in \Lambda^c , \  \pi \in \cS_4 \}  \ .
\end{aligned}
\ee
 Moreover there exists $c >0$ such that 
\begin{equation}
\label{lower.R2}
(\vec{\jmath}, \vec \sigma) \in \fP_\Lambda^{(2)}\setminus \fR_\Lambda^{(2)} \quad \Rightarrow
\quad
\abs{  \sigma_1 |j_1|^\alpha + \sigma_2  |j_2|^\alpha + \sigma_3 |j_3|^\alpha  +\sigma_4 |j_4|^\alpha  } \geq \frac{c}{\max\limits_{a=1,\ldots, 4} ( |j_a|)^{1-\alpha} } \ .  
\end{equation}
\end{itemize}
\end{lemma}
\begin{proof} 
The gauge condition $ \sum_{a=1}^4 \sigma_a = 0 $ implies that exactly two $\sigma_a$'s are $+$, the other are $-$. 
So, up to permutation, we can always assume that $\sigma_1 = \sigma_3 = 1$ and $\sigma_2  = \sigma_4 = -1$.

$(i)$ In this case all indexes $j_1, j_2, j_3, j_4 \in \Lambda$, so  automatically   $|j_1|^\alpha - |j_2|^\alpha + |j_3|^\alpha-|j_4|^\alpha = 0$, so $\fP^{(0)}_\Lambda = \fR^{(0)}_\Lambda$. 
Next the momentum condition  $j_1-j_2 +j_3-j_4 =0$ gives that either $j_1 = j_2 = \mathtt{k}$, $j_3 =j_4 = \ell$,  yielding $\Big( (\mathtt{k}, \mathtt{k}, \ell, \ell), (+,-,+,-)\Big) $,  or 
 $j_1 = j_4 = \mathtt{k}$, $j_2 = j_3 = \ell$, yielding $\Big( (\mathtt{k},\ell,\ell, \mathtt{k}), (+,-,+,-)\Big) $, which is a permutation of the previous one.
 
 $(ii)$ We can always assume that $j_1, j_2, j_3 \in \Lambda$ and $j_4 \in \Lambda^c$.
 Then the resonant condition $|j_1|^\alpha - |j_2|^\alpha + |j_3|^\alpha - |j_4|^\alpha$ reduces to $|j_3|^\alpha - |j_4|^\alpha$, for which we have the lower bound
 $$
 \abs{|j_3|^\alpha - |j_4|^\alpha} \geq 
 \begin{cases}
 2^\alpha -1 , & \mbox{ if } |j_4| \geq 2  \ , \\
 1 , &\mbox{ if } j_4 = 0
 \end{cases} \  \ \ .
 $$
This proves  both $\fR_\Lambda^{(1)} = \emptyset$ and  \eqref{lower.R1}.

$(iii)$ We have two different cases. 

\noindent \underline{Case I}: W.l.o.g.  assume $j_1, j_3 \in \Lambda$, $j_2, j_4 \in \Lambda^c$. 
The momentum condition reads $j_1 + j_3 = j_2 + j_4$. We examine further subcases.\\
$\bullet$ If $j_2 = j_4 = 0$, then 
$
\abs{ |j_1|^\alpha - |j_2|^\alpha + |j_3|^\alpha - |j_4|^\alpha} = 2$.\\
$\bullet$ If $j_2 = 0$ and $j_4 \neq 0$,  from the momentum condition we get $|j_4| \leq 2$, so actually $j_4 = \pm 2$. Then 
$
\abs{ |j_1|^\alpha - |j_2|^\alpha + |j_3|^\alpha - |j_4|^\alpha}  = 2 - 2^\alpha >0$.\\
$\bullet$ If $j_2, j_4 \neq 0$, then $|j_2|, |j_4| \geq 2$. Then 
$
\abs{ |j_1|^\alpha - |j_2|^\alpha + |j_3|^\alpha - |j_4|^\alpha} 
\geq 
  2(2^\alpha -1) >0$.\\
Hence in Case I there are no resonances and the lower bound \eqref{lower.R2} holds.

\noindent \underline{Case II}: W.l.o.g.  assume that $j_1, j_2 \in \Lambda$, $j_3, j_4 \in \Lambda^c$.
The momentum condition reads $j_1 - j_2 = j_4- j_3$. Again we examine further subcases.\\
$\bullet$ If $j_1=j_2 = \tk \in \Lambda$, then,  by the momentum,  $j_3 = j_4 = \ell \in \Lambda^c$ and they form an element  of $\fR^{(2)}_\Lambda$. 
All other cases in \eqref{fR2Lambda} are obtained by permutations.\\
$\bullet$ If $j_1 \neq j_2$, then $j_4 = j_3 \pm 2$. Consider the ``+'' case, the other being analogous. 
The term $ \abs{ |j_1|^\alpha -  |j_2|^\alpha + |j_3|^\alpha - |j_4|^\alpha }$ reduces to 
\begin{equation}\label{res.case2}
\begin{aligned}
 \abs{  |j_3+2|^\alpha - |j_3|^\alpha } 
& \geq 
\begin{cases}
 2^\alpha & \mbox{ if } j_3 = 0 \mbox{ or } j_3 = -2 \\
4^\alpha -2^\alpha & \mbox{ if } j_3 = 2 \\
\dfrac{c_\alpha}{\max\big(|j_3|, \, |j_3+2|\big)^{1-\alpha}}
& \mbox{ if } |j_3|\geq 3
\end{cases}
\end{aligned}
\end{equation}
proving \eqref{lower.R2}.
\end{proof}

\smallskip
\noindent{\bf Projection of cubic vector fields.}
We introduce now projections of  cubic  vector fields on the sets $\fP_\Lambda^{(n)} $ and $\fR_\Lambda^{(n)}$. 
Recall that any real-to-real cubic  vector field $X(U)$,  translation and gauge invariant,  expand in Fourier as (see \eqref{polvect})
\begin{equation}
\label{X3.fou}
X(U)^\sigma = 
\!\!\!\!\!\!
\sum_{(\vec{\jmath}, k, \vec \sigma, - \sigma) \in \fP_4} \!\!\!\!\!\!
X_{\ j_1, j_2, j_3, k }^{\sigma_1, \sigma_2, \sigma_3, \sigma} \ u_{j_1}^{\sigma_1} \, u_{j_2}^{\sigma_2} \, u_{j_3}^{\sigma_3}\, { e^{\im \sigma k x }}  \ , 
\quad 
X_{\ j_{\pi(1)}, \ldots,j_{\pi(3)}, k}^{ \sigma_{\pi(1)}, \ldots, \sigma_{\pi(3)},\sigma} 
=  
X_{\ j_{1}, \ldots, j_{3}, k}^{ \sigma_{1}, \ldots, \sigma_{3},\sigma} \, 
\end{equation}
 for any permutation $ \pi $ of $ \{1, 2, 3 \} $.
Given a 
subset $A \subseteq \fP_4$, we denote by $\Pi_A X$ the vector field  obtained restricting the indexes to belong to $A$, namely 
\begin{equation}
\label{def:PiA}
(\Pi_A X)(U)^\sigma := 
\!\!\!\!\!\!\!
\sum_{(\vec{\jmath}, k, \vec \sigma, - \sigma)  \in A} 
\!\!\!\!\!\!\!
X_{\ j_1, j_2, j_3, k }^{\sigma_1, \sigma_2, \sigma_3, \sigma} \ u_{j_1}^{\sigma_1} \, u_{j_2}^{\sigma_2} \, u_{j_3}^{\sigma_3}\,  { e^{\im \sigma k x }}   \ . 
\end{equation}
We now compute the projections of the  cubic vector field in \eqref{rinormalizzata}, that we denote by 
\be\label{X3}
X_3(U)^+ := |u|^2 u_x  - \sM(u)u_x+ \ii \sP(u)u \ ,
\ee
on the sets $\fR^{(n)}_\Lambda$  defined in \eqref{resj} for  $n=0,1,2$.

\begin{lemma}\label{lem:wX}
The cubic,  translation and gauge invariant vector field $X_3(U)^+$ in \eqref{X3} fulfills: 
\begin{itemize}
\item[(i)] {\bf Structure:} There exists a $2$-homogeneous $1$-operator $M_{\tt NLS}^+(U) \in \wt \cM_2^1$ such that  $X_3(U)^+ = M_{\tt NLS}^+(U)u$;
\item[(ii)] {\bf Resonances:} The projections of the  vector field $X_3(U)^+$ 
on the sets $\fR_\Lambda^{(n)}$, $n=0,1,2$,   defined in \eqref{resj} are given by
\be\label{proj.X3}
\begin{aligned}
& (\Pi_{\fR_\Lambda^{(0)}} X_3)(U)^+ =  - \im |u_1|^2 u_1 \, e^{\im x}  + \im |u_{-1}|^2 u_{-1} \, e^{-\im x} \  ,  \\
& (\Pi_{\fR_\Lambda^{(1)}} X_3) (U)^+=  0 \ , \quad (\Pi_{\fR_\Lambda^{(2)}} X_3)(U)^+  =  0 \ . 
\end{aligned}
\ee
\end{itemize}
\end{lemma} 
 \begin{proof}
 $(i)$ 
Define  $M_{\mathtt{NLS}}^+(U)$ to be  the operator 
\be\label{MNLS+}
M_{\mathtt{NLS}}^+(U)v := 
\big(|u|^2 - \sM(u) \big) \ \pa_x v  + \im \sP(u) v   \ , 
\ee
so that $M_{\mathtt{NLS}}^+(U)u = X_3(U)^+$.  To prove that
$M_{\mathtt{NLS}}^+(U) \in \wt \cM_2^1$ we write it in  Fourier as 
\be\label{MNLS}
\begin{aligned}
& M_{\mathtt{NLS}}^+(U)v  =  \sum_{\sigma_1 j_1 +\sigma_2 j_2 +j = k \atop \sigma_1 + \sigma_2 = 0 }  M_{j_1, j_2, j, k}^{\sigma_1, \sigma_2}  \, u_{j_1}^{\sigma_1}  u_{j_2}^{\sigma_2}    v_{j}  \, {e^{\im k x}} , \\
&  \quad  M_{j_1, j_2, j, k}^{\sigma_1, \sigma_2 } := 
\begin{cases}
\frac{\im }{2}j  & \mbox{ if } j_1 \neq j_2 , \ j \neq k ,   \ \sigma_1 \neq \sigma_2  \\
- \frac{\im }{2} j_1 &\mbox{ if }  j_1 = j_2 ,   \ j = k  \ , \ \sigma_1 \neq  \sigma_2 \ . \\
0 & \mbox{ otherwise } \\
\end{cases}
\end{aligned}
\ee
The coefficients $ M_{j_1, j_2, j, k}^{\sigma_1, \sigma_2 } $ are symmetric in the first two indexes and fulfill \eqref{smoocara} with $m = 1$ and $\mu = 0$.

$(ii)$ 
As we shall compute the projectors using the definition \eqref{def:PiA},  we need first to write  
 $X_3(U)^+$ in the form \eqref{X3.fou}.
So expand  $X_3(U)^+$ in \eqref{X3} in Fourier, getting 
\begin{align*}
  X_3(U)^+ = & \sum_{j_1 - j_2 + j_3 = k \atop j_1 \neq j_2 }  \im j_3   u_{j_1} \bar u_{j_2} u_{j_3} e^{\im k x}  - \sum_{ j_1 = j_2, \ j_3=k} \im j_2 \, |u_{j_2}|^2 u_{j_3} e^{\im k x} =
  \sum_{(\vec \jmath ,k, \vec \sigma,-)\in \mP_4 }  N_{\vec \jmath ,k}^{\vec\sigma,+}   u_{\vec \jmath}^{\vec \sigma} e^{\im k x},
\end{align*}
where 
$$
N_{j_1,j_2,j_3,k}^{\sigma_1,\sigma_2,\sigma_3,+}:= \ii ( j_3 \delta_{j_1 \neq j_2}- j_2\delta_{j_1=j_2}\delta_{j_3=k})\delta_{(\sigma_1,\sigma_2,\sigma_3)=(+,-,+)} \ .
$$
The  coefficients of expansion \eqref{X3.fou} are obtained  by symmetrization 
$$
X_{j_1,j_2,j_3,k}^{\sigma_1,\sigma_2,\sigma_3,+}= \frac16 \sum_{\pi \in \mS_3} N_{j_{\pi(1)},j_{\pi(2)},j_{\pi(3)},k}^{\sigma_{\pi(1)},\sigma_{\pi(2)},\sigma_{\pi(3)},+}
$$
yielding 
\be\label{X3.coeff}
X_{j_1,j_2,j_3,k}^{+,-,+,+} = 
    \frac{\im}{6} \left( j_3 \delta_{j_1 \neq j_2} + j_1 \delta_{j_3 \neq j_2} - j_2 (\delta_{j_1  = j_2} + \delta_{j_3 = j_2} ) \right) 
\ee

\noindent\underline{Projection on  $\fR_\Lambda^{(0)}$:} 
We use the definition of projections in \eqref{def:PiA}.
In view of the  characterization of $\fR_\Lambda^{(0)}$ given in \eqref{R0}, 
we must consider only those monomials with indexes of the form  $\Big( (k, k, \ell, \ell),  (+,-, +, -) \Big)$ with $k, \ell \in \{ \pm 1\}$ and their permutations.
Once the last couple $(\ell,-)$ is fixed, than either $k=\ell$, giving the index $\big((\ell, \ell, \ell, \ell), (+,-, +,-) \big) $ and its 3 permutations, or  $k=-\ell$, giving
$\big((-\ell, -\ell, \ell, \ell), (+,-, +,-) \big) $ and its 6 permutations.
Therefore we obtain
\begin{align*}
( \Pi_{\fR_\Lambda^{(0)}} X_3)(U)^+  & =
\Big(  3 X_{\ 1, 1, 1, 1}^{+, -, +, +}  \, |u_{1}|^2     u_{1}  
+6 X_{ -1, -1, 1, 1}^{+, -, +, +}  \, |u_{-1}|^2     u_{1} \, \Big) {e^{\im  x} } \\
 & \ \ \   + 
 \Big(  6 X_{\ 1, 1, -1, -1}^{+, -, +, +}  \, |u_{1}|^2     u_{-1}  
+3 X_{ -1, -1, -1, -1}^{+, -, +, +}  \, |u_{-1}|^2     u_{-1} \, \Big) {e^{-\im  x} } \\
 & \stackrel{\eqref{X3.coeff}}{=} 
- \im |u_1|^2 u_1   {e^{\im  x} } + \im |u_{-1}|^2 u_{-1}  {e^{-\im  x} }
\end{align*}
proving the first of \eqref{proj.X3}.\\
\underline{Projection on  $\fR_\Lambda^{(1)}$:} It is zero since $\fR_\Lambda^{(1)} = \emptyset$ by Lemma \ref{lem:wres} $(ii)$. \\
\underline{Projection on  $\fR_\Lambda^{(2)}$:} 
In view of the  characterization of $\fR_\Lambda^{(2)}$ in  \eqref{fR2Lambda},
the   monomials surviving the projection have indexes of the form  $\Big( (k, k, \ell, \ell),  (+,-, +, -) \Big)$ (and their permutations) with only one among $k$ and $\ell$ in $\Lambda$.
Once the last index $(\ell, -)$ is fixed in either $\Lambda$ or $\Lambda^c$, and $k$ is fixed in the complementary set,  there 
are 6 possible  permutations.
Hence we get 
\begin{align*}
( \Pi_{\fR_\Lambda^{(2)}} X_3)(U)^+  
& =
  \sum_{k \in \Lambda^c} 6 X_{k, k, 1, 1}^{+, -, +, +}  \, |u_{k}|^2     u_{1}  
{e^{\im  x} } +  \sum_{k\in \Lambda^c} 6 
 X_{k, k, -1, -1}^{+, -, +, +}  \, |u_{k}|^2     u_{-1}   {e^{-\im  x} }\\
  &  \ \ + 
\sum_{\ell \in \Lambda^c}   \sum_{k = \pm1 } 6 X_{k, k, \ell, \ell}^{+, -, +, +}  \, |u_{k}|^2     u_{\ell}  \, 
{e^{\im \ell x} }  \stackrel{\eqref{X3.coeff}}{=} 0
\end{align*}
proving the last of \eqref{proj.X3}.
 \end{proof}
For later use, we prove a lemma about the  projections  on $\fR^{(n)}_{\Lambda}$, $n=0,1,2$,  of cubic  paradifferential vector fields. Precisely we have
\begin{lemma}\label{lem:proj.para}
Let $a(Z; x, \xi)$ be a 2-homogeneous symbol in $\wt \Gamma_2^m$, $m \in \R$, with zero average and fulfilling $a(\tg_\theta Z; \cdot) = a(Z; \cdot)$ for any $\theta \in \T$, where $\tg_\theta$ in \eqref{tra}. Then 
 \be\label{ident.Op.ab}
\Pi_{\fR^{(n)}_{\Lambda} } \left[ \vOpbw{ a(Z; x, \xi)} Z \right] = 0 \ , \quad n = 0,1,2 \ . 
\ee
\end{lemma}
\begin{proof}
Recalling  \eqref{vecop}, $\left( \vOpbw{ a(Z; x, \xi)} Z \right)^+ =  \Opbw{ a(Z; x, \xi)} z$. 
Using definition  \eqref{BW} specialized to quadratic  symbols  
 fulfilling $a(\tg_\theta Z; \cdot) = a(Z; \cdot)$, $\forall \theta \in \T$, and the comments right below \eqref{sym.gauge}, we get 
\begin{equation}\label{para2}
 \Opbw{ a(Z; x, \xi)} z = \sum_{ j_1-j_2 +j=k} \chi_2 \left( j_1, j_2,\frac{j+k}{2}\right) a_{j_1, j_2}^{+,-}\left(\frac{j+k}{2}\right) z_{j_1} \bar z_{j_2} z_j
{e^{\im k x}} \ .
\end{equation}
The point is that, when projecting on $\Pi_{\fR^{(n)}_{\Lambda} } $, $n = 0,1,2$, either  the cut-off  $\chi_2(\cdot, \cdot)$  or the coefficient  $a_{j_1, j_2}^{+,-}$ vanish. 
Recall that $\chi_2(\xi', \xi) \equiv 0$ whenever $|\xi'| > \frac{\la \xi \ra}{10}$.   \\
\underline{Case $n = 0$:} In this case  $j_1, j_2, j, k \in \Lambda$, and 
{$ 
\chi_2 \left( j_1, j_2,\frac{j+k}{2}\right) = 0$ }for any choice of $j_1, j_2, j, k$.\\
\underline{Case $n = 1$:} By Lemma \ref{lem:wres} $\fR_\Lambda^{(1)} = \emptyset$ and there is nothing to prove.\\
\underline{Case $n = 2$:}  By Lemma \ref{lem:wres} the indexes $j_1, j_2, j, k$ are pairwise equal. \\
Assume first that $j_1 = j_2$,  then $a_{j_1, j_1}^{+,-} = 0$ since $a(Z;\cdot)$ has zero-average in $x$.\\
The case  $j_1 = j  \in \Lambda$ and $j_2 = k \in \Lambda^c$ violates the momentum conservation, as well as $j_1 = j \in \Lambda^c$, $j_2=k \in \Lambda$.\\
In case  $j_1 = k \in \Lambda$ and $j_2 = j \in \Lambda^c$, the cut-off vanishes since
$$
 \chi_2 \left( \pm 1, j,\frac{j  +k}{2}\right)  \equiv  0 \quad \forall k \in \Lambda, \ 
j \in \Lambda^c \ . 
 $$
 Analogously  the case  $j_1 = k  \in \Lambda^c$,  $j_2 = j \in \Lambda$ is ruled out, concluding the proof.
\end{proof}

 \smallskip 
 
\noindent{\bf Identification argument.} We prove an  abstract identification argument in the spirit of \cite{BFP,BFF}. In section \ref{sec:nf} we shall conjugate equation \eqref{rinormalizzata} with an admissible transformation. Without doing explicit computations, we shall a posteriori identify the explicit form of the resonant parts of the conjugated vector field thanks to the following proposition.
 \begin{proposition}[Identification of the resonant normal form]\label{prop.ident}
 Let $\bF(U)$ be 
 a $2$-admissible transformation (see Definition \ref{admtra}) with a gain $\vr \geq 0$. There exist 
 	 $r,s_0>0$ such that, provided  $U(t) \in B_{s_0, \R}(I;r)$ is a solution of the system
 	\begin{equation}\label{U.start}
 	\partial_t U= -\ii \vOmega(D)U+ X_3(U), \qquad 
  \vOmega(D): = \begin{pmatrix}
  |D|^\alpha  & 0 \\
0 & - |D|^\alpha  
\end{pmatrix} 
 	\end{equation}
where
 	\begin{equation}\label{X.form}
 X_3(U)=M_2(U)U  \ , \quad M_2(U) \mbox{ a  matrix of operators in }\widetilde \cM^1_2 \ ,  
\end{equation} 
then the variable 
 	$ 	Z:=\mF(U) = \bF(U)U$  solves 
 	\begin{equation}\label{Z.end}
 	\partial_t Z = -\ii \vOmega(D)Z + \widetilde X_3(Z)+ \widetilde M_{\geq 4}(Z)Z \ .  	
 	\end{equation}
 	Here   $\widetilde M_{\geq 4}(Z)$ is a matrix of non-homogeneous $7$-operators in $\cM^7_{\geq 4}[r]$, whereas $\widetilde X_3(Z)$ is a
  cubic vector field fulfilling  
 	\begin{equation}\label{ab.ident}
 	\Pi_\fA \widetilde X_3= \Pi_\fA  X_3, \quad \text{for any } \fA \subseteq \fR \ 
 	\end{equation}
   where $\fR$ is the $4$-waves resonant set in \eqref{res1}.
 \end{proposition}
\begin{proof}
	Defining 
	$	 X(U):= -\ii \vOmega(D)U+ X_3(U)$, 
the variable $Z$ solves the equation 
$$
\pa_t Z= \mF^* X (Z):= \di_U \mF(U)\left[ X(U)\right]_{|U=\mF^{-1}(Z)},
$$
where to invert the nonlinear map $\cal F$ we used Lemma \ref{loc.inv}.

Next we provide a Taylor expansion of the push-forward vector field  $\mF^*X$.
Using the expansion \eqref{esp:F} for $\cF(U) = \bF(U)U$, we get
\begin{align}
& 	\di_U \mF(U)\left[ X(U)\right]\label{push_forward_U}\\
	= & -\ii \vOmega(D)U+ X_3(U)+ \bF_2(U)[-\ii \vOmega(D)U]+\di_U \bF_2(U)[-\ii \vOmega(D)U ]U+ M_{\geq 4}(U)U\notag 
\end{align}
where, using the structure \eqref{X.form} of $X_3(U)$ 
\begin{align}
 M_{\geq 4 }(U)W:= &-\bF_{\geq 4}(U)\ii \vOmega(D)W+\bF_{\geq 4}(U)M_2(U)W+ \di_U \bF_{\geq 4}(U)[X(U) ]W\notag\\
 &+ \bF_2(U)M_2(U)W+ \di_U \bF_2(U)[X_3(U)]W. \label{Mappa:M}
\end{align}
We prove in Lemma \ref{Lem:mappa_homo} below that  $M_{\geq 4}(U)$ is a matrix of non-homogeneous operators in $ \mM_{\geq 4 }^3[r]$.
Next we compute \eqref{push_forward_U} at 
\be\label{Gdec}
U=\mF^{-1}(Z)\stackrel{\eqref{exp.G}}{=} \bG(Z)Z, \qquad \bG(Z)-\uno  = \underbrace{- \bF_2(Z)+ \bG_{\geq 4}(Z)}_{=:\bG_{\geq 2}(Z)} \in \Sigma\cM^4_{2}[r],
\ee
obtaining 
\begin{equation*}
\mF^*X(Z)= -\ii \vOmega(D)Z+ 
\widetilde X_3(Z)+ \widetilde M_{\geq 4}(Z)Z 
\end{equation*}
where 
\begin{align}\label{wtX3}
& \widetilde X_3(Z)  :=  X_3(Z) +		 \comm{\bF_2(Z)Z}{ - \im \vOmega(D) Z } \\
\notag
 	& \ \ \comm{\bF_2(Z)Z}{ - \im \vOmega(D) Z }:= 
    \ii \vOmega(D)\bF_2(Z)Z+ \bF_2(Z)[-\ii \vOmega(D)Z]+\di_Z \bF_2(Z)[-\ii \vOmega(D)Z ]Z
    \end{align}
   and
\begin{align}
\notag
\widetilde M_{\geq 4}(Z)W=	&-\ii \vOmega(D)\bG_{\geq 4}(Z)W+ \big[M_2(\mF^{-1}(Z))\bG(Z)-M_2(Z)\big]W \\
& -\big[ \bF_2(\mF^{-1}(Z))\ii \vOmega(D)\bG(Z)- \bF_2(Z)\ii \vOmega(D)\big]W\notag  \\
&  -\big[\di_U \bF_2(\mF^{-1}(Z))[\ii \vOmega(D)\mF^{-1}(Z) ]\bG(Z)-\di_Z \bF_2(Z)[\ii \vOmega(D)Z ]\big]W\notag\\
&+M_{\geq 4}(\mF^{-1}(Z))\bG(Z)W\label{Mappa:Mtilde}
\end{align}
We prove in Lemma \ref{Lem:mappa_homo} below that $\widetilde M_{\geq 4}(Z)$ belongs to $ \mM_{\geq 4 }^7[r]$. This concludes the proof of \eqref{Z.end}. To prove \eqref{ab.ident} we note that 
 		$$
 	\comm{\bF_2(Z)Z}{ - \im \vOmega(D) Z }^\sigma  =
\!\!\!\!\!\!\! 	
 	 \sum_{(\vec{\jmath}, k, \vec \sigma, - \sigma) \in \fP_4}  \!\!  \!\!\!\!\!\!
 	- \im\left( \sigma_1 |j_1|^\alpha +\sigma_2 |j_2|^\alpha+\sigma_3 |j_3|^\alpha  - \sigma |k|^\alpha \right) 
 	\tF_{\vec{\jmath}, k}^{\vec \sigma, \sigma} \, z_{\vec{\jmath}}^{\vec \sigma}\, e^{\im \sigma k x } \ ;
 		$$
 		it then follows that, for any set $\fA \subseteq \fR$, one has 
 		$$
 		\Pi_\fA  \comm{\bF_2(Z)Z}{ - \im \vOmega(D) Z }\equiv 0
 		$$
 		which, together with \eqref{wtX3}, implies \eqref{ab.ident}.
\end{proof} 
 \begin{lemma}\label{Lem:mappa_homo}
 	There is $r>0$ such that  $M_{\geq 4}(U)$ defined in \eqref{Mappa:M} is a matrix of $3$-operators  in $ \mM^3_{\geq 4}[r]$ and  $\widetilde M_{\geq 4}(Z)$ defined in \eqref{Mappa:Mtilde} is a matrix of $7$-operators in  $ \mM^7_{\geq 4}[r]$. 
  \end{lemma}
\begin{proof}
We need to show that each term in \eqref{Mappa:M} and \eqref{Mappa:Mtilde} fulfills \eqref{piove} with $p=4$, some $ s_0 \geq 0$ and $m$ equal $3$ or $7$.
This is proved exploiting that each term is a composition of either $m$-operators or differentials of admissible transformations and therefore satisfying 
\eqref{stima.d.adm}. 
As an example, we explicitly show  how to bound  the most difficult terms in \eqref{Mappa:M} and \eqref{Mappa:Mtilde}.
Recall that, by definition of admissible transformations,  $\bF(U)  - \uno$ is a matrix of $2$-operators in $\Sigma \cM_2^{2}[r]$ for some $ r >0$. 

We start from $\di_U \bF_{\geq 4}(U)[X(U)]W$ in \eqref{Mappa:M}.
Using \eqref{stima.d.adm} (with $s\leadsto s-1$ and $m =2$) and that $\| X(U)\|_{s-1}\lesssim \| U\|_s$, we get 
\begin{align*}
	\| \di_U \bF_{\geq 4}(U)[X(U) ]W\|_{s-3}\lesssim &\| U\|_{s_0}^3\| X(U)\|_{s_0}\|W\|_{s-1}+ \| U\|_{s_0}^3\| X(U)\|_{s-1}\|W\|_{s_0}\\
	&+\| U\|_{s_0}^2 \| U\|_{s-1}\| X(U)\|_{s_0}\| W\|_{s_0}\\
			\lesssim &\| U\|_{s_0+1}^4\|W\|_{s}+ \| U\|_{s_0+1}^3\| U\|_{s}\|W\|_{s_0+1}
\end{align*}
proving \eqref{piove} with $s_0\leadsto s_0+1$.

Now we consider the  term in the third line of  \eqref{Mappa:Mtilde}.  Using the trilinearity of $(V, V', W) \mapsto \di_U \bF_2(V)[ V']W$ and \eqref{Gdec} we decompose it as
\begin{align}
\label{alg.0}
&\big[\di_U \bF_2(\mF^{-1}(Z))[\ii \vOmega(D)\mF^{-1}(Z) ]\bG(Z)-\di_Z \bF_2(Z)[\ii \vOmega(D)Z ]\big]W\\
\notag
&=  \di_U \bF_2\left(\bG_{\geq 2}(Z)Z\right)[ \ii \vOmega(D)\mF^{-1}(Z)] \bG(Z)W +\di_U \bF_2(Z)[\ii \vOmega(D)\mF^{-1}(Z) ]\bG_{\geq 2}(Z)W \\
\notag
& \quad +\di_U \bF_2(Z)[\ii \vOmega(D)\bG_{\geq 2}(Z) ]W
\end{align}
 We bound each term in \eqref{alg.0} separately. We shall repeatedly use that $\| \vOmega(D)U\|_{s-\alpha}\leq \| U\|_s$.
First, using  \eqref{tri.est.F2} and then \eqref{smoothing}, \eqref{piove},  \eqref{stima.inv.adm} and \eqref{Gdec}, we get
\begin{align} 
\|\di_U \bF_2&\left(\bG_{\geq 2}(Z)Z\right)[ \ii \vOmega(D)\mF^{-1}(Z)] \bG(Z)W\|_{s-7}\notag\\
\lesssim & \| \bG_{\geq 2}(Z)Z\|_{s_0}\| \vOmega(D)\mF^{-1}(Z)\|_{s_0}\| \bG(Z)W\|_{s-5}+\| \bG_{\geq 2}(Z)Z\|_{s_0}\| \vOmega(D)\mF^{-1}(Z)\|_{s-5}\| \bG(Z)W\|_{s_0}\notag\\
&+\| \bG_{\geq 2}(Z)Z\|_{s-5}\|\vOmega(D)\mF^{-1}(Z) \|_{s_0}\|\bG(Z)W \|_{s_0}\notag\\
\lesssim &\| Z\|_{s_0+4}^4\|W\|_{s}+\| Z\|_{s_0+4}^3\|Z\|_{s}\| W\|_{s_0+4}.\label{mappa.sost1}
\end{align}
Similarly one obtains
\begin{equation}
\|\di_U \bF_2\left(Z\right)[ \ii \vOmega(D)\mF^{-1}(Z)] \bG_{\geq 2}(Z)W\|_{s-7} 
\lesssim \| Z\|_{s_0+4}^4\|W\|_{s}+\| Z\|_{s_0+4}^3\|Z\|_{s}\| W\|_{s_0+4}.\label{mappa.sost2}
\end{equation}
Finally, using \eqref{tri.est.F2} and then  \eqref{smoothing}, \eqref{piove}  and \eqref{Gdec}, we get 
\begin{align} 
\|\di_U \bF_2&\left(Z\right)[ \ii \vOmega(D)\bG_{\geq 2}(Z)Z] W\|_{s-7}\notag\\
\lesssim & \| Z\|_{s_0}\| \vOmega(D)\bG_{\geq 2}(Z)Z\|_{s_0}\| W\|_{s-5}+\| Z\|_{s_0}\| \vOmega(D)\bG_{\geq 2}(Z)Z\|_{s-5}\| W\|_{s_0}\notag\\
&+\| Z\|_{s-5}\|\vOmega(D)\bG_{\geq 2}(Z)Z \|_{s_0}\|W \|_{s_0}\notag\\
\lesssim &\| Z\|_{s_0+5}^4\|W\|_{s}+\| Z\|_{s_0+5}^3\|Z\|_{s}\| W\|_{s_0+5}.\label{mappa.sost3}
\end{align}
Estimates \eqref{mappa.sost1}, \eqref{mappa.sost2} and \eqref{mappa.sost3} prove that the operator in \eqref{alg.0} is a non-homogeneous $7$-operator in $\mM_{\geq 4}^{7}[r]$.
\end{proof}
 
\section{Paradifferential  normal form} \label{sec:nf}
The goal of this section is to use
paradifferential transformations and  Birkhoff normal forms, in the spirit of \cite{BD}, to 
put the quasilinear equation \eqref{rinormalizzata} into a suitable  normal form.
However, the normal form that we shall obtain is rather different from the one of \cite{BD} and of \cite{BFP, BFF, BMM, MMS}; indeed, in these papers, the paradifferential part has symbols with constant coefficients (at least at low homogeneity), and the smoothing vector field is in Birkhoff  normal form, namely supported only on resonant monomials.
On the contrary, our normal form has to 
  two important and different features,  see Theorem \ref{thm:nf}: 
$(i)$ the cubic part of the paradifferential vector field has  a  dominant transport term  with {\em variable coefficients} and supported only on resonant sites, see \eqref{VresZ}, and 
$(ii)$ the  cubic smoothing vector field is in  a suitable weak normal form,  that we call  $\Lambda$-normal form and we now introduce.

\begin{definition}[$\Lambda$-normal form] 
\label{def:wr}
Let $\Lambda = \{1,-1\}$ as in \eqref{Lambda}.  A cubic, translation and gauge invariant vector field $X(U)$ is said to be in 
\begin{itemize}
\item {\em weak-$\Lambda$ normal form} if  all its  monomials with at most two indexes outside $\Lambda$ are resonant, i.e. 
\begin{equation}
\begin{aligned}
& \Pi_{\fP^{(n)}_\Lambda} X = \Pi_{\fR^{(n)}_\Lambda} X  \ , \quad n = 0,1,2 \ ;
\end{aligned}
\end{equation}
\item {\em strong-$\Lambda$ normal form} if  in addition there are no resonant  monomials with one or two indexes outside $\Lambda$, i.e. 
\begin{equation}
\begin{aligned}
& \Pi_{\fP^{(0)}_\Lambda} X = \Pi_{\fR^{(0)}_\Lambda} X  \ ,  \quad 
 \Pi_{\fP^{(1)}_\Lambda} X = \Pi_{\fP^{(2)}_\Lambda} X  = 0 \ ,
\end{aligned}
\end{equation}
\end{itemize}
the sets $\cP^{(n)}_\Lambda$, $\fR^{(n)}_\Lambda$ being defined in \eqref{momj} and \eqref{resj}.
\end{definition}
Note that a cubic vector field in  {\em strong-$\Lambda$ normal form}  is composed by monomials 
$
u_{j_1}^{\sigma_1}  u_{j_2}^{\sigma_2} u_{j_3}^{\sigma_3} e^{\im \sigma k x}  
$
whose indexes $\big( (j_1, j_2, j_3, k), (\sigma_1, \sigma_2, \sigma_3, -\sigma) \big) $  are 
\begin{itemize}
\item either in  $\Lambda$ and resonant, i.e. $\big( (j_1, j_2, j_3, k), (\sigma_1, \sigma_2, \sigma_3, -\sigma) \big) \in  \fR^{(0)}_\Lambda$;
\item or  at least three indexes are  outside $\Lambda$, i.e. $\big( (j_1, j_2, j_3, k), (\sigma_1, \sigma_2, \sigma_3, -\sigma) \big) \in  \fP^{(3)}_\Lambda \cup \fP^{(4)}_\Lambda$.
\end{itemize}

\smallskip 
To start the  normal form procedure, 
it is convenient to write \eqref{rinormalizzata} as the system in the variable $U := \vect{u}{\bar u}$ given by 
\begin{equation}
\label{eq.U}
\pa_t U = - \im \vOmega(D) U + X_3(U) , \quad X_3(U) = 
\lvect{|u|^2   u_x - \sM(u) u_x + \im \sP(u) u}{|u|^2\bar{u_x}  - \sM(u)\bar{u_x}  - \im \sP(u) \bar u}
\end{equation}
where $\vOmega(D)$ is defined in \eqref{U.start} and, with $M_{\mathtt{NLS}}^+$ the $1$-operator in $\wt \cM_2^1$    in  \eqref{MNLS+}, 
\begin{equation}\label{X3.mop}
X_3(U) =M_{\mathtt{NLS}}(U)U  \ , \qquad  M_{\mathtt{NLS}}(U) := \begin{pmatrix}
M_{\mathtt{NLS}}^+(U) & 0 \\
0 & \bar{M_{\mathtt{NLS}}^+(U) } 
\end{pmatrix} \ .
\end{equation}
The first step is to paralinearize such system.
 \begin{lemma}[Paralinearization] \label{para} 
Fix $\vr \geq 0$ and $s_0 > \varrho +\frac32$. If $u(t) \in H^{s_0}(\T, \C)$ solves 
equation  \eqref{rinormalizzata}, then $U(t) = \vect{u(t)}{\bar u(t)}$ solves the system in paradifferential form (recall the notation in \eqref{vecop})
\begin{equation}\label{U.para}
\pa_t U= 
- \im \vOmega(D) 
U + \vOpbw{\im \, \und \tV(U;x)  \xi +\im\,  \und \td(U;x) } U
+ 
\zOpbw{b(U;x)} U
+ R_2(U) U
\end{equation}
where: \\
$\bullet$ $ \vOmega(D) $ is the matrix of Fourier multipliers in \eqref{U.start};
\\
$\bullet$ $\tV(U;x),\, \und \td(U;x),  \in \wt \cF^\R_2 $ and $ b(U;x)  \in \wt \cF_2$  are the zero-average,  2-homogeneous functions
\begin{align}
&\und \tV(U;x) := |u|^2 - \sM(u)  =  \sum_{k_1\not=k_2} u_{k_1}\bar u_{k_2} \, e^{ \im (k_1-k_2) x }, \label{vuoneunder}\\
&\und \td(U;x):= \,\textup{Im}(u_x \bar u) - \sP(u)  = 
\Im \sum_{k_1\not= k_2 } \im \,  k_1 u_{k_1} \bar u_{k_2} e^{ \im (k_1-k_2) x } \ , \\
&b(U;x):= u u_x = \sum_{k_1,k_2\in \Z} \im  \frac{k_1+k_2}{2} u_{k_1}u_{k_2}e^{\im (k_1+k_2)x}, 
\end{align}
where $\sM(u),\, \sP(u)$ are the mass and momentum defined in \eqref{MP};\\
$\bullet$ $R_2(U)$  is a real-to-real, gauge invariant matrix of smoothing operators in $\wt \cR^{-\varrho}_{2}$.
\end{lemma}
\begin{proof}
The nonlinearity $|u|^2 u_x$ is paralinearized in a  standard way using  Lemma \ref{bony} and Proposition  \ref{teoremadicomposizione}, getting  a smoothing remainder $R(U)$ whose coefficients fulfill
\eqref{smoocara} with $\mu \leadsto \varrho +1$ and $m\leadsto  - \varrho$.
Note also that, in view of the Bony quantization \eqref{regula12}, \eqref{BW} for homogeneous symbols 
\begin{align*}
\sM(u) u_x = \Opbw{\sM(u) \im \xi} u + R(U)u  \ , 
\quad
\sP(u) u = \Opbw{\sP(u)} u + R(U)u
\end{align*}
for some smoothing remainders in $\wt \cR^{-\varrho}_{2}$. 
Finally,  remark that equation \eqref{eq.U} is real-to-real and gauge invariant. Since also the paradifferential operators 
in \eqref{U.para} are  real-to-real and gauge invariant (see \eqref{prodotto} and \eqref{sym.gauge}),  by difference so is the  matrix of smoothing operators $R_2(U)$.
\end{proof}

\begin{remark}\label{rem:LWP}
 Exploiting the continuity Theorem \ref{thm:contS} and the symbolic calculus of Proposition \ref{teoremadicomposizione}, one checks easily that a solution of \eqref{U.para} (namely the paralinearization of  \eqref{rinormalizzata}) fulfills the  cubic energy estimate
\begin{equation}\label{esercizio}
     \pa_t \norm{U(t)}_s^2 \lesssim \norm{U(t)}_{s_0}^2 \norm{U(t)}_s^2
\end{equation}
 for any $s > s_0 > \frac32$. It is then standard to deduce local well-posedness in $H^s$, $s > \frac32$, for equation \eqref{U.para} --  see e.g. the scheme in
  \cite[Chapter 7]{Metivier}. Moreover, the energy estimate \eqref{esercizio} shows that initial data of size $0<\delta \ll 1$ gives rise to solution remaining of size $\sim 2 \delta$ for times of order $\delta^{-2}$.
\end{remark}

The main result of the section is the following normal form theorem.

\begin{theorem}\label{thm:nf}
There exist $s_0, r >0$ and a $2$-admissible transformation $\bF(U) \in \cM_{\geq 0}^0[r]$ with gain $3$ (see  Definition \ref{admtra}) such that
 if $U(t) \in B_{s_0, \R}(I;r)$ solves \eqref{U.para} then
the variable 
\be \label{Def:ZU}
 Z:= \mF(U):=\bF(U)U \quad \mbox{ solves }
  \ee    
	 \be \label{Z.eq}
	 \begin{aligned} 
	 	\pa_t Z= &	 	- \im \vOmega(D)Z +
		 \vOpbw{ \im  \langle \,\und \tV\, \rangle(Z;x) \xi +   \im { a}^{(\alpha)  }_2(Z;x,\xi)}Z +  {R}_2^{(\Lambda)}(Z)Z
 \\
		&  + \vOpbw{    \im \wt V_{\geq 4 }(Z;x)  \xi +
\im   \wt a^{(\alpha) }_{\geq 4} (Z;x,\xi)} Z	 + \wt {B}_{\geq 4}(Z) Z
	 \end{aligned}
	 \ee
	 where:\\
$\bullet$ $ \vOmega(D) $ is the matrix of Fourier multipliers in \eqref{U.start};\\
$\bullet$  $ \langle \,\und \tV\, \rangle(Z;x)$ is the  zero-average, real valued   function in $\wt \cF^\R_2$ defined by
		\be \label{VresZ}
		 \langle \,\und \tV\, \rangle(Z;x) := 
		2\,  \Re \Big(\sum_{n \in \N} z_{n} \, \bar{z_{-n}} \, e^{\im 2n x} \Big) \ ;
		\ee 
$\bullet$ ${ a}_2^{(\alpha)}(Z;x,\xi)$ is a  zero average, gauge-invariant, real symbol in $\wt \Gamma_2^{\alpha }$; \\
$\bullet$   $  \wt V_{\geq 4 }(Z;x) $ is a real function in $\mF_{\geq 4}^{\R}[r]$ and 
  $ \wt a^{(\alpha) }_{\geq 4} (Z;x,\xi)$  a real non-homogeneous symbol in $\Gamma^\alpha_{\geq 4}[r]$;\\
$\bullet$ ${R}_2^{(\Lambda)}(Z)$ is a real-to-real and gauge invariant matrix of smoothing operators in $\wt \cR_2^{-4}$ such that the  cubic vector field 
\be\label{Xlambda}
X^{(\Lambda)}(Z):=  {R}_2^{(\Lambda)}(Z)Z 
\ee
 is in strong-$\Lambda$ normal form (see Definition \ref{def:wr}).
Precisely, with the notation in \eqref{def:PiA}, 
\begin{equation}\label{Y.str}
\begin{aligned}
& (\Pi_{\fP^{(0)}_\Lambda} X^{(\Lambda)}) (Z) = 
\begin{pmatrix}
- \im |z_1|^2 z_1 \, e^{\im x}  + \im |z_{-1}|^2 z_{-1} \, e^{-\im x}   \\
 \im |z_1|^2 \bar{z_1} \, e^{-\im x}  - \im |z_{-1}|^2  \bar{z_{-1}} \, e^{\im x}  
\end{pmatrix}
\  ,  \\
& \Pi_{\fP_\Lambda^{(1)}} X^{(\Lambda)}=  \Pi_{\fP_\Lambda^{(2)}} X^{(\Lambda)} =  0 \ .
\end{aligned}
\end{equation}
$\bullet$ Finally $\wt B_{\geq 4}(Z)$ is a real-to-real  matrix of $0$-operators in $ \mM^0_{\geq 4}[r]$.
\end{theorem}

The rest of the section is devoted to the proof of Theorem \ref{thm:nf}.


 \subsection{Block diagonalization}
The goal of this section is to remove the out-diagonal term $\zOpbw{b(U;x)}$ from equation \eqref{U.para} up to quadratic smoothing  operators and quartic bounded operators. Precisely we prove:
\begin{proposition}[Block-diagonalization]\label{propdiag}
Let  $\vr \geq  1-\alpha$. There exist $s_0, r >0$ and a $0$-admissible transformation $\Psi(U) \in \cM_{\geq 0}^0[r]$ with gain 5 (see  Definition \ref{admtra}) such that
if $U(t) \in B_{ s_0, \R}(I;r)$ solves \eqref{U.para}, then 
the variable 
\be \label{psione}
W:= \Psi(\bU)U  \quad \mbox{ solves }
\ee
	\begin{equation} \label{scalare}
		\pa_t W= 
	- \im \vOmega(D) 
	W + \vOpbw{\im \underline{\tV}(\bU;x)  \xi +\im \underline{\td}(\bU;x)} W+ R_2(\bU) W+ B_{\geq 4}(U)W
	\end{equation}
where:\\
$\bullet$ $\vOmega(D)$ is the matrix of Fourier multipliers defined in \eqref{U.start};\\
$\bullet$ 
$\underline{\tV}(\bU;x)$ and $\underline{\td}(\bU;x)$ are the zero average functions defined  in \eqref{vuoneunder};\\
	 $\bullet$  $R_2(\bU)$ is a real-to-real, gauge invariant matrix of homogeneous smoothing remainders in $  \wt{\cR}^{-\vr }_2$;\\
	  $\bullet$  $B_{\geq 4}(\bU)$ is a real-to-real matrix of non-homogeneous bounded operators   in $   \cM_{\geq 4}^0[r]$.
\end{proposition}
\begin{proof}
We define the map  $\Psi(U)$ as the time-1 flow $\Psi(U):= \Psi^\tau(U)_{|\tau=1}$ of the paradifferential equation
\be 
\begin{cases} 
\partial_\tau \Psi^\tau(U)= G(U)\Psi^\tau(U)\\
	\Psi^0(U)=\uno,
\end{cases}\quad \text{where} \quad G(U):= \zOpbw{g_2(U;x,\xi)}
\ee
and with the  2-homogeneous symbol $ g_2$ of the form 
 \be \label{g2}
  g_2(U;x,\xi)=\sum_{j_1,j_2\in \Z}  g_{j_1,j_2}(\xi)u_{j_1}u_{j_2}e^{\ii ( j_1+ j_2)x} \in \wt\Gamma_2^{-\alpha}
 \ee
 to be determined.  
By  Lemma \ref{lem:flow.ad},  $\Psi(U)$ is a $0$-admissible transformation with arbitrary gain, which, to be concrete,  we fix to  $5$.
Moreover, $G$ is gauge invariant (see the bullet of formula \eqref{sym.gauge}), so is $\Psi^\tau$ (Remark \ref{rem:gauge.inv}).
The variable $W=\Psi(U)U$ solves 
\begin{align}
\label{bd.1}
\partial_t W= &\Psi(U) \vOpbw{- \im |\xi|^\alpha +\im \und \tV(U;x)  \xi +\im \und \td(U;x) }\Psi(U)^{-1} W \\
\label{bd.2}
&+ \Psi(U)\left[
\zOpbw{b(U;x)}+R_2(U)\right]\Psi(U)^{-1} W \\
\label{bd.3}
& + \left(\partial_t \Psi(U)\right) \Psi(U)^{-1} W  \ .
\end{align}
We first expand \eqref{bd.1}.
The Lie expansion formula (see e.g. Lemma A.1 of \cite{BFP}) says that for any operator $A(U)$, setting $\mathtt{Ad}_{B}[A] := [B,A]$, one has 
\begin{align}\label{lie1}
\Psi(U)A(U)\Psi(U)^{-1}= A(U)+ \left[G(U),A(U)\right]+ \int_0^1 (1-\tau) 
\Psi^\tau(U) \, \mathtt{Ad}^2_{G(U)}[A(U)] \, (\Psi^\tau(U))^{-1} \, \di \tau  \ .
\end{align}
Applying this formula with $A=\vOpbw{- \im |\xi|^\alpha +\im \und \tV \xi +\im \und \td}$, using   formulas \eqref{commurule} 
we get 
\begin{align}
\eqref{bd.1} 
&=\vOpbw{- \im |\xi|^\alpha +\im \und \tV   \xi +\im \und \td } W \\
&+ \: \zOpbw{ \im (g_2\#_\vr |\xi|^\alpha+ |\xi|^\alpha \#_\vr g_2) }W +
R_2'(U)W +  B_{\geq 4}(U)W
\end{align}
where
$R_2'$ is a matrix of smoothing remainders in $\wt \cR^{-\varrho}_2$ (coming from the first of  \eqref{commurule}), and 
 the operator $B_{\geq 4}$ is given by 
\begin{equation}\label{bd.B}
\begin{aligned}
B_{\geq 4}(U):= &  \, \zOpbw{\im \big(g_2  \#_\vr \und \tV  \xi - \und \tV  \xi \#_\vr g_2\big) - \im \big(g_2 \#_\vr \und \td + \und \td \#_\vr g_2 \big)} 
+ R'(U)
\\
&
+ \int_0^1 (1-\tau) 
\Psi^\tau(U) \, \mathtt{Ad}^2_{G(U)}[\vOpbw{- \im |\xi|^\alpha +\im \und \tV   \xi +\im \und \td }] \, (\Psi^\tau(U))^{-1} \, \di \tau  \ ,
\end{aligned} 
\end{equation}
where $R'$ is a matrix of smoothing operators in $\cR^{-\varrho + (1-\alpha)}_{\geq 4}[r]$.
We claim that
$B_{\geq 4}$ is a non-homogeneous  bounded operator in $\cM^0_{\geq 4}[r]$.
Indeed, since $g_2 \in \wt \Gamma^{-\alpha}_2$,  $\und \tV$ and $\und \td$ belong to $\wt\cF^\R_2$, and $-\varrho +1-\alpha \leq 0$, 
we get  that
both the first line of \eqref{bd.B} and $\mathtt{Ad}^2_{G(U)}[\vOpbw{- \im |\xi|^\alpha +\im \und \tV   \xi +\im \und \td }]$ are matrices of  $0$-operators in 
 $\wt \cM^0_{4}$ and so  in $\cM^0_{\geq 4}[r]$ (use the 
 symbolic calculus of Proposition \ref{teoremadicomposizione} 
and the  bullets after Definition \ref{def.m-op}).
 Finally, being $\Psi^\tau$ an admissible transformation, also the  second line of \eqref{bd.B} is a  matrix of  non-homogeneous $0$-operators in $\cM^0_{\geq 4}[r]$ (see Remark \ref{rem.b.ad}-- $(2)$).

Consider now \eqref{bd.2}. 
Expanding as in \eqref{lie1} one see that the $2$-homogeneous component remains the unchanged, getting  
\begin{equation}
\begin{aligned}
\eqref{bd.2} =  \zOpbw{b(U;x)}W+ R_2(U)W+ B_{\geq 4}(U)W
\end{aligned}
\end{equation}
where $ B_{\geq 4}(U)$ is another  matrix of non-homogeneous $0$-operators in  $\cM^0_{\geq 4}[r]$.

Finally we consider line \eqref{bd.3}. This time we use the Lie expansion (Lemma A.1 of \cite{BFP})
\begin{equation}\label{lie2}
\big(\partial_t \Psi(U)\big) \Psi(U)^{-1}= \partial_t G(U)  
+\int_{0}^{1}
(1- \tau ) 
 {\Psi}^\tau (U)\, \mathtt{Ad}_{G(U)} \,[ \pa_{t}G(U)]({\Psi}^\tau(U))^{-1} \di \tau \ .
\end{equation}
Then,   using that 
$g_2(U) \equiv g_2(U,U)$ is a symmetric function of $U$,  we get 
that $\pa_t G(U)= \zOpbw{\pa_t g_2(U;x,\xi)} = 2\zOpbw{ g_2( \pa_t U, U;x,\xi)}$. Since $U$ solves equation \eqref{eq.U}, we get 
\begin{align*}
\left( \partial_t \Psi(U) \right)  \Psi^{-1}(U)
&=\zOpbw{2 g_2(-\ii \vOmega(D)U,U;x,\xi))}+ B_{\geq 4}(U)
\end{align*}
where, using also \eqref{X3.mop},
\begin{align*}
B_{\geq 4}(U):= & \zOpbw{ 2 g_2(M_{\mathtt{NLS}}(U)U,U;x,\xi)}  \\
& +\int_{0}^{1}
(1- \tau ) 
 {\Psi}^\tau (U)\, \mathtt{Ad}_{G(U)} \,[ 2\zOpbw{ g_2( -\ii \vOmega(D)U +M_{\mathtt{NLS}}(U)U , U;x,\xi)}]({\Psi}^\tau(U))^{-1} \di \tau \, 
\end{align*}
By Lemma \ref{nuovetto}, the fact that $\Psi^\tau$ is an admissible transformation, and the bullets after Definition \ref{def.m-op}, we deduce that 
$B_{\geq 4}$ is a matrix of $(-\alpha)$-operators in $\cM^{-\alpha}_{\geq 4}[r]$.

In conclusion, we get that 
\begin{align}
\partial_t W=&  \vOpbw{- \im |\xi|^\alpha +\im \und \tV(U)  \xi +\im \und \td(U) } W\notag\\
&+  \zOpbw{\ii  \big[(g_2(U)\#_\vr |\xi|^\alpha+ |\xi|^\alpha\#_\vr g_2(U) -2 g_2( \vOmega(D)U,U))\big]+b(U)}W\notag\\
 &+ (R_2(U) + R_2'(U))W+ B_{\geq 4}(U)W 
 \label{quasidiag}
\end{align}
where $B_{\geq 4}(U)$ is a matrix of $0$-operators in  $\cM^{0}_{\geq 4}[r]$.
Then the thesis follows from the following lemma.
\begin{lemma}[The out-diagonal homological equation]\label{Lem:homolo}
Let $\vr > 0$. There exists a symbol $g_2(U;x,\xi)\in \tilde \Gamma_2^{-\alpha}$ of the form \eqref{g2} such that  
\be\label{homolo}
r_2(U; \cdot):= \ii \big[(g_2(U)\#_\vr |\xi|^\alpha + |\xi|^\alpha \#_\vr g_2(U) -2 g_2( \vOmega(D)U,U))\big]+ b(U)\in \wt\Gamma_2^{-\vr}
\ee
and $r_2(U; \cdot)$ fulfills the second of \eqref{sym.gauge}.
\end{lemma}
\begin{proof}
Thanks to symbolic calculus formula \eqref{espansione2} (see also \eqref{poisson}), we have that for any $g\in \widetilde \Gamma_2^m$, $m\in \R$,
$$
\begin{cases}
\mathtt{r}[ g](U) := g(U)\#_\vr |\xi|^\alpha+ |\xi|^\alpha\#_\vr g(U)- 2 g(U)|\xi|^\alpha \in \wt \Gamma_2^{m+\alpha- 2}\\
\mathtt{f}[g](U):= 2 g( \vOmega(D)U,U) \in \widetilde \Gamma_2^{m}
\end{cases}
$$
Moreover if
$g$ fulfills the second of \eqref{sym.gauge}, so do 
  $\mathtt{r}[ g]$ and $\mathtt{f}[ g]$. 
Then the homological equation in \eqref{homolo} reads 
\be \label{homoloproof}
r_2(U)= 2 \ii  g_2(U) |\xi|^\alpha +\ii \mathtt{r} [ g_2](U)- \ii \mathtt{f}[g_2](U)+b(U)\in \widetilde \Gamma_2^{-\vr} \ ,
\ee
which we solve  iteratively  exploiting that $g \mapsto \tr[g]$ and $g \mapsto \tf[g]$ are linear. Namely we put  $g_2:= g^{(1)}+g^{(2)}+\dots+ g^{(p)}$ with 
\begin{align*}
&g^{(1)}(U;x,\xi):= -\frac{b(U;x)}{2 \ii |\xi|^\alpha}\in \widetilde \Gamma_2^{-\alpha},\\
&g^{(2)}(U;x,\xi):= - \frac{\ii \mathtt{r}[g^{(1)}](U;x,\xi)- \ii \mathtt{f}[g^{(1)}](U;x,\xi)}{2\ii |\xi|^\alpha}\in \widetilde \Gamma_2^{-2\alpha}\\
&\footnotesize \vdots\\
\normalsize
&g^{(p)}(U;x,\xi):= - \frac{\ii \mathtt{r}[g^{(p-1)}](U;x,\xi)- \ii \mathtt{f}[g^{(p-1)}](U;x,\xi)}{2\ii |\xi|^\alpha}\in \widetilde \Gamma_2^{-p\alpha}.
\end{align*}
With this choice we have 
$
r_2(U)
= \ii \mathtt{r}[g^{(p)} ](U)- \ii \mathtt{f}[g^{(p)}](U)\in \widetilde \Gamma_2^{-p\alpha}
$ 
which implies the thesis choosing $p>{\vr}/{\alpha}$.
Moreover, since $b$ fulfills the second of \eqref{sym.gauge} (recall \eqref{vuoneunder}), so does $g^{(1)}$, and by construction each $g^{(\ell)}$, $\ell \geq 2$ and the symbol  $r_2(U)$.
In particular $g_2$ has the claimed form in \eqref{g2}.
\end{proof}
Applying Lemma \ref{Lem:homolo}, equation \eqref{quasidiag} becomes
\be\label{bd.4}
\partial_t W=  \vOpbw{- \im |\xi|^\alpha +\im \und \tV(U)  \xi +\im \und \td(U) } W+ (R_2(U) + R_2'(U) + R_2''(U))W+ B_{\geq 4}(U)W
\ee
where  $R_2''(U) = \zOpbw{r_2(U;\cdot)}  \in \widetilde \mR_2^{-\vr}$ is the  paradifferential operator of order $-\varrho$ coming from the symbol in \eqref{homolo}. 
This proves the identity \eqref{scalare}, renaming  $ R_2+R_2'+R_2''  \leadsto   R_2$.

Finally we prove that the matrices of smoothing operators are gauge invariant. Indeed each operator on the right of \eqref{bd.1}--\eqref{bd.3} is gauge invariant (recall Lemma \ref{para}), as well as  the $2$-homogeneous matrix of paradifferential operators in \eqref{bd.4}.
Then, by difference, the $2$-homogeneous smoothing operators $R_2 + R_2' + R_2''$ are gauge invariant as well.

\end{proof}
\subsection{Reduction  of the highest order}
In this section we perform a transformation that reduces the symbol of  the highest order paradifferential operator
$\vOpbw{\underline{\tV}(\bU;x) \im \xi }$ to its resonant normal form. 

\begin{proposition}[Paracomposition]\label{propdiag1}
Let  $\vr \geq 1$. 
There are 
 $  s_0, r  > 0$ and 
  a $2$-admissible transformation $\Phi(U) \in \cM_{\geq 0}^0[r]$  with  gain $5$ (see  Definition \ref{admtra}) such that
if $U(t) \in B_{ s_0, \R}(I;r)$ solves \eqref{U.para}, then 
the variable 
 \be \label{Phione}
 W_1:=\Phi(U)W \stackrel{\eqref{psione}}{=}\Phi(U)\Psi(U)U \quad \mbox{ solves }
 \ee
\begin{equation}\label{teo62}
	\begin{aligned}
		\pa_t W_1 = & - \im \vOmega(D) W 	+	 \vOpbw{   \im  \langle \,\und \tV\, \rangle(U;x) \xi     +   \im V_{\geq 4 }(U;x)  \xi+   \im { a}^{(\alpha)  }_2(U;x,\xi) +
\im   a^{(\alpha) }_{\geq 4} (U;x,\xi) } W_1\\
	& + R_2(\bU)W_1+ B_{\geq 4}(U)W_1
	\end{aligned}
	\end{equation}
	where:
 
 		\noindent $\bullet$ $\vOmega(D)$ is the matrix of Fourier multipliers defined in \eqref{U.start};
  
			\noindent  	$\bullet$ $ \langle \,\und \tV\, \rangle(U;x)$ is the 
resonant part of the function $\und \tV(U;x)$ in \eqref{vuoneunder}, namely the 	
		 zero-average, real valued   function in \eqref{VresZ};
   
 		 	\noindent $\bullet$  $  V_{\geq 4 }(U;x) $ is a real function in $\mF_{\geq 4}^{\R}[r]$;
     
		 	\noindent 	$\bullet$   ${ a}_2^{(\alpha)}(\bU;x,\xi)$ is a zero average,  gauge invariant (fulfills the first of \eqref{sym.gauge}), real symbol in $\wt \Gamma_2^{\alpha }$  
		and $ a^{(\alpha) }_{\geq 4} (U;x,\xi)$  a real non-homogeneous symbol in $\Gamma^\alpha_{\geq 4}[r]$;
  
 	\noindent 	$\bullet$ $R_2(\bU)$ is a real-to-real, gauge invariant matrix of homogeneous smoothing operators in  $\wt \mR^{-\vr  }_{2}$;
  
	 	\noindent $\bullet$  $B_{\geq 4}(U)$ is a real-to-real matrix of $0$-operators in $ \mM^0_{\geq 4}[r]$.
\end{proposition}
\begin{proof}

We define the transformation $\Phi(U)$ as the time-1 flow  
of the paradifferential equation 
\be\label{flusso1para}
\begin{cases} 
	\partial_\tau \Phi^\tau(U)= G(U)\Phi^\tau(U)\\
	\Phi^0(U)=\uno,
\end{cases}\quad \text{where} \quad G(U):= \vOpbw{\frac{\beta_2(U;x)}{1+\tau (\beta_2)_x(U;x)}\ii \xi}
\ee

{and $\beta_2(U,V)$  the real valued,  $2$-symmetric function 
\be\label{beta2.s}
\beta_2(U,V;x):= \sum_{j_1, j_2 \in \Z \atop \sigma_1, \sigma_2 \in {\pm}} \beta_{j_1,j_2}^{\sigma_1,\sigma_2} u_{j_1}^{\sigma_1} \, v_{j_2}^{\sigma_2} \, e^{\im (\sigma_1 j_1 + \sigma_2 j_2)x}
\ee
where
the symmetric coefficients 
$$ \beta_{j_1,j_2}^{\sigma_1,\sigma_2}:=\begin{cases}
    \frac{1}{2\ii (\sigma_1|j_1|^\alpha+\sigma_2 |j_2|^\alpha)}& \sigma_1\sigma_2=-1, \ |j_1| \neq |j_2| \\
    0& \textnormal{otherwise}
\end{cases} $$
fulfill \eqref{homosymbo} with $\mu = 1-\alpha$.
Note that
\be\label{beta2}
 \beta_2(U;x):=\beta_2(U,U;x)= \sum_{|j_1| \neq |j_2| } \frac{1}{\ii (|j_1|^\alpha- |j_2|^\alpha) }u_{j_1}\overline{u_{j_2}}e^{\ii ( j_1- j_2)x}  
\ee
 }
By  Lemma \ref{lem:flow.ad}, $\Phi$ is a $2$-admissible transformation with an arbitrary gain, which again we fix to  $5$.
 Moreover, since  $\beta_2$ fulfills the first of \eqref{sym.gauge},  $G$ as well as  $\Phi^\tau$,  $\tau \in [0,1]$, are gauge invariant (see 
 the bullet of formula \eqref{sym.gauge} and 
 Remark \ref{rem:gauge.inv}).

Recalling \eqref{scalare}, the variable $W_1:=\Phi(U)W$ solves 
\begin{align}\label{diff.1}
	\partial_t W_1= &\Phi(U) \vOpbw{- \im |\xi|^\alpha +\im \und \tV(U)  \xi +\im \und \td(U) }\Phi(U)^{-1} W_1\\
	\label{diff.2}
& 	+ 
	 \left(\partial_t \Phi(U)\right)  \Phi(U)^{-1} W_1 \\
	 \label{diff.3}
	&+ \Phi(U)\left[
	R_2(U) + B_{\geq 4}(U)\right]\Phi(U)^{-1} W_1 \ .
\end{align}
We now compute each term, starting from  \eqref{diff.1}. 
By Proposition \ref{prop:Egorov}--$2$ (with \footnotesize $\vr \leadsto \vr +\alpha$\normalsize) 
 we get 
\begin{align}
\Phi(U)& \vOpbw{- \im |\xi|^\alpha }\Phi(U)^{-1} =  \vOpbw{
- \im |\xi|^\alpha  
+ {\im a^{(\alpha)}_{2}}
 + \im a^{(\alpha) }_{\geq 4}
 } +B_{\geq 4}(U)+  R'_2(U)  
 \end{align}
 where 
 $a^{(\alpha)}_{2} $ is a real, zero average, gauge invariant  symbol in $\wt \Gamma^\alpha_2$, 
  $a^{(\alpha)}_{\geq 4}$ is a real symbol in 
$\Gamma_{\geq 4}^\alpha[r]$,  $B_{\geq 4 }= \vOpbw{ \ii a_{\geq 4}^{(\alpha -2)}}+ R_{\geq 4}$ (see \eqref{eq:conj_fou}) is a real-to-real matrix of $0$-operators in $\mM^0_{\geq 4}[r]$ and finally $R'_2(U)$ is a real-to-real,  gauge invariant matrix of  smoothing operators in $\wt \cR_2^{-\varrho}$.

Then, by Proposition \ref{prop:Egorov}--$1$, we get
 \begin{align}
 \Phi(U)& \vOpbw{\im \und \tV  \xi + \im \, \und \td  }\Phi(U)^{-1} =  \vOpbw{\im \und \tV  \xi  +\im V'_{\geq 4 }  \xi  + \im\, \und \td}  + B_{\geq 4}(U)
 \end{align}
 with $ V'_{\geq 4 } \in \cF^\R_{\geq 4}[r]$ and, thanks to $\varrho \geq 1$,   $B_{\geq 4 }$  a  real-to-real matrix of $0$-operators  in $\mM^0_{\geq 4}[r]$.
 
Next we consider the term in \eqref{diff.2}. We apply Proposition \ref{prop:Egorov}--$4$ and  get 
\begin{align}
	\big( \partial_t \Phi(U) \big) \Phi(U)^{-1}= \vOpbw{2\ii \beta_2(-\ii \vOmega(D)U, U) \xi+ \ii V''_{\geq 4}(U)\xi}+ B_{\geq 4 }(U)
\end{align}
where $V''_{\geq 4} \in \cF^\R_{\geq 4}[r]$ and, using again $\varrho \geq 1$,    $B_{\geq 4 }$  a real-to-real matrix of $0$-operators in $\mM^0_{\geq 4}[r]$.

Finally we consider line \eqref{diff.3}. By 
Proposition \ref{prop:Egorov}--$3$ and Remark \ref{rem.b.ad}-- $(2)$
\be
\eqref{diff.3} 
= R_2(U) + B_{\geq 4}(U)  
\ee
with $R_2(U) $ the same real-to-real,  gauge invariant matrix of  smoothing operators in $\wt\cR_2^{{- \varrho}}$ of Proposition \ref{propdiag} and 
with  $B_{\geq 4 }$  a real-to-real matrix of $0$-operators  in  $\mM^0_{\geq 4}[r]$. 

Altogether we have the expansion 
\begin{align}
\pa_t W_1  =  & \vOpbw{
- \im |\xi|^\alpha  
+ \im \und \tV \xi  +
2\ii \beta_2(-\ii \vOmega(D)U, U) \xi
+ \im a^{(\alpha)}_2 
 }W_{1}  +  ( R_2(U)+ R_2'(U)) W_{1} \\
 & +
 \vOpbw{
 \im V_{\geq 4} \xi 
 + \im  a^{(\alpha) }_{\geq 4}
 }W_{1}
 + B_{\geq 4}(U)W_{1}   \ .
\end{align}
One verifies  that   $\beta_2$  in \eqref{beta2} solves the homological equation
\be\label{hom.tra}
2 \beta_2(-\ii \vOmega(D)U,U;x)+ \und \tV(U;x)= 	\langle \, \und \tV \, \rangle (\bU;x) \ ,
\ee
using the expressions of $\und \tV$ in \eqref{vuoneunder},   $\vOmega(D)$ in \eqref{U.start}, and 
 $\langle \, \und \tV \, \rangle$  in \eqref{VresZ}.
 This proves the expansion in \eqref{teo62}, renaming $R_2 + R_2'\leadsto R_2$; note that we proved that it is gauge invariant being sum of gauge invariant operators. 
\end{proof}

 \subsection{The weak $\Lambda$-normal form}
 In this section we perform a Poincaré  normal form, with the goal of putting the smoothing operator $R_2(U)W_1$ in \eqref{teo62} into  weak-$\Lambda$ normal form (see Definition \eqref{def:wr}).

\begin{proposition}[Weak-$\Lambda$ normal form] \label{birkfinalone} 
Let  $\vr \geq 2-\alpha$.
There are 
 $  s_0, r  > 0$ and 
  a $0$-admissible transformation $\Upsilon(U) \in \cM_{\geq 0}^0[r]$  with gain  $\vr -1+ \alpha $ (see  Definition \ref{admtra}) such that
if $U(t) \in B_{ s_0, \R}(I;r)$ solves \eqref{U.para}, then 
the variable 
\be\label{Zetone}
	Z:= \Upsilon(U)W_1 \stackrel{\eqref{Phione}, \eqref{psione}}{=} \Upsilon(U)\Phi(U)\Psi(U)U \quad \mbox{ solves}
\ee 
	 \begin{align} \notag
	 	\pa_t Z= &	 	- \im \vOmega(D)Z +
		 \vOpbw{ \im  \langle \,\und \tV\, \rangle(U;x) \xi     +   \im V_{\geq 4 }(U;x)  \xi+   \im { a}^{(\alpha)  }_2(U;x,\xi) +
\im   a^{(\alpha) }_{\geq 4} (U;x,\xi)} Z\\
\label{birk}
& +  {R}_2^{(\Lambda)}(\bU)Z 
	 	+ {B}_{\geq 4}(\bU) Z
	 \end{align}
	 where $\langle \,\und \tV\, \rangle$, $V_{\geq 4 }$, ${ a}^{(\alpha)  }_2$ and $a^{(\alpha) }_{\geq 4}$ are the same symbols of  Proposition \ref{propdiag1}, whereas\\
$\bullet$ ${R}_2^{(\Lambda)}(\bU)$ is a real-to-real, gauge invariant matrix of smoothing operators in $\wt \cR_2^{-\varrho}$  such that the  cubic vector field $X^{(\Lambda)}(Z):=  {R}_2^{(\Lambda)}(Z)Z $ is in 
weak-$\Lambda$ normal form, namely it fulfills 
 \be\label{wnf}
 \Pi_{\fP^{(n)}_\Lambda} X^{(\Lambda)} = \Pi_{\fR^{(n)}_\Lambda} X^{(\Lambda)}   \ , \quad n = 0,1,2 \ . 
 \ee
$\bullet$ $B_{\geq 4}(U)$ is a real-to-real matrix of $0$-operators in $ \cM_{\geq 4 }^0[r]$. 
\end{proposition}
 \begin{proof}
 We look for a transformation $\Upsilon(U)$ as the time-1 flow  of the  equation 
\begin{equation}
\pa_\tau \Upsilon^\tau(U) = {\tt Q}_2(U)  \Upsilon^\tau(U)  , \quad \Upsilon^0(U)  = \uno  
\end{equation}
where ${\tt Q}_2$ is a matrix of smoothing operators  in $\wt \mR^{-\vr +1-\alpha}_{2}$ to be determined. 
By Lemma \ref{flow.s.ad}, the map $\Upsilon^\tau$ is a $0$-admissible transformation with gain $\varrho -1+ \alpha$ which is non-negative.
Recalling \eqref{teo62}, the variable $Z :=   \Upsilon(U) W_1 $ fulfills 
\begin{equation}
	\begin{aligned} 
		\pa_t Z= &   \Upsilon(U) \left( - \im \vOmega(D)  \right) \Upsilon(U)^{-1} Z +	\Upsilon(U) \vOpbw{\im \, \tm^{(1)}} \Upsilon(U)^{-1} Z   \\
& +	\Upsilon(U)  \big( R_2(\bU) + B_{\geq 4}(U) \big) \Upsilon(U)^{-1} Z   
		+\big(\pa_t \Upsilon(U) \big) \Upsilon(U)^{-1} Z
	\end{aligned}
	\end{equation}
where we set $\tm^{(1)}:=  \langle \, \und \tV \, \rangle \xi +      a_2^{(\alpha)} +  {V}_{\geq 4 } \xi +  { a^{(\alpha)}_{\geq 4}}\in \Sigma \Gamma^1_{2}[r]$.
By Proposition \ref{prop:Egorov_smoothing} (with $\vr\leadsto \vr-(1-\alpha)$) we get 
\begin{equation}\label{teo631}
	\begin{aligned}
		\pa_t Z= &    - \im \vOmega(D)Z  +	\vOpbw{\im \tm^{(1)}} Z\\
		&
		+2 \tQ_2\left(- \im \vOmega(D) U, U\right)Z  + [ \tQ_2(U), - \im \vOmega(D) ] Z +  R_2(\bU)Z \\
		& +	B_{\geq 4}(U) Z+ R_{\geq 4}(U)Z
	\end{aligned}
	\end{equation}
	where $ B_{\geq 4}(\bU)$ is a  real-to-real matrix of $0$-operators in $\cM^0_{\geq 4}[r]$ and $R_{\geq 4}(U)$ is a real-to-real matrix of smoothing operators in $\mR^{-\vr+2-\alpha}_{\geq 4}[r]$ 
	which we shall regard as a $0$-operator in $ \mM_{\geq 4}^0[r]$ since $\vr\geq 2-\alpha$. 

 To determine $\tQ_2(U)$,   expand the  vector field $R_2(U) Z$ in \eqref{teo62} in Fourier components as 
 $$
( R_2(U) Z)^\sigma_k = \sum_{ \mP_4
} R_{j_1, j_2, j, k}^{\sigma_1, \sigma_2, \sigma', \sigma}
u_{j_1}^{\sigma_1} \, u_{j_2}^{\sigma_2} z_{j}^{\sigma'}  $$
 where with the sum over $\mP_4$ we mean that the indexes $(j_1, j_2, j, k,\sigma_1, \sigma_2, \sigma', -\sigma)$ belong to $\mP_4$.  Below we use the same notation. 
 Note that this writing is possible since  $R_2(U)$ is gauge invariant. 
 
Then we define 
 $$(R^{(\Lambda)}_2(U)Z)^\sigma_k:=  \sum_{\mP_4
}  \Lambda_{j_1, j_2, j, k}^{\sigma_1, \sigma_2, \sigma', \sigma}
u_{j_1}^{\sigma_1} \, u_{j_2}^{\sigma_2} z_{j}^{\sigma'}, \quad \Lambda_{j_1, j_2, j, k}^{\sigma_1, \sigma_2, \sigma', \sigma}:=
    R_{j_1, j_2, j, k}^{\sigma_1, \sigma_2, \sigma', \sigma} \delta\Big( (j_1, j_2, j, k,\sigma_1, \sigma_2, \sigma', -\sigma)\in \mC\Big),
   $$
where $\mC:=\bigcup\limits_{n =0}^2 \fR^{(n)}_\Lambda \cup  \bigcup\limits_{n=3}^4 \fP^{(n)}_\Lambda $, {and the sets $\fP^{(n)}_\Lambda$ , $\fR^{(n)}_\Lambda$ defined in \eqref{momj}, \eqref{resj}}. 
We choose $\tQ_2(U)$ so that
\be\label{R2w}
2\tQ_2\left(- \im \vOmega(D) U,U\right)  + [ \tQ_2(U), - \im \vOmega(D) ]  +  R_2(\bU)  = R^{(\Lambda)}_2(U).
\ee
We claim that one can put, denoting $\vec{\jmath}=(j_1, j_2)$, $\vec \sigma=(\sigma_1, \sigma_2)$, 
 \be\label{tQ2.def}
( \tQ_2(U)Z )^\sigma_k  := \sum_{ \cP_4} \tQ_{\vec{\jmath} , j,   k}^{\vec \sigma, \sigma',  \sigma} \, u_{j_1}^{\sigma_1} \, u_{j_2}^{\sigma_2} z_{j}^{\sigma'} 
\ee
where
\be\label{tQ2.def2}
 \tQ_{\vec{\jmath}, j,   k}^{\vec \sigma, \sigma', \sigma}
 := 
 \begin{cases}
  \dfrac{R_{\vec{\jmath}, j,  k}^{\vec \sigma, \sigma',  \sigma}}{\im (\sigma_1 |j_1|^\alpha +\sigma_2 |j_2|^\alpha + {\sigma'} |j|^\alpha - \sigma |k|^\alpha   )}  \ , & (\vec{\jmath},j,  k, \vec\sigma , \sigma',  -\sigma) \in \bigcup\limits_{n =1}^2 \left( \fP^{(n)}_\Lambda \setminus \fR^{(n)}_\Lambda\right) \\
  0  \ , & (\vec{\jmath}, j,  k, \vec \sigma, \sigma', -\sigma) \in \mC
 \end{cases}
 \ee

 \begin{lemma}
  $\tQ_2(U)$ in \eqref{tQ2.def}--\eqref{tQ2.def2} is a matrix of smoothing operators in  
 $\wt \cR^{-\vr + 1-\alpha}_2 $ 
 fulfilling \eqref{R2w}.
 \end{lemma}
 \begin{proof}
 As $R_2(U)$ is a smoothing operator in $\wt \cR^{-\varrho}_2$, its coefficients fulfill the estimate: for some $\mu\geq 0$, $C>0$, 
\begin{equation}\label{boundR}
\abs{ R_{\vec{\jmath}, j,  k}^{\vec \sigma, \sigma',  \sigma} } \leq C 
\frac{
{\rm max}_2\{ \la j_1\ra, \la j_2 \ra,  \la j\ra \}^{\mu}}{ \max\{ \la j_1\ra, \la j_2 \ra, \la j\ra \}^{\varrho}}
   \, , \quad \forall (\vec{\jmath}, j, k, \vec \sigma, \sigma',  -\sigma) \in  \fP_4 \ ,
\end{equation}
 and satisfy the symmetric and reality properties  \eqref{M.coeff.p} and \eqref{M.realtoreal}.
 
 Consider now the coefficients $ \tQ_{\vec{\jmath}, j,   k}^{\vec \sigma, \sigma', \sigma}$ in \eqref{tQ2.def2}. 
 Clearly they satisfy the symmetric and reality properties \eqref{M.coeff.p} and \eqref{M.realtoreal}. 
 We now bound them. By 
\eqref{boundR},   Lemma \ref{lem:wres}  and the momentum relation $\sigma k= \sigma_1 j_1 + \sigma_2 j_2 + \sigma' j$, 
 $$
 \abs{  \tQ_{\vec{\jmath}, j,   k}^{\vec \sigma, \sigma', \sigma}} \leq C \frac{
{\rm max}_2\{ \la j_1\ra,\la j_2 \ra, \la j_3\ra \}^{\mu}}{ \max\{ \la j_1\ra,\la j_2 \ra, \la j_3\ra \}^{\varrho - (1-\alpha)}}\qquad  
\forall  (\vec{\jmath} , j,  k, \vec \sigma, \sigma', -\sigma) \in 
\fP^{(1)}_{\Lambda} \cup (\fP^{(2)}_\Lambda \setminus \fR^{(2)}_\Lambda ) \ ,
 $$
 (recall that  $\fR^{(1)}_\Lambda = \emptyset$).
This shows that  $\tQ_2(U)$ is a matrix of smoothing operators in  $\wt\cR_2^{-\varrho + 1-\alpha}$.

It is clear that  $\tQ_2(U)$ fulfills \eqref{R2w}, also noting that  $\Pi_{\fP^{(0)}_\Lambda}(R_2(Z)Z) =\Pi_{\fR^{(0)}_\Lambda}(R_2(Z)Z) $  in view of Lemma \ref{lem:wres} $(i)$.
 \end{proof}
 
\noindent  With such  $\tQ_2(U)$, system \eqref{teo631} reduces to \eqref{birk}.

We prove now that 
the vector field  $X^{(\Lambda)}(Z)=  {R}_2^{(\Lambda)}(Z)Z $ is in weak-$\Lambda$ normal form, i.e. it fulfills 
  \eqref{wnf}. Indeed the coefficients of the vector field $X^{(\Lambda)}$ are obtained as in \eqref{simmetrizzata} and, being the set $ \mC$ symmetric with respect to the first three indexes, they have the form 
  $$
  X_{j_1, j_2, j_3, k}^{\sigma_1, \sigma_2, \sigma_3, \sigma}= \frac13 \Big(  R_{j_1, j_2, j_3, k}^{\sigma_1, \sigma_2, \sigma_3, \sigma} + R_{j_3, j_2, j_1, k}^{\sigma_3, \sigma_2, \sigma_1, \sigma}+ R_{j_1, j_3, j_2, k}^{\sigma_1, \sigma_3, \sigma_2, \sigma} \Big) \delta\Big( (j_1, j_2, j_3, k,\sigma_1, \sigma_2, \sigma_3, -\sigma) \in \mC\Big) \ . 
  $$
Proposition \ref{birkfinalone} is proved.
 \end{proof}

 \subsection{Identification and proof of Theorem \ref{thm:nf}}
With the aid of paradifferential normal form,  we have conjugated the original system 
\eqref{eq.U} to the new  system  \eqref{birk}. 
The next steps are:  
$(i)$ to  write  \eqref{birk} as a system in the single variable $Z(t)$, and 
 $(ii)$  to compute explicitly 
$\Pi_{\fP^{(n)}_\Lambda}X^{(\Lambda)}$ in \eqref{wnf} for $n = 0,1,2$,   deducing \eqref{Y.str}.

To achieve $(i)$, recall that the map 
in  \eqref{Zetone} has the form 
 \be\label{Z=FU}
Z =  \cF(U)= \bF(U)U , \quad \bF(U):= \Upsilon(U) \Phi(U) \Psi(U)   \ . 
 \ee
{ Since $\Upsilon(U) $ is $0$-admissible with gain $\varrho - 1 + \alpha$, $\Phi(U) $ is $2$-admissible with gain $5$ and $\Psi(U) $ is $0$-admissible with gain $5$ (Propositions \ref{propdiag}, \ref{propdiag1}, \ref{birkfinalone}), by Lemma \ref{lem:comp}
the map $\bF(U)$ a $2$-admissible  with gain  $\min(3, \vr - 1 + \alpha) = 3$ provided $\vr \geq 4- \alpha$.}

Then Lemma \ref{loc.inv} ensures that $\cF$ is locally invertible in a small ball $B_{s_0'}(r')$ for some $s_0', r'>0$, with  inverse map $\cF^{-1}$  having the  structure 
 \be\label{inv.F}
 U = \cF^{-1}(Z) = \bG(Z) Z , \quad  \mbox{ with }  \bG(Z)   =  \uno + \bG_{\geq 2}(Z) , \ \ \bG_{\geq 2}(Z)  \in \Sigma\cM^{4}_{2}[r'] \ , 
 \ee
 for some $r' >0$.
 We then substitute $U$ in the internal variables of the operators in \eqref{birk}.  Consider first the 2-homogeneous operators. We have,  using Lemma \ref{nuovetto}--1, 
 \begin{align*}
  \langle \,\und \tV\, \rangle(\cF^{-1}(Z);x) \xi    -    \langle \,\und \tV\, \rangle(Z;x) \xi 
  \in \Gamma^{1}_{\geq 4}[r'] 
  ,  \quad   
    {a}^{(\alpha)}_2(\cF^{-1}(Z);x,\xi)     -  {a}^{(\alpha)}_2(Z;x,\xi)  \in \Gamma^{\alpha}_{\geq 4}[r'] 
 \end{align*}
 and, using Lemma \ref{nuovetto}--2, 
 $ {R}_2^{(\Lambda)}(\cF^{-1}(Z)) - {R}_2^{(\Lambda)}(Z) \in \cR_{\geq 4}^{- \varrho+4}[r']$.
 Then we substitute $U=\cF^{-1}(Z)$ in the non-homogeneous operators $ \vOpbw{   \im V_{\geq 4 }(U;x)  \xi+ \im   a^{(\alpha) }_{\geq 4} (U;x,\xi)}$ and $B_{\geq 4}(U)$, applying Lemma \ref{nuovetto}--1\& 5.
 In conclusion, setting $\vr:=4$,  we obtain the following:
 \begin{proposition} \label{prop.megliodemax}
There are 
 $  s_0, r  > 0$ such that
if $U(t) \in B_{ s_0, \R}(I;r)$ solves \eqref{U.para}, then 
the variable  $Z(t)$ in \eqref{Z=FU} solves  the system 
  \be \label{birk2}
	 \begin{aligned} 
	 	\pa_t Z= &	 	- \im \vOmega(D)Z +
		 \vOpbw{ \im  \langle \,\und \tV\, \rangle(Z;x) \xi +   \im { a}^{(\alpha)  }_2(Z;x,\xi)} +  {X}^{(\Lambda)}(Z)
 \\
		&  + \vOpbw{    \im \wt V_{\geq 4 }(Z;x)  \xi +
\im   \wt a^{(\alpha) }_{\geq 4} (Z;x,\xi)} Z	 + \wt {B}_{\geq 4}(Z) Z
	 \end{aligned}
	 \ee
	 where  $ \langle \,\und \tV\, \rangle$ and  ${ a}^{(\alpha)  }_2$ 
	 are the quadratic symbols  in Proposition \ref{propdiag1}, ${X}^{(\Lambda)}(Z)$ is the cubic vector field  in weak-$\Lambda$ normal form 
	  of Proposition \ref{birkfinalone} is , whereas
\begin{itemize}
\item  $  \wt  V_{\geq 4 }(Z;x) $ is a real function in $\mF_{\geq 4}^{\R}[r]$;
		\item     $ \wt a^{(\alpha) }_{\geq 4} (Z;x,\xi)$ is a real non-homogeneous symbol in $\Gamma^\alpha_{\geq 4}[r]$;
 \item $\wt B_{\geq 4}(Z)$ is a real-to-real matrix of $0$-operators in $ \cM_{\geq 4 }^0[r]$.
\end{itemize}
 \end{proposition}
 The next step 
 $(ii)$   is   to compute explicitly 
$\Pi_{\fP^{(n)}_\Lambda}{X}^{(\Lambda)}$, $ n = 0,1,2$: 
 \begin{proposition}\label{prop.megliodemax2}
 The vector field ${X}^{(\Lambda)}(Z) $ of Proposition \ref{birkfinalone} is actually in strong-$\Lambda$ normal form (Definition \ref{def:wr}) and fulfills   \eqref{Y.str}. 
 \end{proposition}
 \begin{proof}
 We combine the abstract identification argument of
 Proposition  \ref{prop.ident}  with the  characterization of  the resonant monomials of the original vector field $X_3$ in Lemma \ref{lem:wX}.
 
 Precisely, we apply the identification result of Proposition  \ref{prop.ident} to the starting NLS equation  \eqref{eq.U}  (which  has the required structure in \eqref{X.form} in view of \eqref{X3.mop})
and with the admissible transformation $\bF(U)$ in  \eqref{Z=FU}, getting that $Z$ fulfills an equation of the form 
\eqref{Z.end}. 
Identifying the cubic vector field of \eqref{Z.end} with the one of
   \eqref{birk2}  we get the identity
$$
 \vOpbw{ \im  \langle \,\und \tV\, \rangle(Z;x) \xi +   \im { a}^{(\alpha)  }_2(Z;x,\xi)} +  {X}^{(\Lambda)}(Z)  = \widetilde X_3(Z) \ . 
 $$
In addition,   in view of \eqref{ab.ident}, we have 
 \be\label{ident.X3}
\Pi_{\fR^{(n)}_{\Lambda} }\left( \vOpbw{ \im  \langle \,\und \tV\, \rangle(Z;x) \xi +   \im { a}^{(\alpha)  }_2(Z;x,\xi)} Z+  {X}^{(\Lambda)} \right)=\Pi_{\fR^{(n)}_{\Lambda} } X_ 3   \ , \quad n = 0,1,2 \ . 
\ee
Now we apply  Lemma \ref{lem:proj.para} to the cubic vector field 
$ \vOpbw{ \im  \langle \,\und \tV \rangle\,  \xi +   \im { a}^{(\alpha)}_2} Z$;  this can be done since the symbols 
 $\langle \,\und \tV\, \rangle(Z;x) \xi $ and ${ a}^{(\alpha)  }_2(Z;x,\xi)$ have both  zero-average (Proposition \ref{propdiag1}) and are gauge invariant (i.e. fulfills the first of \eqref{sym.gauge}).
  We conclude that 
\be\label{ident.Op}
\Pi_{\fR^{(n)}_{\Lambda} } \left[ \vOpbw{ \im  \langle \,\und \tV\, \rangle(Z;x) \xi +   \im { a}^{(\alpha)  }_2(Z;x,\xi)} Z \right] = 0 \ ,  \quad n = 0,1,2 \ , 
\ee
from which we get immediately 
\begin{align*}
\Pi_{\fP^{(n)}_{\Lambda} }{X}^{(\Lambda)}  
& \stackrel{\eqref{wnf}}{=} 
\Pi_{\fR^{(n)}_{\Lambda} }{X}^{(\Lambda)} 
\stackrel{\eqref{ident.X3},  \eqref{ident.Op}}{=}  \Pi_{\fR^{(n)}_{\Lambda}} {X}_3   \ ,  \quad n = 0,1,2 \ . 
\end{align*}
 This last vector field is computed in Lemma \ref{lem:wX}, proving   \eqref{Y.str}.

\end{proof}

\begin{proof}[Proof of Theorem \ref{thm:nf}]
It follows from Proposition \ref{prop.megliodemax} and \ref{prop.megliodemax2}. 
\end{proof}
  
 \section{The effective equation}\label{sec:effective}
The goal of this section is to  study the long-time dynamics of solutions of equation \eqref{Z.eq}  fulfilling certain upper-bounds, that we  call {\em long-time controlled}, see Definition \ref{A}. 
In view of the reality of  system \eqref{Z.eq}, we regard it as a scalar equation in $z(t)$.
We study separately the dynamics of the modes supported on $\Lambda$, namely  $z_{\pm 1}(t)$, and those supported on $\Lambda^c$. Specifically we decompose 
\be
\begin{aligned}\label{ztp}
z(t) = z^\top(t) + z^\perp(t) \ , \quad
&  z^\top(t) := z_1(t) \, e^{\im x} + z_{-1}(t) \, e^{-\im x} \ , 
\quad   z^\perp(t):=   \sum_{|j| \neq 1} z_j(t) \, e^{\im j x} \ . 
\end{aligned}
\ee

$\bullet$ {\bf Parameters:} From now on we {\bf fix} $\mathfrak{s}_0, \fr>0$ as follows: $\fs_0:=\max\{ s_0, s_0'\}$ and $\fr:=\min\{r,r'\}$  where $s_0, r >0 $ are given in Theorem \ref{thm:nf} whereas $s_0', r'>0$ are the parameters required to invert the map $ \mF$ in \eqref{Def:ZU}, see  \eqref{inv.F}.  We also fix 
\be\label{parameters}
s> 3 \fs_0 , \quad \theta \in (0, \theta_*), \ \ \theta_* := \min\left(\frac{s-3 \fs_0}{2s - \fs_0},  \frac15 \right) \ .
\ee
The first step is the following one:
\begin{lemma}\label{lem:sys}
If $Z(t)=\vect{z(t)}{\bar z(t)} \in B_{\fs_0, \R}(I;\fr)$ solves \eqref{Z.eq}, then
 the variables
 $\big(z^\top(t), z^\perp(t)\big)$  defined in \eqref{ztp} fulfill  the system
\begin{align}
\label{eq.ztop}
 \pa_t z^\top =  &- \im |D|^\alpha  z^\top  + 
Y_3^{(\Lambda)}(z^\top) +  
 Y^\top_3 (z) + Y^\top_{\geq 5}(z) \\
\label{eq.zperp}
 \pa_t z^\perp  =  & - \im |D|^\alpha  z^\perp  
+ \Opbw{  \im \,  \tm(z; x, \xi) }z^\perp +Y^\perp_3(z) 
+  Y^\perp_{\geq 5}(z)
\end{align}
where\\
$\bullet$ $Y_3^{(\Lambda)}(z)$ is the integrable vector field
\begin{equation}\label{Y.int}
Y^{(\Lambda)}_3 (z) := Y^{(\Lambda)}_3 (z^\top) =  - \im |z_1|^2 z_1 \, e^{\im x}  + \im |z_{-1}|^2 z_{-1} \, e^{-\im x}    \ ;
\end{equation}
$\bullet$ $Y^\top_3(z)$ and $Y^\perp_3(z)$ 
are  cubic smoothing vector fields fulfilling: for any $\ts \geq \fs_0$
\begin{equation}\label{Y.perp}
\norm{Y^\top_3(z)}_{\ts} \lesssim   \norm{z^\perp}_{\fs_0}^3 \,   \ , \qquad 
\norm{Y^\perp_3(z)}_{\ts + 4} \lesssim  \left( \norm{z^\top}_{\fs_0} + \norm{z^\perp}_{\fs_0} \right) \,  \norm{z^\perp}_{\fs_0} \, \norm{z^\perp}_\ts \ ;
\end{equation}
$\bullet$ $\tm(z;x, \xi)$ is the symbol in $\Sigma\Gamma^1_{\geq 2}[\fr]$ given by
\be\label{tm}
\tm(z;x, \xi):=  \langle \,\und \tV\, \rangle(Z;x) \xi +    {a}^{(\alpha)}_2(Z;x,\xi) +  
 \wt V_{\geq 4 }(Z;x)  \xi +   \wt a^{(\alpha) }_{\geq 4} (Z;x,\xi) \ ,
\ee
with $ \langle \,\und \tV\, \rangle(Z;x)$ defined in \eqref{VresZ}.\\
$\bullet$
$Y^\top_{\geq 5}(z)$ and $Y^\perp_{\geq 5}(z)$ are non-homogeneous vector fields fulfilling the estimate:  for any $\ts\geq \fs_0$ there are  $C>0$, $\tr:=\tr(\ts) \in (0,\fr)$ and for any  $z\in B_{\fs_0}(\tr)\cap {H}^{\ts}(\T,\C)$, 
\be\label{Y5}
\norm{Y^\top_{\geq 5}(z)}_{\ts} + \norm{Y^\perp_{\geq 5}(z)}_{\ts} \leq C  \norm{z}^4_{\fs_0} \norm{z}_\ts  \ . 
\ee
\end{lemma}
\begin{proof}
We introduce the projectors
$$
\Pi^\top z := \sum_{j = \pm 1} z_j \, e^{\im j x}  \ , \quad \Pi^\perp z := \sum_{j \neq  \pm 1} z_j \, e^{\im j x} \ 
$$
and compute the projections of the first component of each term in system \eqref{Z.eq}. Since $(-\im \vOmega(D))^+ =  -\im |D|^\alpha$ is a Fourier multiplier, it commutes with the projectors. 
So consider  the paradifferential vector field $\big(\vOpbw{\im \tm } Z \big)^+  =  \Opbw{\im \, \tm } z$.  We  decompose
\begin{align}
\Opbw{\im \, \tm }  & = 
\Pi^\top\Opbw{\im \, \tm } \Pi^\top  + \Pi^\top\Opbw{\im \, \tm } \Pi^\perp  + \Pi^\perp \Opbw{\im \, \tm } \Pi^\top  + \Pi^\perp\Opbw{\im \, \tm } \Pi^\perp  \ .
\end{align}
Writing  
$
\tm_2(z; x, \xi):=  \langle \,\und \tV\, \rangle(Z;x) \xi +    {a}^{(\alpha)}_2(Z;x,\xi)$, $\tm_{\geq 4}(z; x, \xi):=  \wt V_{\geq 4 }(Z;x)  \xi +   \wt a^{(\alpha) }_{\geq 4} (Z;x,\xi)$, 
we 
claim that 
\begin{align}
\label{proj.op}
& \Pi^\top \Opbw{\im \, \tm } \Pi^\top  = \Pi^\top \Opbw{\im \, \tm_{\geq 4} } \Pi^\top \ , \\
\label{proj.op2}
&  \Pi^\top\Opbw{\im \, \tm } \Pi^\perp = \Pi^\perp \Opbw{\im \, \tm } \Pi^\top = 0 \ , \\
\label{proj.op3}
& \Pi^\perp\Opbw{\im \, \tm } \Pi^\perp  = \Opbw{\im \, \tm } \Pi^\perp  \ .
\end{align}
\underline{Proof of \eqref{proj.op}}. We shall exploit that the symbol $\tm_2(z; x, \xi)$ has zero average in $x$ (see Theorem \ref{thm:nf}).
Using the definition  \eqref{BW} for $2$-homogeneous paradifferential operators applied to the  quadratic, gauge invariant, zero-average symbol $\tm_2(z; \cdot)$ we get
 \be\label{proj.op00}
\begin{aligned}
\Pi^\top  \Opbw{\im \, \tm_2 (z; x, \xi)} \Pi^\top  z = & 
\!\!\!\!\!\! \sum_{ j_1-j_2 +j=k  \atop j_1 \neq j_2, \ j, k \in \Lambda } \!\!\!  \chi_2 \left( j_1, j_2,\frac{j+k}{2}\right)  \im \, \tm_{j_1, j_2}^{+,-}\left(\frac{j+k}{2}\right) z_{j_1} \bar z_{j_2} z_j
\, {e^{\im k x}}  \ .
\end{aligned}
\ee
We  show that the cut-off is always vanishing.
Indeed, recalling that   $\chi_2(\xi', \xi) \equiv 0$ when $|\xi'| \equiv \max(|\xi'_1|, |\xi'_2|) \geq \langle \xi \rangle/10$, and using 
$\max(|j_1|, |j_2|) \geq 1$ (as $j_1, j_2$ cannot be both $0$), $ j = k-j_1 +j_2$ and $k \in \Lambda=\{\pm 1\}$, one has 
\be\label{proj.op01}
\frac{1}{10}{\langle \frac{j_1 - j_2 \pm 2}{2} \rangle} 
= \frac{1}{10}\big(1+  \frac{|j_1 - j_2 \pm 2|}{2} \big) 
\leq \frac{4 + 2 \max(|j_1|, |j_2|)}{20}  \leq \frac{3 \max(|j_1|, |j_2|)}{10} \leq \max(|j_1|, |j_2|)  \ , 
\ee
proving that 
$\chi_2 \left( j_1, j_2,\frac{j+k}{2}\right) \equiv 0$.
Consequently $\Pi^\top  \Opbw{\im \, \tm_2 } \Pi^\top  =0$ and \eqref{proj.op} follows.

\noindent \underline{Proof of \eqref{proj.op2}}. 
Again we write   explicitly the action of $\Pi^\top\Opbw{\im \, \tm } \Pi^\perp $, using the quantization \eqref{BW} for the $2$-homogeneous symbol $\tm_2(z; \cdot)$ and \eqref{BWnon} for the non-homogeneous symbol  $\tm_{\geq 4}(z; \cdot)$, getting 
  \be\label{proj.op20}
\begin{aligned}
\Pi^\top  \Opbw{\im \, \tm (z; \cdot)} \Pi^\perp  z = & 
\!\!\!\!\!\! \sum_{ j_1-j_2 +j=k  \atop j_1 \neq j_2, \ j \in \Lambda^c,   k \in \Lambda } \!\!\!  \chi_2 \left(j_1, j_2,\frac{j+k}{2}\right)  \im \, \tm_{j_1, j_2}^{+,-}\left(\frac{j+k}{2}\right) z_{j_1} \bar z_{j_2} z_j
\, {e^{\im k x}} 
 \\
&  + \!\!\!\!\!\! \sum_{j \in \Lambda^c, \,   k \in \Lambda } \!\!\!  \chi \left(k-j,\frac{j+k}{2} \right)
\im \, \hat{\tm}_{\geq 4}\left(z; k-j, \frac{k+j}{2}\right) \, z_j \, {e^{\im k x}} \ .
\end{aligned}
\ee
Arguing as in \eqref{proj.op01},  the first line of \eqref{proj.op20} vanishes. 
To deal with the second line, recall that also  $\chi(\xi', \xi) \equiv 0$ when $|\xi'| \geq \langle \xi \rangle/10$, so when $k \in \Lambda$ and $j \in \Lambda^c$ (so $|j-k|\geq 1$)
\be\label{proj.op21}
\frac{1}{10}{\langle \frac{j + k }{2} \rangle} 
= \frac{1}{10}\big(1+  \frac{|j  \pm 1|}{2} \big) 
\leq \frac{3 + |j|}{20}  \leq \frac{4 + |j - k| }{20} \leq \frac{ |j - k| }{4} \leq |j-k|  \ , 
\ee
proving that 
$\chi \left( k-j, \frac{j+k}{2}\right) \equiv 0$.
In conclusion, also the second line of \eqref{proj.op20} vanishes, proving the first of \eqref{proj.op2}.
The second identity is analogous exchanging the roles of $j$ and $k$.

\noindent \underline{Proof of \eqref{proj.op3}.} It follows writing $\Pi^\perp = \uno - \Pi^\top$ and using the first of \eqref{proj.op2}. 

This concludes the analysis of the projection of the  paradifferential vector field 
$ \Opbw{\im \, \tm } z$.

 We pass to the cubic vector field $X^{(\Lambda)}(Z)$ in \eqref{Xlambda}. 
We set
$$
Y_3^{(\Lambda)}(z):= (\Pi_{\cP^{(0)}_\Lambda} X^{(\Lambda)})(Z)^+ \ , 
$$
which has the claimed form \eqref{Y.int} in view of \eqref{Y.str}.
Then we put 
$$
Y_3^\top(z) := \Pi^\top  \left( X^{(\Lambda)}(Z)^+ - (\Pi_{\cP^{(0)}_\Lambda} X^{(\Lambda)})(Z)^+ \right) \ , \quad
Y_3^\perp(z) := \Pi^\perp X^{(\Lambda)}(Z)^+ .
$$
To prove estimates \eqref{Y.perp} we exploit 
that $X^{(\Lambda)}(Z)$ is in strong-$\Lambda$ normal form, see \eqref{Y.str}. 
 
 \noindent \underline{Estimate of $Y_3^\top(z)$.} By definition 
$$
Y_3^\top(z) = \sum_{k \in \Lambda}  \ \ \sum_{(\vec{\jmath}, k, \vec \sigma, -) \in \cP \setminus \cP^{(0)}_\Lambda } X_{\vec{\jmath}, k }^{\vec \sigma, +} \, z_{\vec{\jmath} }^{\vec \sigma} \, e^{\im k x}  \ , \quad
\vec{\jmath} = (j_1, j_2, j_3)  , \quad \vec\sigma=(\sigma_1, \sigma_2, \sigma_3) \ . 
$$
By \eqref{Y.str},   $ \Pi_{\fP_\Lambda^{(1)}} X^{(\Lambda)}=  \Pi_{\fP_\Lambda^{(2)}} X^{(\Lambda)} =  0 $, so, since 
$k \in \Lambda$, 
the only possibly remaining monomials are those with
$(\vec{\jmath}, k, \vec \sigma, -) \in  \cP^{(3)}_\Lambda$ and in addition $\vec{\jmath} \in (\Lambda^c) ^3$.
Then, recalling \eqref{Xlambda}, $Y_3^\top(z) = \Pi^\top \left(R^{(\Lambda)}(Z^\perp)Z^\perp\right)^+$,  $Z^\perp:=\vect{z^\perp}{\bar z^\perp}$, and the 
 first estimate \eqref{Y.perp} follows from 
 $\norm{Y_3^\top(z) }_{\ts} \lesssim \norm{Y_3^\top(z) }_{L^2}  $ and
 estimate \eqref{smoothing}. 

\noindent\underline{Estimate of  $Y^\perp_3(z)$.}   Again by \eqref{Y.str},   we expand $Y^\perp_3(z)$ as
$$
Y_3^\perp(z) = \sum_{k \in \Lambda^c}  \ \ \sum_{(\vec{\jmath}, k, \vec \sigma, -) \in \cP^{(3)}_\Lambda \cup \cP^{(4)}_\Lambda } X_{\vec{\jmath}, k }^{\vec \sigma, +} \, z_{\vec{\jmath} }^{\vec \sigma} \, e^{\im k x}  \ .
$$
Then either $(i)$  two indexes among $(j_1, j_2, j_3)$ belong to $\Lambda^c$ and  one to $\Lambda$, or $(ii)$ all three indexes belong to $\Lambda^c$. 
Consequently $Y_3^\perp(z) =  \Pi^\perp \left(R^{(\Lambda)}(Z^\perp)Z^\perp 
+ R^{(\Lambda)}(Z^\perp)Z^\top + 2 R^{(\Lambda)}(Z^\perp, Z^\top)Z^\perp  
\right)^+$, $Z^\top:=\vect{z^\top}{\bar z^\top}$. 
The 
 second  estimate \eqref{Y.perp} follows again from estimate \eqref{smoothing} (with $m \leadsto -4$), using also the trivial bound $\norm{z^\top}_{\ts} \leq C_{\ts, \fs_0} \norm{z^\top}_{\fs_0}$.
This concludes the analysis of the projection of $X^{(\Lambda)}(Z)$. 

Finally we consider the projections of the vector field $\wt B_{\geq 4}(Z)Z$ in \eqref{Z.eq}.
We put 
\begin{align*}
& Y^\top_{\geq 5}(z) := \Pi^\top \big(\wt B_{\geq 4}(Z) Z\big)^+ +  \Pi^\top \Opbw{\im \, \tm_{\geq 4} } \Pi^\top z \ , \quad 
Y^\perp_{\geq 5}(z) := \Pi^\perp \big(\wt B_{\geq 4}(Z) Z\big)^+  \ . 
\end{align*}
\underline{Estimate of $Y^\perp_{\geq 5}(z)$.}  It follows
since $ \wt B_{\geq 4}(Z)$ is a matrix of non-homogeneous $0$-operators in $\cM^0_{\geq 4}[r]$, see \eqref{piove}. \\
\underline{Estimate of $Y^\top_{\geq 5}(z)$.}  As the previous one, using  also \eqref{stimapar2} and $\norm{\Pi^\top z}_\ts \lesssim \norm{z}_{\ts-1}$.
\end{proof}

The next step is to extract an effective system driving the dynamics of particular solutions of 
 \eqref{eq.ztop}--\eqref{eq.zperp}  which we call {\em long-time controlled}, see Definition \ref{A} below. 
These solutions have two main features: $(i)$ the initial data is supported mostly on $\Lambda$ and $(ii)$ 
 they have a  large a-priori bound on the high norm 
 $\norm{\cdot}_s$ for long  times. 
These features allow us to propagate smallness of  both tangential and normal modes in the   low norm $\norm{\cdot}_{\fs_0}$ for  long times, 
and moreover to ensure that the normal modes keep having a size  much smaller than the  tangential ones, i.e.  
$\norm{ z^\perp(t) }_{\fs_0}  \ll \norm{z^\top(t)}_{L^2}$, see \eqref{boot}, \eqref{z.s0}. 
 This is possible because of the normal form procedure of the previous section, and in particular because
\begin{itemize}
\item[(i)] the leading term in the dynamics of the  low modes $z^\top(t)$ in  \eqref{eq.ztop} is  the  cubic integrable vector field $Y_3^{(\Lambda)}(z^\top)$ (the non-explicit cubic term $Y^\top_3(z) = \cO((z^\perp)^3)$, hence its size  is much smaller);   
\item[(ii)] in  equation \eqref{eq.zperp} for 
 $z^\perp(t)$, the term 
 $\Opbw{  \im \,  \tm(z; x, \xi) }z^\perp$ is skew-adjoint, hence it  vanishes in a 
 $L^2$-energy estimate; 
 consequently the dominant term  becomes 
 $Y_3^\perp(z)$ 
 which, in view of \eqref{Y.perp}, fulfills the quadratic estimate 
 $\norm{Y_3^\perp(z)}_{\fs_0} \lesssim \norm{z^\top}_{\fs_0} \norm{z^\perp}_{\fs_0}^2$ and therefore has a very small size. 
To obtain such estimate is the reason why we put   $X^{(\Lambda)}(Z)$ in \eqref{Xlambda} in strong-$\Lambda$ normal form, namely it does not contain monomials of the form
$z_{j_1}^{\sigma_1} z_{j_2}^{\sigma_2} z_{j_3}^{\sigma_3} e^{\im j x}$ supported in  $\cP^{(2)}_\Lambda$. Otherwise,  $Y^\perp_3(z)$ would have had monomials  with exactly two frequencies among $(j_1, j_2, j_3)$ in $\Lambda$ and one in  $\Lambda^c$, and the estimate in \eqref{Y.perp} would have had  an  additional term 
$\norm{z^\top}_{\fs_0}^2 \norm{z^\perp}_{s}$, which is  too large for the bootstrap lemma \ref{lem:boot} below.  
%
\end{itemize}

We now introduce precisely the notion of long-time controlled solutions.

\begin{definition}[{\bf Long-time controlled solutions}]\label{A}
Let  $s, \theta$ as in \eqref{parameters}. Let also $T_\star>0$ and $ \e \in (0,\fr)$. 
We say that a solution $z(t) \in H^s(\T, \C)$ of system  \eqref{eq.ztop}--\eqref{eq.zperp} is {\em long-time controlled}  with parameters $(s, \theta, T_\star, \e)$ if 
\begin{itemize}
\item[$(A1)$] at time $0$ fulfills
\begin{equation}
\label{app.ass1}
\norm{ z^\top(0, \cdot)}_{L^2} \leq \e    \ , \qquad
\norm{ z^\perp(0, \cdot)}_{L^2} \leq \e^{3}   \ ; 
\end{equation}
\item[$(A2)$] it 
 exists   over the time interval $[0, T_\star]$
where  it fulfills  
the large  a-priori bound
\begin{equation}
\label{app.ass2}
\sup_{0 \leq t \leq T_\star } 
\norm{z(t)}_s \leq  \e^{-\theta}    \ . 
\end{equation}
\end{itemize}
\end{definition}
One crucial property of any long-time controlled solution is that
its   low norm $\norm{\cdot}_{\fs_0}$  is automatically  small for all $0 \leq t \leq T_\star$, as we shall now prove.

\begin{lemma}[\bf Bootstrap lemma]\label{lem:boot}
Let  $s, \theta$ as in \eqref{parameters}. Fix also 
 $T_0 >0$.   
There exists $\e_\star = \e_\star(\theta, T_0) >0$ such that for any $\e \in (0, \e_\star)$ the following holds true. 

Let $z(t)$ be a solution  of  \eqref{eq.ztop}--\eqref{eq.zperp} which is long-time controlled   with parameters $(s, \theta, T_\star, \e)$ (according to Definition \ref{A}) and with 
\be\label{Tstar.T0}
T_\star\leq \frac{T_0}{\e^2}\log\left(\frac{1}{\e} \right) \ .
\ee
Then 
  $z(t)$ fulfills  the improved $L^2$-bound
\begin{equation}
\label{boot}
\norm{z^\top(t)}_{L^2} \leq 2 \e  \ , \qquad 
\norm{z^\perp(t)}_{L^2} \leq \e^{3- \frac{ 3}{2}\theta}  \ ,  \quad \forall  0\leq t \leq  T_\star 
\end{equation}
and the improved low-norm bound 
\begin{equation}
\label{z.s0}
\norm{z(t)}_{\fs_0} \leq 3 \e \ , \qquad 
\norm{z^\perp(t)}_{\fs_0} \leq \e^{2}  \ \ ,  \quad \forall  0 \leq t \leq  T_\star \  . 
\end{equation}
\end{lemma}
\begin{proof}
The proof is by a bootstrap argument. 
We assume the bound 
\begin{equation}
\label{boot2}
 \norm{z^\top(t)}_{L^2} \leq 10 \e  , \quad 
\norm{z^\perp(t)}_{L^2} \leq \e^{3- 2 \theta}  \ ,  \quad \forall 0 \leq t \leq T_\star
\end{equation}
and show that, provided $\e \in (0, \e_\star)$ with $\e_\star$  sufficiently small,   the better bound \eqref{boot} holds.

First we   bound $\norm{z^\perp(t)}_{\fs_0}$.
This is  done interpolating the bound on 
$\norm{z^\perp(t)}_{L^2}$ that we have by the 
bootstrap assumption \eqref{boot2} 
 and the   large bound  that we have on $\norm{z^\perp(t)}_{s}$  in \eqref{app.ass2},  being $z(t)$ long-time controlled by assumption.
We obtain 
\begin{align}
\label{z.perp.s0}
\norm{z^\perp(t)}_{\fs_0} 
& \leq  \norm{z^\perp(t)}_{L^2}^{1-\frac{\fs_0}{s}} \, 
\norm{z^\perp(t)}_{s}^{\frac{\fs_0}{s}} \stackrel{\eqref{boot2}, \eqref{app.ass2} }{\leq}
 \e^{3 - \theta(2-\frac{\fs_0}{s}) - 3 \frac{\fs_0}{s}}  \leq \e^2 \
\end{align}
which is possible for $s, \theta$ as in \eqref{parameters}. 
Using again the first of \eqref{boot2}  we also get 
\begin{equation}
\label{boot3}
\norm{z(t)}_{\fs_0} \leq 11 \e \ , \qquad 
\forall 0\leq t  \leq T_\star \ . 
\end{equation}
Next we consider  $\norm{z^\top(t)}_{L^2}$ and prove the  improved estimate \eqref{boot}.
Recall that the function $z^\top(t)$ fulfills equation \eqref{eq.ztop}; since $Y^{(\Lambda)}_3(z)$ is integrable, we get that for all times  $ 0 \leq t \leq T_\star$
\begin{align*}
\frac{\di}{\di t} \norm{z^\top(t)}_{L^2}^2  & = 2\underbrace{ \Re \la -\ii |D|^\alpha z^\top+Y^{(\Lambda)}_3(z) , z^\top \ra }_{ = 0}
+ 2 \Re \la Y^{\top}_3(z) + Y^\top_{\geq 5}(z), z^\top \ra \\
& \stackrel{\eqref{Y.perp}, \eqref{Y5}}{ \leq}
C \big( \norm{z^\perp(t)}_{\fs_0}^3 + 
\norm{z(t)}_{\fs_0}^5  \big) \, \norm{z^\top(t)}_{L^2}  \stackrel{\eqref{z.perp.s0}, \eqref{boot3},  \eqref{boot2}}{\leq} C \e^6  \ . 
\end{align*}
Then, since $z(t)$ is long-time controlled, its initial datum $z^\top(0)$ is bounded by   \eqref{app.ass1}; hence  for all times  $ 0 \leq t  \leq T_\star \leq \frac{T_0}{\e^2}\log\left(\frac{1}{\e} \right)$, 
\begin{equation}\label{boot.res}
 \norm{z^\top(t)}_{L^2}^2 \leq  \norm{z^\top(0)}_{L^2}^2 + |t| C \e^6 \leq \e^2 +  C T_0 \e^4 \log(\e^{-1}) \leq 4\e^2 
\end{equation}
provided $ 0< \e \leq \e_\star$ and $\e_\star$ is sufficiently small. This proves the first estimate in \eqref{boot}.

Next we bound  $\norm{z^\perp(t)}_{L^2}$.  
We exploit that the paradifferential operator in  equation \eqref{eq.zperp} is   skew-adjoint, 
so we get, for all times $ 0 \leq t  \leq T_\star \leq  \frac{T_0}{\e^2}\log\left(\frac{1}{\e} \right)$, 
\begin{align*}
\frac{\di}{\di t} \norm{z^\perp(t)}_{L^2}^2 
& = 
2\underbrace{ \Re \la \Big(-\ii |D|^\alpha +\Opbw{\im \tm(z;\cdot)}\Big) z^\perp , z^\perp \ra }_{ = 0}+
 2 \Re \la  Y^\perp_3(z) + Y^\perp_{\geq 5}(z), z^\perp \ra \\
& \stackrel{\eqref{Y.perp}, \eqref{Y5}}{ \leq}
C \left( \norm{z(t)}_{\fs_0}  \, \norm{z^\perp(t)}_{\fs_0}^2  
+
 \norm{z(t)}_{\fs_0}^5
\right)
\norm{z^\perp(t)}_0 \\
& 
 \stackrel{\eqref{boot3}, \eqref{z.perp.s0},\eqref{boot2}}{\leq} C \e^{8- 2 \theta}  \ .
\end{align*}
Again, being $z(t)$ long-time controlled, its initial datum $z^\perp(0)$ fulfills  \eqref{app.ass1}; hence for all times 
$ 0 \leq t \leq T_\star  \leq \frac{T_0}{\e^2}\log\left(\frac{1}{\e} \right)$ we bound
\begin{equation}\label{boot.res2}
 \norm{z^\perp(t)}_{L^2}^2 \leq  \norm{z^\perp(0)}_{L^2}^2 + |t|  C \e^{8-2\theta} \leq \e^6 +  C T_0 \e^{6- 2\theta} \log(\e^{-1}) \leq \e^{2(3-\frac32 \theta)} \ ,
\end{equation}
which is true shrinking $\e_\star$. 
Estimates  \eqref{boot.res} and \eqref{boot.res2} prove  \eqref{boot}. 
{This verifies the bootstrap assumption and so, by  \eqref{z.perp.s0},  also the second of   \eqref{z.s0}. Together with \eqref{boot}, we get also the first of \eqref{z.s0}.}
\end{proof}

A second important property of  any long-time controlled solution is that it fulfills an {\em effective equation} with a very precise structure: up to higher order corrections, for long times,  
the modes $z_{\pm 1}(t)$ rotate with constant speed, whereas  $z^\perp(t)$ fulfills a  linear Schr\"odinger equation
whose  Hamiltonian 
$
- \im |D|^\alpha  + \im \Opbw{\tv(x-\tJ_1 t) \xi }
$
{\em does not have constant coefficients}. 
We shall show, in the next section,  that this Hamiltonian is actually responsible for the growth of Sobolev norms of the solution. 
Precisely we prove the following result:
\begin{proposition} \label{prop:eff}
Let $s, \theta$ as in \eqref{parameters}. 
Fix also 
 $T_0 >0$.   
There exists $\e_\star = \e_\star(s, \theta, T_0) >0$ such that for any $\e \in (0, \e_\star)$ the following holds true. 
Let $z(t)$ be a solution  of  \eqref{eq.ztop}--\eqref{eq.zperp} which is long-time controlled  with parameters $(s, \theta, T_\star, \e)$  (see Definition \ref{A}) and with $T_\star$ fulfilling 
 \eqref{Tstar.T0}. 
Then   $z(t) = (z_{1}(t), z_{-1}(t),  z^\perp(t)) $ fulfills the system
\be\label{eff.sys}
\begin{cases}
 \pa_t z_1 = - \im \big( 1+ |z_1(0)|^2\big) z_1  + \td_1(t) \\
 \pa_t z_{-1}  = - \im (1-  |z_{-1}(0)|^2 \big)  z_{-1} + \td_{-1}(t) \\
 \pa_t z^\perp  = - \im |D|^\alpha z^\perp  + \im \Opbw{\fv(x - \tJ_1 t) \xi + \tV(t;x) \xi + \tb(t;x, \xi)}z^\perp + Y(t)
\end{cases}
\ee
where\\
$\bullet$ $\tJ_1$ is the real number
 \begin{equation}\label{J1}
\tJ_1:=  \frac{|z_1(0)|^2 + |z_{-1}(0)|^2}{2} \ , 
\end{equation}
$\bullet$ the real valued function  $ \fv(x)$ is given by
\begin{equation}\label{V0}
 \fv(x) :=  2 \Re \left(z_1(0) \, \bar{z_{-1}(0)} \, e^{\im 2 x} \right) 
 \end{equation}
whereas the real valued, time dependent function $ \tV(t;x)$ fulfills the estimate
\begin{equation}\label{Vpert}
\norm{ \tV(t; \cdot)}_{W^{2,\infty}}  \leq C \e^{{4}- \theta} \ , \quad
\forall 0 \leq t \leq T_\star \ ;
\end{equation}
$\bullet$ the real valued symbol $\tb(t;x , \xi) \in \Gamma^\alpha_{W^{2,\infty}}$ fulfills the estimate  (recall \eqref{seminorm}): for every $n \in \N_0$, there is $C_n >0$ such that 
\begin{equation}\label{bpert}
 \, \abs{ \tb(t; \cdot) }_{\alpha, W^{2, \infty}, n}  \leq C_n \e^2 \ , \quad
 \forall 0 \leq t \leq T_\star \ ;
\end{equation}
$\bullet$ the functions $\td_{\pm 1}(t)$ fulfill the estimates
\be\label{d.bound}
\abs{\td_{\pm 1}(t)} \leq \e^{5-\theta}   \ , \quad
 \forall 0 \leq t \leq T_\star \ ;
\ee
$\bullet$ the vector field $Y(t) \equiv Y(t, x)$ fulfills the estimate 
\begin{equation}\label{Y(t,U)}
\norm{Y(t; \cdot)}_{s } \leq C  \e^{3-\theta} \ , \quad
 \forall 0 \leq t \leq T_\star \ .
\end{equation}
\end{proposition}
\begin{proof}
We shall use that $z(t)$, being long-time controlled  with parameters $(s, \theta, T_\star, \e)$ and with $T_\star$ fulfilling 
 \eqref{Tstar.T0}, satisfies the bounds \eqref{boot}, \eqref{z.s0}. 
 
\noindent\underline{Equations for $z_{\pm 1}(t)$.}
Write equation  \eqref{eq.ztop} in components,  using the explicit expression of $Y_3^{(\Lambda)}$ in \eqref{Y.int}, to get the coupled system
\be\label{z1.z-1}
\begin{cases}
 \pa_t z_1 = - \im z_1 - \im |z_1|^2 z_1  + \la Y_3^\top(z) + Y^\top_{\geq 5}(z),  e^{\im x} \ra \\
 \pa_t z_{-1}  = - \im z_{-1} + \im |z_{-1}|^2 z_{-1} + \la  Y_3^\top(z) + Y^\top_{\geq 5}(z), e^{-\im x} \ra \ . 
\end{cases}
\ee
Consider the equation for $z_1$. We write it as
\be\label{z1(t)}
\begin{aligned}
\pa_t z_1 & = - \im ( 1+  |z_1(0)|^2)  z_1   +  \td_1(t) , \\
& \quad 
\td_1(t) := -\im \left(|z_1(t)|^2 - |z_1(0)|^2 \right) z_1(t) 
+ \la Y_3^\top(z) + Y^\top_{\geq 5}(z), e^{\im x} \ra 
\end{aligned}
\ee
giving the first equation in \eqref{eff.sys}. We prove now that $\td_1(t)$ fulfills the bound claimed in \eqref{d.bound}.  
First,  using the first of \eqref{z1.z-1} and assumption \eqref{Tstar.T0}, we get for all times $ 0 \leq t \leq T_\star \leq \frac{T_0}{\e^2}\log\left(\frac{1}{\e} \right)$ 
 \begin{align*}
\frac{\di}{\di t} |z_1(t)|^2 & = 2 \Re  \left( \la Y_3^\top(z)+ Y^\top_{\geq 5}(z), e^{\im x} \ra \,  \bar z_1 \right) \\
&
 \stackrel{\eqref{Y.perp}, \eqref{Y5}}{ \leq}
C \big( \norm{z^\perp(t)}_{\fs_0}^3 + 
\norm{z(t)}_{\fs_0}^5  \big) \, \norm{z^\top(t)}_{0}  \stackrel{\eqref{z.s0}, \eqref{boot3},  \eqref{boot}}{\leq} C \e^6 \ ,
 \end{align*}
 which implies, on the same time scale, 
\be\label{controlz1}
\abs{ \big( |z_1(t)|^2 - |z_1(0)|^2 \big)  } \leq  C |t|  \e^6 \leq C  T_0 \, \e^4 \log(\e^{-1})  \ .
 \ee
Hence we get that $\td_1(t)$ in \eqref{z1(t)} is bounded for $0 \leq t \leq T_\star \leq \frac{T_0}{\e^2}\log(\e^{-1})$ by 
 \be\label{d1.bound}
 \abs{\td_1(t)} \leq 
 \abs{ \big( |z_1(t)|^2 - |z_1(0)|^2 \big) z_1(t) } +
 \abs{\la Y_3^\top(z)+ Y^\top_{\geq 5}(z), e^{\im x} \ra }  
 \stackrel{\eqref{controlz1}, \eqref{boot}}{\leq} C  T_0 \, \e^5 \log(\e^{-1}) + C \e^5 \ , 
 \ee
proving  \eqref{d.bound} provided $\e_\star$ is sufficiently small.
  An analogous argument proves that $z_{-1}(t)$ fulfills the second of \eqref{eff.sys}.

A consequence, which we shall use in a moment, is that 
\begin{equation}
\label{z1z-1.new}
z_{\pm 1}(t) = \tz_{\pm 1}(t) + r_{\pm 1}(t) \ , \quad \mbox{ where } 
\tz_{\pm 1}(t):= e^{- \im t(1\pm |z_{\pm 1}(0)|^2)} z_{\pm 1}(0)
\end{equation}
whereas 
\begin{equation}\label{r1}
r_{\pm 1} (t) := \int_0^t e^{- \im (t - \tau) (1 \pm |z_{\pm 1}(0)|^2)}  \, \td_{\pm 1} (\tau)  \, \di \tau  \ 
\end{equation}
 fulfill, by \eqref{d1.bound}, \eqref{Tstar.T0} and eventually shrinking again $\e_\star$,  the bounds 
\begin{equation}
\label{r.bound}
|r_{\pm 1} (t)| \leq \e^{{3}-\theta} , \qquad \forall 0 \leq t \leq T_\star \ .
\end{equation}

\noindent\underline{Equation for $z^\perp(t)$.} We start from equation \eqref{eq.zperp} and we substitute the explicit expression of $z_{\pm1}(t)$ in \eqref{z1z-1.new}. Consider first the symbol $\tm (z;x, \xi)$ in \eqref{tm}. We shall extract from its component  $\la \, \und \tV \, \ra(Z; x)$, defined in \eqref{VresZ}, the main contribution which is the one supported on $z_{\pm 1}(t)$.
Precisely
\begin{align*}
\la \, \und \tV \, \ra(Z(t); x) & = 2\,  \Re \left( z_{1}(t) \, \bar{z_{-1}(t)} \, e^{\im 2 x} \right) 
 + 2\,  \Re \Big( \sum_{n \geq 2} z_{n}(t) \, \bar{z_{-n}(t)} \, e^{\im 2n x} \Big)\\
 & \stackrel{\eqref{z1z-1.new}}{= }
 \underbrace{ 2\,  \Re \left( z_{1}(0) \, \bar{z_{-1}(0)} \, e^{\im 2 x-  2\tJ_1 t} \right) }_{=\fv(x-\tJ_1 t) \mbox{ by } \eqref{V0}, \eqref{J1}}
+ \underbrace{2\,  \Re \left( 
\Big( \tz_{1}(t) \, \bar{r_{-1}(t)} + r_1(t) \bar{\tz_{-1}(t)} +
r_1(t) \bar{r_{-1}(t)} \Big)
 \, e^{\im 2 x} \right)}_{=: \tV_1(t; x)} \\
 & 
 +\underbrace{  2\,  \Re \Big( \sum_{n \geq 2} z_{n}(t) \, \bar{z_{-n}(t)} \, e^{\im 2 n  x} \Big)}_{=: \tV_2(t; x)}\ .
\end{align*}
The functions  $\tV_1(t; x)$ and $\tV_2(t;x)$ fulfill, by  \eqref{app.ass1}, \eqref{r.bound} and \eqref{z.s0}, the bounds
\be\label{est.tV12}
\norm{\tV_1(t; \cdot)}_{W^{2,\infty}} \leq C  \e^{4-\theta} \ , \quad 
\norm{\tV_2(t; \cdot)}_{W^{2,\infty}} \leq C \e^{4}  \ , \quad
\forall 0 \leq t \leq T_\star \ .
\ee
Then we write  $\tm(z;\cdot)$ in \eqref{tm} as 
$$
\tm(z(t); x, \xi)= \fv(x-\tJ_1 t)\xi +  \underbrace{(\tV_1(t; x) +  \tV_2(t; x) + \wt V_{\geq 4}(z(t); x) )}_{=: \tV(t; x)}\xi 
+
\underbrace{  {a}^{(\alpha)}_2(z(t);x,\xi) +  
   \wt a^{(\alpha) }_{\geq 4} (z(t);x,\xi) }_{=: \tb(t; x, \xi)}
$$
We bound  $\tV(t;x)$ using  estimates  \eqref{est.tV12} for $\tV_1$ and $\tV_2$, and that
$$
\norm{\wt V_{\geq 4}(z(t); \cdot )}_{W^{2, \infty}} \stackrel{\eqref{nonhomosymbo}}{\leq} C \norm{z(t)}_{\fs_0}^4 \stackrel{\eqref{z.s0}}{\leq} C \e^4 \ , \quad
\forall 0 \leq t \leq T_\star \ ,
$$
getting the claimed bound \eqref{Vpert}.

The bound \eqref{bpert} for $\tb(t;x, \xi)$ follows from \eqref{nonhomosymbo.homo}, \eqref{nonhomosymbo} and \eqref{z.s0}.

Finally we put
$$
Y(t, z):= Y_3^\perp(z(t)) + Y^\perp_{\geq 5}(z(t))
$$
which fulfills the estimates \eqref{Y(t,U)} by \eqref{Y.perp}, \eqref{Y5} and using \eqref{z.s0} and \eqref{app.ass2}.

\end{proof}
 
 \section{Instability via paradifferential Mourre theory}\label{sec:mourre}
The goal of this section is to give sufficient conditions on the initial datum $z(0)$ 
ensuring that, 
if the corresponding solution $z(t)$ is long-time controlled, {then} its high $H^s$-norm  undergoes Sobolev norm explosion,  becoming  larger than $\e^{-\theta}$.
We will achieve this via a  positive commutator estimate.

We will focus on the  third equation in \eqref{eff.sys}; actually  it is more convenient to work with  the translated variable
\be\label{zeta}
{
\zeta(t,x):= z^\perp \big( t,  x +\tJ_1 t \big) }\ , \quad \tJ_1 \mbox{ in } \eqref{J1} \ .
\ee
Clearly one has   
\be\label{zetaez}
\norm{\zeta(t, \cdot)}_s = \norm{z^\perp(t, \cdot)}_s \ , \quad \forall t , \quad \forall s \in \R \ , 
\ee
so it is equivalent to prove growth of Sobolev norms for $\zeta(t)$ and $z^\perp(t)$.
The equation fulfilled by 
 $\zeta(t)$ is easily derived from the third of \eqref{eff.sys} as 
\begin{align}\label{eq.zeta}
 \pa_t \zeta  = &  - \im |D|^\alpha \zeta   + \im \Opbw{{(\tJ_1+\fv(x ))} \xi }\zeta + \im \Opbw{ \wt \tV(t; x) \xi + \wt \tb(t; x,  \xi)}\zeta  + \wt Y(t)
 \end{align}
where we defined  the real valued function $\wt \tV(t; x)$, the real valued symbol  
$\wt \tb(t; x,  \xi)$ and the vector field $ \wt Y(t; x)$ as
\be
\wt \tV(t; x) := \tV(t, x+ \tJ_1 t) \ , \quad 
\wt \tb(t; x,  \xi):= \tb(t; x+ \tJ_1 t,  \xi) \ , \quad
 \wt Y(t; x) :=   Y(t; x+ \tJ_1 t) \ . 
\ee
It  follows, by \eqref{Vpert}, \eqref{bpert} and \eqref{Y(t,U)},  the estimates 
\be\label{est.706}
\norm{ \wt\tV(t; \cdot)}_{W^{2,\infty}}  \leq  C \e^{{4}- \theta} \ , 
\quad
\vert \wt \tb(t; \cdot) \vert_{\alpha, W^{2, \infty}, n}  \leq C_n \e^2  \ , 
\quad
\norm{\wt Y(t; \cdot)}_{s } \leq C  \e^{3-\theta}  \ , \quad
\forall \ 0 \leq t \leq T_\star \ .
\ee
\subsection{The Mourre operator}
The leading term in  equation \eqref{eq.zeta} is the {\em non-constant coefficient transport} operator
 \begin{equation}\label{fH}
  \Opbw{ \big(\tJ_1+\fv(x)\big) \xi}  \ , \quad \tJ_1 \mbox{ in } \eqref{J1}  \ , \ \ \ \fv(x) \mbox{ in } \eqref{V0} \ .
\end{equation}
 The crucial point is that, provided $z_{1}(0)$ and $z_{-1}(0)$ fulfill
  \be\label{cond.z1}
 \tJ_1\equiv   \frac{|z_1(0)|^2 + |z_{-1}(0)|^2}{2} < 2  |z_1(0) | \, |{z_{-1}(0)} |  \ ,
 \ee
 corresponding to the function 
 $\tJ_1 + \fv(x)$ having a zero, 
the  operator  $ \Opbw{ \big(\tJ_1+\fv(x)\big) \xi}$
admits a Mourre-conjugate operator, namely  an operator $\fA$ such that the commutator
$\im [\fA,  \Opbw{ \big(\tJ_1+\fv(x)\big) \xi}]$ is positive. 
Actually this also shows that the operator in \eqref{fH} has a non-trivial absolutely continuous spectrum, although we shall not exploit directly this property.

Precisely, take $s$ as in \eqref{parameters} and $\tR \gg 1$ (to be fixed later) and   define the (formally) self-adjoint operator
\begin{equation}
\label{fA}
\begin{aligned}
\fA := \fA_{s,\tR}&  := \Opbw{\fa (x, \xi)} , \quad \fa(x, \xi):=\ta(x) \, |\xi| ^{2s} \, \eta^2_{\tR}(\xi)  \\ 
&\mbox{ where }  \ta(x):= -  \Im  \, \left( z_{1}(0) \,  \bar{z_{-1}(0)} \, e^{\im 2  x }\right) 
\end{aligned}
\end{equation}
and
$\eta_\tR(\xi)$  the
smooth step function 
 \be \label{etaR}
 \eta_\tR(\xi):= \eta\left(\frac{\xi}{\tR}\right),\quad \eta(y):= 
 \begin{cases}
0& \mbox{ if } y\leq 1\\ 
 \dfrac{e^{-\frac{1}{y-1}}}{e^{-\frac{1}{y-1}} + e^{-\frac{1}{2-y}}}& \mbox{ if } y \in (1,2) \\
 1& \mbox{ if } y \geq 2
 \end{cases}
  \ .
 \ee
Note that 
 $\fa(x, \xi)$ is a symbol in   $ \Gamma^{2s}_{W^{2, \infty}}$, and for any $n \in \N_0$, there is $C_n >0$ such that 
\begin{equation}\label{a.sem}
| \fa |_{2s , W^{2,\infty},  n } \leq  C_{s,n} \,   |z_1(0)|\,|z_{-1}(0)|  \ , \qquad 
| \fa |_{ 2s +1 , W^{2,\infty},  n } \leq C_{s,n} \frac{  |z_1(0)|\,|z_{-1}(0)| }{\tR}  \ ,
\end{equation} 
 as it follows from its definition and from  Lemma \ref{lem:cutoff} with $a \leadsto \ta(x)|\xi|^{2s} \eta_\tR(\xi)$, $m \leadsto 2s$, $N \leadsto 2$ and $\nu \leadsto 1$.
 Moreover we will ensure that $|z_1(0) z_{-1}(0)| >0$, so that $\fA$ is non trivial, see Remark \ref{rem:nontriv}.
\smallskip
 
 The choice of the function $\fa(x,\xi)$ in \eqref{fA} is motivated by the fact that it is an escape function for the symbol $(\tJ_1 + \fv(x))\xi$ of the operator in \eqref{fH}; precisely one has the following result:
 \begin{lemma}\label{lem:pb}
Fix $s, \tR >1$. Let $\fa(x,\xi)$ as in \eqref{fA} and $\tJ_1$, $\fv(x)$ as in \eqref{J1}, \eqref{V0}. Then
\be\label{comm.av0}
\{ \fa(x,\xi) , \, \big(\tJ_1 + \fv(x) \big)\xi \} = 
\tI_1\, |\xi|^{2s} \,  \eta_\tR^2(\xi)   + a(x,\xi)  
\ee
where $\tI_1$ is the real number
\be\label{tI1}
\tI_1:=  2  |z_1(0)| \, |z_{-1}(0)| \, \Big( 2  |z_1(0)| \, |z_{-1}(0)| - \frac{|z_1(0)|^2 + |z_{-1}(0)|^2}{2}  \Big)
 \ee
 whereas  $ a(x,\xi)$ is a smooth, non-negative symbol having the  structure
\be\label{sym.a}
 a(x,\xi) =  a_1(x)  \psi_1(\xi)^2 +  a_2(x) \psi_2(\xi)^2 \ .
\ee
Here  $ a_j(x)$, $j=1,2$, are  smooth, real valued,  non-negative functions fulfilling
\be\label{est.aj}
\norm{a_j(x)}_{W^{3,\infty}} \leq C \left( |z_1(0)|^4 +  \, |z_{-1}(0)|^4 \right) \ , 
\ee
and 
 $\psi_j(\xi)$, $j = 1,2$, are smooth, real valued symbols in $\wt \Gamma^s_0$ with  support in $[\tR, +\infty)$.
 \end{lemma}
 \begin{proof}
We compute, using \eqref{poisson},  \eqref{fA}, \eqref{J1}, \eqref{V0} and denoting  $(\eta')_\tR(\xi):= \eta'(\xi/\tR)$,
\begin{align}
\notag
 \{ \fa(x, \xi)  ,  &  \big(\tJ_1 + \fv(x)\big)\xi  \} = 
\left(  2 s\,  \ta \,  \fv_x - \fv \, \ta_x  - \tJ_1 \ta_x   \right)   |\xi|^{2s} \, \eta_\tR^2      + \frac{2}{\tR}\ta  \fv_x \, |\xi|^{2s}\xi  \, \eta_\tR \,  (\eta')_\tR \\
\label{comm.av}
& =
\left(   \ta \,  \fv_x - \fv \, \ta_x  - \tJ_1 \ta_x \right) |\xi|^{2s} \, \eta_\tR^2 + (2s-1) \ta \fv_x   |\xi|^{2s} \eta_\tR^2 
+  2 \ta  \fv_x \, |\xi|^{2s}\eta_\tR \, \frac{\xi}{\tR}   \,  (\eta')_\tR \ .
\end{align}
Now, using the explicit definition of $\ta(x) $ in \eqref{fA}, of $\fv(x)$ in \eqref{V0} and of $\tJ_1$ in \eqref{J1} and that $\ta_x(x)=- 2 \Re  \, \left( z_{1}(0) \,  \bar{z_{-1}(0)} \, e^{\im 2  x }\right) $, $\fv_x(x)=-4 \Im \left(z_1(0) \, \bar{z_{-1}(0)} \, e^{\im 2 x} \right) $,
we get the lower bound
\begin{align}
\notag
\ta \fv_x - \fv \ta_x  - \tJ_1 \ta_x 
& = 4\,  \Im \left( z_{1}(0) \,  \bar{z_{-1}(0)} \, e^{\im 2  x } \right)^2  + 4 \,  \Re \left( z_{1}(0) \,  \bar{z_{-1}(0)} \, e^{\im 2  x } \right)^2  - \ta_x \tJ_1 \\
&  \geq 
4 \, | z_{1}(0)|^2 \,  | z_{-1}(0)|^2 - 2 \tJ_1 |z_1(0)| \, |z_{-1}(0)| \\
\label{av-va}
& \geq 2  |z_1(0)| \, |z_{-1}(0)| \  \Big( 2  |z_1(0)| \, |z_{-1}(0)| - \tJ_1 \Big) \equiv \tI_1 \ , 
\end{align}
where to pass from the first to the second line  we also used that
$$
|\ta_x| \leq 2  |z_1(0)| \, |z_{-1}(0)| \ .
$$ 
Hence, adding and subtracting $\tI_1 |\xi|^{2s} \eta_\tR^2(\xi)$ in  \eqref{comm.av}, we get the claimed formula \eqref{comm.av0} with 
\begin{align}\label{gar.sym}
 a(x, \xi) := & 
\underbrace{\left(\ta \fv_x - \fv \ta_x  - \tJ_1 \ta_x - \tI_1 
+(2s-1) \ta \fv_x 
  \right)}_{=:a_1(x)} \underbrace{|\xi|^{2s} \eta_\tR^2 }_{=:\psi_1(\xi)^2} +   \underbrace{2 \ta  \fv_x }_{=:a_2(x)}\, \underbrace{|\xi|^{2s}\eta_\tR \, \frac{\xi}{\tR}   \,  (\eta')_\tR}_{=:\psi_2(\xi)^2}   \ .
\end{align}
Note that both  $a_1(x)$ and $a_2(x)$ are non-negative functions in view of  \eqref{av-va} and the fact that  $\ta \fv_x = 4\,  \Im \left( z_{1}(0) \,  \bar{z_{-1}(0)} \, e^{\im 2  x } \right)^2 \geq 0$. 
They clearly are smooth, and  estimate \eqref{est.aj} follows from the definitions of $\ta(x), \fv(x)$ in \eqref{fA}, \eqref{V0}, of $\tJ_1$ in \eqref{J1} and $\tI_1$ in \eqref{tI1}.

We claim that the functions  $\psi_1(\xi)= |\xi|^s \eta_\tR$ and $\psi_2(\xi) = |\xi|^s \sqrt{\eta_\tR \, \frac{\xi}{\tR} (\eta')_\tR}$ are smooth symbols in $\wt \Gamma^{s}_0$ supported in $[\tR, \infty)$.
We prove the claim only for $\psi_2$ since the one  for $\psi_1$ is trivial.
First notice that $\psi_2$ is well defined since, by  \eqref{etaR}, one has $\xi (\eta')_\tR \geq 0$.
Define
 $$
f(y):= \sqrt{\eta(y) \, y \eta'(y)} \ , \quad \textup{supp}(f) \subset [1,2] \ .
$$
 Then  $\psi_2(\xi)=|\xi|^s f(\xi/\tR)$ and is supported in $[\tR, 2\tR]$. So we are left to prove that 
$f(y)$ is a smooth function. It is easy to see that $\sqrt{y \eta(y)}$ is smooth on its support.  The function 
$$
\sqrt{\eta'(y)} =
\begin{cases}
0 \ , & y\leq 1\\
\dfrac{\sqrt{2 y^2-6 y+5}}{e^{-\frac{1}{2-y}}+e^{-\frac{1}{y-1}} } \cdot \dfrac{e^{-\frac{1}{2(y-1)}}}{y-1} \cdot  
\dfrac{e^{-\frac{1}{2(2-y)}}}{2-y}   \ , & y \in (1,2) \\
0 \ ,& y\geq 2
\end{cases}
$$
 is smooth by direct inspection.
 \end{proof}


%
%

Thanks to Lemma \ref{lem:pb}, we now prove that the  commutator between  $\fA$ in \eqref{fA} and $\Opbw{(\tJ_1 + \fv(x))\xi}$   is a non-negative operator up to a small remainder.
In the following, given two operators $\fA, \fB$,  we write $\fA \geq \fB$ with the meaning $\la \fA u, u \ra \geq \la \fB u, u \ra $ for any $u \in \bigcap_s H^s$.
  Precisely we have: 
\begin{lemma}\label{lem:pos.comm}
Fix $s, \tR > 1$.  Let $\fA\equiv \fA_{s, \tR}$ be defined in \eqref{fA}. Then:
\begin{itemize}
\item[(i)] {\bf Positive commutator:}  Let $\tJ_1$ in \eqref{J1} and $\fv(x)$ in \eqref{V0}. One has  
\begin{equation}\label{mourre}
\im \big[\fA,   \Opbw{ \big( \tJ_1 + \fv(x)\big) \xi} \big] \geq \tI_1   \, \Opbw{  |\xi|^{2 s} \eta^2_\tR(\xi)}  + \fR
\end{equation}
with $\tI_1$ in \eqref{tI1} and the operator  $\fR\colon H^{s} \to H^{-s}$ with estimate
\begin{equation}\label{mourre.R}
\norm{\fR  u}_{-s } \leq C_{s}  \frac{ |z_1(0)|^4 +|z_{-1}(0)|^4}{\tR}  \norm{u}_s    \ . 
\end{equation}
\item[(ii)] {\bf Upper bound:} One has 
 \begin{equation}\label{upper.bound}
 \fA \leq 2 |z_1(0)|\,|z_{-1}(0)|\,  \,      \Opbw{|\xi|^{2s} \eta_\tR^2(\xi)}   + \fR
\end{equation}
with $\fR\colon H^{s} \to H^{-s}$  satisfying the estimate
\begin{equation}\label{upper.R}
\norm{\fR  u}_{-s } \leq C_{s}  \frac{ |z_1(0)|^2 + |z_{-1}(0)|^2}{\tR^2}  \norm{u}_s    \ . 
\end{equation}
\end{itemize}
\end{lemma}
\begin{proof}
$(i)$ First note that   $\big(\tJ_1 +\fv(x)\big) \xi $ is 
 a symbol in $\Gamma^1_{W^{2, \infty}}$ with seminorm 
\begin{equation}\label{fv.sem}
| \big( \tJ_1 + \fv(x)\big)  \xi |_{1, W^{2, \infty}, 7} \leq C \left(|z_1(0)|^2 +  |z_{-1}(0)|^2 \right)  \ .
\end{equation}
We now compute the commutator between $\fA$ and $\Opbw{\big( \tJ_1+ \fv(x)\big) \xi} $.
We use the composition Theorem \ref{teoremadicomposizione}  $(i)$  regarding $\fa(x, \xi)$ as a symbol in $\Gamma^{2s+1}_{W^{2, \infty}}$ (so putting $m \leadsto    2s+1$, $m'  \leadsto  1$, $\varrho  \leadsto  2$);   we get 
\begin{equation}\label{mourre1}
\im  [\fA,   \Opbw{\big( \tJ_1 +  \fv(x) \big) \xi} ] = \Opbw{ \{ \fa(x, \xi) \,   ,  \big( \tJ_1+ \fv(x)\big) \xi  \}} + \breve R 
\end{equation}
where the operator $\breve R \colon H^{s} \to H^{ -s}$  satisfies 
\begin{align}\label{breveR}
\norm{\breve R u}_{-s}& \lesssim  | \fa |_{2s +1, W^{2,\infty}, 7}  \, 
| \big( \tJ_1 + \fv(x) \big) \xi |_{1, W^{2,\infty}, 7} \, 
\norm{u}_{s}    \stackrel{\eqref{a.sem}, \eqref{fv.sem}}{\lesssim} \frac{ |z_1(0)|^4+ |z_{-1}(0)|^4}{\tR} \norm{u}_{s} \ . 
\end{align}
Back to formula \eqref{mourre1}, the Poisson bracket  $\{ \fa(x, \xi) \,   ,  \big( \tJ_1+ \fv(x)\big) \xi  \}$ was already computed in \eqref{comm.av0}, hence 
\begin{equation}\label{mourre20}
\Opbw{ \{ \fa(x, \xi)\,   ,  \big(\tJ_1 +  \fv(x) \big) \xi  \}}    = 
\tI_1     \Opbw{  |\xi|^{2\ts} \eta^2_\tR}  +  \Opbw{a(x,\xi)}  \  
\end{equation}
with  $a(x, \xi)$ a  smooth, non-negative symbol having  the  structure  \eqref{sym.a}. Thanks to these properties we bound the operator $\Opbw{a}$ from below using the  
 strong Garding inequality \ref{garding}, getting 
\begin{equation}\label{mourre30}
\la\Opbw{a} u, u \ra   \geq - C   
\frac{\norm{a_1}_{W^{3,\infty}} + \norm{a_2}_{W^{3,\infty}}}{\tR^2} \norm{u}_s^2  \stackrel{\eqref{est.aj} }{\geq }
-C \frac{|z_1(0)|^4 + |z_{-1}(0)|^4}{\tR^2} \la \la D\ra^{2s} u, u \ra \ . 
\end{equation}
We conclude  by \eqref{mourre1}, \eqref{mourre20}, \eqref{mourre30} that 
\begin{equation}\label{mourre2}
\im  [\fA,   \Opbw{ \big( \tJ_1 + \fv(x) \big) \xi} ]   \geq   
\tI_1   \, \Opbw{  |\xi|^{2\ts} \eta^2_\tR} + \fR \ , \quad 
\fR := \breve R  -C \frac{|z_1(0)|^4 + |z_{-1}(0)|^4}{\tR^2}\la D\ra^{2s} 
\end{equation}
where 
the operator $\fR \colon H^{s} \to H^{-s }$ fulfills the  estimate \eqref{mourre.R}. 

$(ii)$ Define the positive symbol
$\tilde a(x,\xi):= \left( 2 |z_1(0)|\,|z_{-1}(0)| -\ta(x)\right)  |\xi|^{2s} \eta_\tR^2(\xi)$ and 
apply again Garding's inequality \ref{garding}.
\end{proof}

\subsection{Growth of Sobolev norms}
We now give sufficient conditions on the initial data  
of a 
 long-time controlled solution $z(t)$  ensuring  growth of Sobolev norms. 
\begin{definition}[{\bf Well-prepared data}]\label{B} 
Fix  $s, \theta$ as in \eqref{parameters}. Fix also   $\nu_0\in (0, \frac12)$, $ \e >0$. 

We say that an initial datum $z(0) \in H^s(\T, \C)$ is {\em well prepared} with parameters $(s, \theta, \nu_0, \e)$ if 
\begin{itemize} 
\item[{\rm (B1)}]  On the modes on $\Lambda$ 
\begin{equation}\label{rho1-1}
2 |z_1(0) | \, |{z_{-1}(0)} |  - \frac{|z_1(0)|^2 + |z_{-1}(0)|^2}{2}  \geq  \nu_0 \e^2  \ ;
\end{equation}
\item[{\rm (B2)}] On the modes on $\Lambda^c$
\be\label{tRbound}
 \la \fA_{s,\tR }  z^\perp(0),   z^\perp(0) \ra > \e^{3-3\theta} \ , \quad \mbox{with } 
\tR := \e^{-(3+\theta)/(1-\alpha)} \ 
\end{equation}
and $\fA_{s,\tR }$ in \eqref{fA}.
\end{itemize}
\end{definition}

\begin{remark}\label{rem:nontriv}
Condition \eqref{rho1-1} ensures that $|z_1(0) z_{-1}(0)| >0$, hence both $\fv(x)$ in \eqref{V0} and the symbol $\fa(x,\xi)$ in \eqref{fA} are non-trivial.
\end{remark}

The next result proves that a solution $z(t)$ which is long-time controlled for times $ {T_0}{\e^{-2}} \log\left({\e^{-1}} \right) $ with $T_0$ sufficiently large and whose initial datum is well-prepared, undergoes growth of Sobolev norms. 
Precisely: 
 \begin{proposition}\label{prop:instab} 
 Fix  $s, \theta$ as in \eqref{parameters}. Fix also   $\nu_0\in (0, \frac12)$.
 There exists $\e_1 = \e_1(s, \theta, \nu_0)>0$ such that for any  $\e \in (0, \e_1)$, the following holds true. 
 Let $z(t)\in H^s(\T, \C)$ be  a solution of system  \eqref{eq.ztop}--\eqref{eq.zperp}  such that 
 \begin{itemize}
 \item[(i)] it is long-time controlled with parameters $(s, \theta, T_\star, \e)$ (see Definition \ref{A}),
with 
\be\label{T0}
T_\star = \frac{T_0}{\e^2} \log\left(\frac{1}{\e} \right)  \ , \quad T_0  := \frac{1}{\nu_0}  \ ;
\ee
\item[(ii)] its initial datum $z(0)\in H^s(\T, \C)$ is well-prepared with parameters $(s, \theta, \nu_0, \e)$ (see Definition \ref{B}).
 \end{itemize}
  Then 
the solution  $z(t)$ undergoes growth of Sobolev norms, i.e. 
 \begin{equation}\label{zeta.grows}
\sup_{|t | \leq  T_\star } \norm{z(t)}_s \geq  \frac{1}{ \e^{\theta}} \ .
\end{equation}

 \end{proposition}

The first  step to prove such result is to define the $\cA$-functional 
\begin{equation}\label{cA}
\cA(t):= \langle \fA_{s,\tR} \,  \zeta(t), \zeta(t) \rangle, \quad \fA_{s,\tR}  \mbox{ in } \eqref{fA} \ ,\quad \zeta(t)  \mbox{ in } \eqref{zeta} \end{equation}
and exploit Lemma \ref{lem:pos.comm} to give a {\em lower bound} on the   time derivative  $\frac{\di}{\di t} \cA(t)$.
Precisely we have:
 \begin{lemma}\label{lem:dAdt}
Under the same assumptions of Propositon \ref{prop:instab}, there are a constant $C>0$ and $ \e_1=\e_1(s,\theta,\alpha,\nu_0)>0$ such that if $ \e \in (0,\e_1)$
 the $\cA$- functional in  \eqref{cA}, with $\tR$ in \eqref{tRbound}
  fulfills:  then  
\begin{align}\label{dAt}
\frac{\di}{\di t} \cA(t)\geq & \, {\e^2 \nu_0} \left(  \cA(t)
 -  C  \e^{3-2\theta}   \right)  \ , \qquad \forall  0\leq t \leq \frac{T_0}{\e^2}\log\left(\frac{1}{\e} \right) \ .
\end{align}
 \end{lemma}
 \begin{proof}
First note that if $z(t)$ is a long-time controlled  solution with parameters $(s, \theta, T_\star, \e)$ 
and has initial datum well prepared with parameters $(s, \theta, \nu_0, \e)$ then  the translated solution 
 $\zeta(t)$ defined in \eqref{zeta} is long-time controlled and has initial data well-prepared with the same parameters. 
 
 From now on we shall simply denote $\fA\equiv \fA_{s,\tR}$.
 Since $\zeta(t)$ fulfills \eqref{eq.zeta}, we compute 
\begin{align}
\label{I}
\frac{\di}{\di t} \cA(t)  = & \la \im [ \fA, \Opbw{\big( \tJ_1 + \fv(x) \big) \xi}]  \zeta, \zeta\ra \\
\label{II}
&  + \la \im  [ \fA,\Opbw{ \wt \tV(t;x) \xi}] \zeta, \zeta \ra  \\
\label{III}
& + \la \im [ \fA,\Opbw{ - |\xi|^\alpha + \wt \tb(t; x, \xi)}] \zeta, \zeta \ra\\
\label{IV}
& + 2\Re\,  \la \fA \wt Y(t) , \zeta \ra 
\end{align}
We shall use that, for well-prepared data, 
 the number $\tI_1$ in \eqref{tI1} fulfills (see \eqref{rho1-1})
 \be\label{lowI1}
 \tI_1 \geq 2 | z_1(0)|\, | z_{-1}(0)| \nu_0  \, \e^2,
 \ee whereas for long-time controlled solutions (see \eqref{app.ass1}), one has  
 \be \label{z1-1}
 |z_1(0)|^2+ |z_{-1}(0)|^2\leq \e^2.
 \ee
We first estimate the term \eqref{I}  from below using Lemma \ref{lem:pos.comm}.
  Precisely we get 
\begin{align}
\notag
 \la \im [ \fA, \Opbw{(\tJ_1+\fv(x)) \xi}]  \zeta, &\zeta \ra 
  \stackrel{\eqref{mourre},\eqref{mourre.R}}{\geq }  \tI_1 \,  \,  \la  \Opbw{  |\xi|^{2 s} \eta^2_\tR(\xi)} \zeta, \zeta \ra - C_s \frac{\e^{4}}{\tR} \norm{\zeta}_s^2\\
 &\stackrel{\eqref{lowI1}}{\geq}  2 | z_1(0)|\, | z_{-1}(0)| \nu_0  \, \e^2   \,  \la  \Opbw{  |\xi|^{2 s} \eta^2_\tR(\xi)} \zeta, \zeta \ra - C_s \frac{\e^{4}}{\tR} \norm{\zeta}_s^2 \\
 \label{1259}
& \stackrel{\eqref{upper.bound}, \eqref{upper.R}}{  \geq }
{\nu_0}  \, \e^2  \cA(t)  - C_s \frac{\e^{4}}{\tR}\norm{\zeta}_s^2   \ . 
\end{align}
Next  we estimate \eqref{II} from above. We first use  estimate  \eqref{est.aR.comp} (with $\nu =0$, $m' = 1$, $m = 2s$),
\begin{equation}\label{II.est}
\abs{\eqref{II}} \leq | \fa |_{2s, W^{2,\infty}, 7} \, |\wt \tV(t,\cdot) |_{1, W^{2,\infty}, 7} \, \norm{\zeta}_s^2
\stackrel{\eqref{a.sem}, \eqref{est.706}, \eqref{z1-1}}{\leq }
C_s \e^{6-\theta}\norm{\zeta}_s^2 \ . 
\end{equation}
Next we estimate \eqref{III} from above.  We use again estimate  \eqref{est.aR.comp} (this time with $\nu = 1-\alpha$, $m' = \alpha$, $m = 2s$, thinking  $\fa(x,\xi)$ as a symbol in $\Gamma^{2s + 1-\alpha}_{W^{2, \infty}}$  supported on high frequencies) to bound
\begin{equation}\label{III.est}
\abs{\eqref{III}} \leq \frac{1}{\tR^{1-\alpha}} \, | \fa |_{2s, W^{2,\infty}, 7} \  | |\xi|^\alpha + \wt\tb(t, \cdot) |_{\alpha, W^{2,\infty}, 7} \, \norm{\zeta}_s^2
\stackrel{\eqref{a.sem}, \eqref{est.706} }{\leq}
C_s  \frac{\e^{2}}{\tR^{1-\alpha}}\norm{\zeta}_s^2 \ . 
\end{equation}
Finally we estimate \eqref{IV} from above.
We use estimate \eqref{cont00} 
to bound
\begin{align}
\abs{\eqref{IV}}  & \leq 
\norm{\fA \wt Y(t)}_{-s} \norm{\zeta}_{s} 
{\leq}
C_s | \fa |_{2s, L^{\infty}, 7} \norm{ \wt Y(t)}_{s}\norm{\zeta}_{s} 
\stackrel{\eqref{a.sem}, \eqref{est.706}}{\leq}
C_s  \e^{5-\theta} \norm{\zeta}_{s} \ . 
\label{IV.est}
\end{align}
Then \eqref{dAt} follows from \eqref{1259}, 
 \eqref{II.est},  \eqref{III.est} and  \eqref{IV.est}, choosing $\tR$ as in \eqref{tRbound}, and using that $\zeta(t)$, being long-time controlled, fulfills $\norm{\zeta(t)}_s \leq \e^{-\theta}$ and provided $ \e$ is sufficiently small.
 \end{proof}

We are finally able to prove Proposition \ref{prop:instab}. 
 
 \begin{proof}[Proof of Proposition \ref{prop:instab}]
Let $z(t)\in H^s(\T, \C)$ be  a solution of system  \eqref{eq.ztop}--\eqref{eq.zperp}  
whose initial datum $z(0)\in H^s(\T;\C)$ is well-prepared 
with parameters $(s, \theta, \nu_0,  \e)$ and which 
 is long-time controlled with parameters 
$(s, \theta, T_\star, \e)$, $T_\star$ in \eqref{T0}. 
By Lemma \ref{lem:dAdt}, provided $\e>0$  is sufficiently small, the functional $\cA(t)$  in \eqref{cA} fulfills the inequality \eqref{dAt}. Integrating in time, we get 
\begin{equation}
\label{A(t).lower}
\cA(t) \geq e^{ \nu_0 \e^2 t} \, \left( \cA(0) - C \e^{3-2\theta} \right) + C \e^{3-{2\theta}}  \ , \qquad 0\leq t \leq \frac{T_0}{\e^2}\log\left(\frac{1}{\e} \right) \ .
\end{equation}
A sufficient  condition for $\cA(t)$ to grow in time is that $\cA(0) >  C \e^{3-2\theta}$; this condition is  fulfilled for well-prepared initial data provided $\e$ is sufficiently small; indeed by \eqref{tRbound}
$$\cA(0) = \la \fA \zeta(0), \zeta(0) \ra = \la \fA  z^\perp(0),  z^\perp(0) \ra  >
\e^{3-3 \theta} >
 2 C \e^{3-2\theta}  \ . $$
{Then, using also the penultimate of the above inequalities,  $\cA(0) - C \e^{3-2\theta} >\e^{3-3 \theta}-  C \e^{3-2\theta} > \frac12\e^{3-3\theta}$, and we get from \eqref{A(t).lower}, the definition \eqref{cA} and the continuity Theorem \ref{thm:contS} 
\begin{equation}
\begin{aligned}
\frac{1}{2} \e^{3-3\theta} \, e^{{\nu_0}\e^2 t} \leq \cA(t)   \leq  \norm{\fA_{s,\tR}\zeta(t)}_{-s} \norm{\zeta(t)}_s 
\stackrel{\eqref{fA}, \eqref{a.sem}}{\leq} C_s \e^2 \norm{\zeta(t)}_s^2 \stackrel{\eqref{zetaez}}{\leq} C_s \e^2 \norm{z(t)}_s^2 
\end{aligned}
\end{equation}
for some $C_s >1$.}
Hence, when $t = \frac{T_0}{\e^2} \log\left(\frac{1}{\e}\right)$, eventually shrinking $\e$, one gets 
\begin{equation}
\label{}
 \norm{z(t)}_s^2 \geq 
\frac{1}{2 C_s} \e^{1-3\theta} \, e^{ \nu_0 T_0 \log(\e^{-1})} \stackrel{\eqref{T0}}{\geq}
 \frac{1}{\e^{2\theta}}
\end{equation}
yielding \eqref{zeta.grows}.
 \end{proof}

 \subsection{Conclusion and proof of Theorem \ref{thm:main}}
 Fix  $s, \theta$ as in \eqref{parameters}.
 We give now an example of a  well-prepared initial data.  
 \begin{lemma}\label{lem:z0.wp}
Let $\rho_1, \rho_{-1}>0$ in the non-empty region limited by
\be\label{realz0}
\rho_1^2 + \rho_{-1}^2 \leq 1 \ , \qquad 
 \nu_0 := 2 \rho_1 \rho_{-1}  - \frac{\rho_1^2 + \rho_{-1}^2}{2}  > 0    \ .
\ee
There exists $\e_0 >0$ and, for any  $\e \in (0, \e_0)$, an interval $I(\e)$ such that  the initial datum 
 \be\label{z0}
 z(0):= \e \rho_1 e^{\im x }  + \e \rho_{-1} e^{-\im x} + \rho\, e^{\im 3\tN x} + \im \rho \,  e^{\im (3\tN + 2) x} \ , \quad \tN := \lceil \tR  \rceil
 \ee
with  $\tR =   \e^{-(3+\theta)/(1-\alpha)}  $  and $\rho \in I(\e)$, fulfills:
\begin{itemize}
	\item {\bf well-prepared:} $z(0)$ in \eqref{z0} is a well-prepared initial datum with parameters $(s, \theta, \nu_0, \e)$ (according to Definition  \ref{B});
	\item  {\bf $L^2$-smallness:} the bounds in  \eqref{app.ass1}  holds true; 
	\item {\bf $H^s$-smallness:} $z(0)$ fulfills the high norm bound 
	\be\label{z0.s}
	\norm{z(0)}_s \leq \e^\theta  \ . 
	\ee
\end{itemize}  
 \end{lemma}
 \begin{proof}
We first prove that each of the three claimed properties gives a restriction on the choice of 
 $ \rho$. Then we prove that such conditions are compatible.

  {\bf \underline{Well-prepared}:}
 Condition (B1) follows immediately from \eqref{realz0}. We now check condition (B2).
 Using the definition of paradifferential operator in 
\eqref{BW}, the form of $\fA$ in \eqref{fA} and of $z(0)$ in \eqref{z0},  we get 
 \begin{align*}
 \la \fA \Pi^\perp z(0), \Pi^\perp z(0) \ra & = 
 \sum_{k } {\e^2 \rho_{1} \rho_{-1}} \, \abs{k+1}^{2s} \, \eta_\tR^2\big(k+1\big) \, \chi_2 \big(1, - 1, k+1 \big) \,  \Im \big(\bar z_{k}^\perp(0) \, z_{k+2}^\perp(0)\big) \\
 & =  {\e^2 \rho_{1} \rho_{-1}} \, \abs{3\tN+1}^{2s} \, \underbrace{\eta_\tR^2\big(3\tN+1\big)}_{=1} \, \underbrace{\chi_2 \big(1, -1, 3\tN+1 \big)}_{=1} \,  \rho^2  = 
 \e^2 \rho_{1} \rho_{-1} \abs{3\tN+1}^{2s} \,  \rho^2 \ .
 \end{align*}
 Then \eqref{tRbound} is fulfilled provided $\rho_{1} \rho_{-1} 3^{2s} \tR^{2s} \rho^2 \geq \e^{1-3\theta}$, which using \eqref{tRbound} gives 
\be \label{prima}
\rho \geq\,  \frac{\e^{\frac12-\frac32 \theta + s\frac{3+\theta}{1-\alpha}}}{3^s \sqrt{\rho_{1} \rho_{-1}}}  .
\ee
This proves that $z(0)$ is well prepared.

 {\bf \underline{$L^2$-smallness}:} The first condition in \eqref{app.ass1} is satisfied thanks to the first assumption in \eqref{realz0} and the second condition in \eqref{app.ass1} is satisfied provided that 
\be \label{seconda}
\rho \leq \frac{\e^3}{\sqrt{2}}.
\ee 
\indent {\bf \underline{$H^s$-smallness}:} The condition \eqref{z0.s} is satisfied provided that 
$$(\rho_1^2+\rho_{-1}^2)\e^2\leq \frac{\e^{2\theta}}{2} \quad \text{and} \quad     \rho^2 (3\tN+1)^{2s}+  \rho^2 (3\tN+3)^{2s}\leq \frac{\e^{2\theta}}{2}.$$ 
The first condition follows automatically from \eqref{realz0} and taking $\e$ sufficiently small, while the second one, using $\tN \leq \tR +1$ and \eqref{tRbound}, is fulfilled for example for 
\be \label{terzera}
 \rho \leq \frac{\e^{\theta+s\frac{3+\theta}{1-\alpha}}}{6^{s} 2 }.
 \ee
 Note also that, since $ s\geq 3\fs_0 \geq 1$, for $\e$ small enough the second condition \eqref{seconda} is less restrictive than the third one \eqref{terzera}. 
 Note that, provided  $\e$ is small enough and using   $\theta<\frac15$,  
 conditions \eqref{prima} and \eqref{terzera} are compatible.  Then,  taking 
 $$
 \rho \in I(\e):=\Big( \frac{1}{3^s \sqrt{\rho_{1} \rho_{-1}}} \e^{\frac12-\frac32 \theta + s\frac{3+\theta}{1-\alpha} } ,\,  \frac{\e^{\theta+s\frac{3+\theta}{1-\alpha}}}{6^{s} 2 } \Big),
 $$
the datum  $z(0)$ satisfies all the claimed conditions.
 \end{proof}

 We now show that any solution of system \eqref{Z.eq} with a well prepared initial datum as in 
 Lemma \ref{lem:z0.wp} undergoes Sobolev norm explosion. Precisely we have:
 \begin{lemma}\label{lem:cresce}
  Fix  $s, \theta$ as in \eqref{parameters}. 
 There exists $\e_2 >0$ such that, provided $\e \in (0, \e_2)$ the following holds true. Let $z(0) \in H^s(\T, \C)$ as in  Lemma \ref{lem:z0.wp} and so well-prepared with parameters $(s, \theta, \nu_0, \e)$, for some $\nu_0 \in (0, \tfrac12)$.  
 Consider the solution $z(t)$ of system  \eqref{eq.ztop}--\eqref{eq.zperp}  with initial datum $z(0)$.
 Denote by  
 \be\label{T1}
0<T_1:= T_1(\e; z(0)) := \inf\left\lbrace t \geq 0 \colon \ \ \norm{z(t)}_s \geq \e^{-\theta} \right\rbrace \ .
 \ee
 Then $T_1$ is finite and bounded by $T_1 \leq \frac{T_0}{\e^2}\log\left(\frac{1}{\e} \right)$, $T_0 = \nu_0^{-1}$. 
 Moreover one has
 \be\label{cresce}
\sup_{0 \leq t \leq T_1}  \norm{z(t)}_{\fs_0} \leq 3\e \ , 
 \quad  \norm{z(0)}_s \leq \e^\theta \ , \quad
\norm{z(T_1)}_s \geq \e^{-\theta} \ .
\ee
 \end{lemma}
 \begin{proof}
 Define $\e_2:= \min(\e_\star, \e_0, \e_1,\fr)$ with
$\e_\star$ of Lemma \ref{lem:boot},  
  $\e_0$ of Lemma \ref{lem:z0.wp} and $\e_1$ of Proposition \ref{prop:instab}.
First note that the solution  $z(t)$ is long-time controlled with parameters $\big(s, \theta, T_1, \e \big)$ (see Definition \ref{A}); indeed condition (A1)  holds true in view of the $L^2$-smallness of Lemma \ref{lem:z0.wp}, whereas condition (A2)  holds true with $T_\star  \leadsto  T_1$ by the minimality of $T_1$. \\
We now show that  $T_1$ is finite and  bounded by $\frac{T_0}{\e^2}\log\left(\frac{1}{\e} \right)$.
 Assume by contradiction that $T_1 > {T_0}{\e^{-2}}\log\left(\e^{-1} \right)$. Then, by the very definition of $T_1$, 
 $$
\sup_{0 \leq t \leq {T_0}{\e^{-2}}\log\left(\e^{-1} \right) } \norm{z(t)}_s \leq \e^{-\theta} \ ,
 $$
 namely the solution $z(t)$ is long-time controlled also with parameters $\big(s, \theta, \frac{T_0}{\e^2}\log\left(\frac{1}{\e} \right), \e \big)$. 
 Then, since by Lemma \ref{lem:z0.wp} the initial data $z(0)$ is well prepared, Proposition \ref{prop:instab} applies and therefore
 $$
 \sup_{0 \leq t \leq  {T_0}{\e^{-2}}\log\left(\e^{-1} \right) } \norm{z(t)}_s \geq \e^{-\theta} \ ,
 $$
 contradicting the minimality of $T_1$. This proves that $T_1 \leq \frac{T_0}{\e^2}\log\left(\frac{1}{\e} \right)$. \\
To control the low norm $\norm{z(t)}_{\fs_0}$, we apply the 
 bootstrap lemma \ref{lem:boot}  with the parameter $T_\star =T_1$ that we have just proved satisfy the required condition \eqref{Tstar.T0}.
 The last two inequalities of \eqref{cresce} follow by  \eqref{z0.s} and \eqref{T1}.
 \end{proof}

 We conclude with:
 \begin{proof}[Proof of Theorem \ref{thm:main}]
 Recall that the variables $u(t)$ and $z(t)$ are related by  the admissible transformation $Z(t) =\cF(U(t))\equiv \bF(U(t))U(t)$ in \eqref{Z=FU}. 
  By Lemma \ref{loc.inv}, the map $Z=\cF(U)$ is locally invertible provided  $\norm{Z}_{\fs_0}\leq r'$ is   sufficiently small, 
 and has the form
 $\cF^{-1}(Z)= \bG(Z)Z$ for some $\bG(Z)$ fulfilling the bound in \eqref{lin.est.F}.

  So consider  $Z(0) = \vect{z(0)}{\bar z(0)}$ with $z(0)$ as in Lemma \ref{lem:z0.wp} and therefore fulfilling  $\norm{Z(0)}_{\fs_0} \leq \e^\theta \leq \fr$. We  define
 $$
 U(0) := \cF^{-1}(Z(0))=\bG(Z(0))Z(0) \ .
  $$
We take  $U(0)$ as the initial data for equation \eqref{eq:main};
by \eqref{stima.inv.adm}, its Sobolev norm 
 \be
 \norm{U(0)}_s \leq C_s \norm{Z(0)}_s \stackrel{\eqref{cresce}}{\leq} C_s \e^\theta   \ . 
 \ee
 Consider now the solution $U(t)$ of \eqref{eq:main} with initial data $U(0)$. 
 By Theorem \ref{thm:nf},  $Z(t) =\cF(U(t))$  is the solution of equation \eqref{Z.eq} with initial datum $Z(0)$ of Lemma \ref{lem:z0.wp}; consequently,  in view of Lemma \ref{lem:sys} and Lemma \ref{lem:cresce}, $z(t)$ has a small $H^{\fs_0}$-norm  for all times $0 \leq t \leq T_1$, but large $H^s$-norm  at time $T_1$. 
 We deduce that 
 $ U(t)  = \cF^{-1}(Z(t))$ fulfills the bound
 $$
 \norm{U(t)}_{\fs_0} \leq C_{\fs_0} \norm{Z(t)}_{\fs_0} \leq C_{\fs_0} \e < \fr , \quad \forall  0 \leq t \leq T_1 \ .  
 $$ 
At time $T_1$, 
 we bound from below 
 the $H^s$-norm of $U(T_1)$  using the identity $Z(T_1) = \cF(U(T_1))$, the fact that 
 $\norm{U(T_1)}_{\fs_0} \leq \fr $ and 
   estimate  \eqref{lin.est.F}, to get 
 \be
 \norm{U(T_1)}_s \geq  C_s^{-1} \norm{Z(T_1)}_s  \stackrel{\eqref{cresce}}{\geq }  C_s^{-1}  \e^{-\theta}   \ .
 \ee
Given arbitrary $\delta \in (0,1)$ and $K \geq  1$, shrink  $\e$ to conclude the proof of  Theorem \ref{thm:main}.
 \end{proof}

\smallskip

\appendix 
 
\section{High frequency paradifferential calculus}
In this section we consider paradifferential operators with symbols supported only on high frequencies and prove a commutator estimate and a Garding inequality keeping track of the size of the support of the symbols.

 \begin{lemma}\label{lem:cutoff}
 Let $N \in \N_0$, $m \in \R$ and $\tR\geq 1$. If  $a \in \Gamma^m_{W^{N, \infty}}$, then 
 $$
 a_\tR(x, \xi):= a(x, \xi)\,  \eta_\tR(\xi) , \quad \eta_\tR \mbox{ in }  \eqref{etaR}$$ is a symbol in $\Gamma^{m+\nu}_{W^{N, \infty}}$ for any $\nu \geq 0$ with quantitative bound
 \begin{equation}\label{a.etaR.sem}
| a_\tR |_{m + \nu, W^{N, \infty}, n} \leq C_{n}\,  \tR^{-\nu} \, | a |_{m , W^{N, \infty}, n} \quad \text{for any } n\in \N_0 . 
\end{equation}
In addition, if $N \geq 2$ and  $b \in \Gamma^{m'}_{W^{2, \infty}}$, $m' \in \R$, one has the commutator estimate
\begin{equation}\label{est.aR.comp}
\norm{ [ \Opbw{ a_\tR}, \Opbw{b} ] u }_{s - m - m' - \nu +1 } \leq C 
\tR^{- \nu} \,  | a |_{m , W^{2, \infty}, 7}  \, | b |_{m' , W^{2, \infty}, 7} \, \norm{u}_s  \ .
\end{equation}
 \end{lemma}
 \begin{proof}
For any $\alpha, \beta \in \N_0$, $\alpha \leq N$, $\beta \leq n$, we have 
 \begin{align*}
\abs{( \pa_x^\alpha  \pa_\xi^\beta a_\tR(x, \xi) }
&  \lesssim  \sum_{\beta_1 + \beta_2 = \beta} 
 \abs{\pa_x^\alpha \pa_\xi^{\beta_1} a(x, \xi)}  \, \abs{\pa_\xi^{\beta_2} \eta_\tR (\xi)} \\
 & \lesssim
 \sum_{\beta_1 + \beta_2 = \beta}  |a|_{m, W^{N,\infty}, n}\la \xi \ra^{m - \beta_1} \frac{1}{\tR^{\beta_2}}  \abs{\eta^{(\beta_2)}\big(\frac{\xi}{\tR}\big) } \\
 & \lesssim |a|_{m, W^{N,\infty}, n}\sum_{\beta_1 + \beta_2 = \beta} \la \xi \ra^{m - \beta_1- \beta_2 + \nu } \sup_{\xi} \abs{\la \xi \ra^{-\nu}
    \la \frac{\xi}{\tR} \ra^{\beta_2} \eta^{(\beta_2)}\big(\frac{\xi}{\tR}\big) }  \\
 & \lesssim \frac{1}{\tR^{\nu}} |a|_{m, W^{N,\infty}, n}  \, \la \xi \ra^{m - \beta + \nu}\ ,
 \end{align*}
 where in the last step we used that the function $ \la \frac{\xi}{\tR} \ra^{\beta_2} \eta^{(\beta_2)}\big(\frac{\xi}{\tR}\big)$ is uniformly bounded on $\R$ and has support on $\xi \geq \tR$.
 
 We prove now \eqref{est.aR.comp}.  By {Proposition \ref{teoremadicomposizione}} with $\varrho =2$ we have
 $$
 [ \Opbw{ a_\tR}, \Opbw{b} ]  = \Opbw{ \{a_\tR, b \} } + R^{-2}(a_\tR, b).
 $$
 We now bound both terms in the above equation regarding $a_\tR$ as a symbol in $\Gamma^{m+\nu}_{W^{N,\infty}}$ and 
  $ \{a_\tR, b \}$ as a symbol in $ \Gamma^{m + m'+\nu-1}_{W^{N-1,\infty}}$.  By \eqref{cont00} and \eqref{poi.est}, we get 
\begin{align}
\notag
\norm{\Opbw{ \{a_\tR, b \} } u }_{s- m - m' -\nu +1}
& \lesssim \abs{ \{a_\tR, b \}}_{m+m'+\nu-1, L^\infty, 4} \ \norm{u}_s \\
& \lesssim
|  a_\tR|_{m+ \nu , W^{1,\infty}, 5} | b |_{ m', W^{1,\infty}, 5}   \, \norm{u}_s \\
& \stackrel{\eqref{a.etaR.sem}}{\lesssim}
\tR^{- \nu} \,  | a |_{m , W^{1,\infty}, 5}  \, | b |_{m' , W^{1,\infty}, 5} \, \norm{u}_s  \ .
\label{est.aR.comp1}
\end{align}
Next we estimate the norm of $R^{-2}(a_\tR, b)$ using \eqref{comp020}:
\begin{align}
\notag
\norm{R^{-2}(a_\tR, b)u }_{s-m - m' -\nu +2}
&  \lesssim
| a_\tR |_{m + \nu, W^{2,\infty}, 7}  \, | b|_{m', W^{2,\infty}, 7}  \, \norm{u}_s \\
\label{est.aR.comp2} 
&\stackrel{\eqref{a.etaR.sem}}{\lesssim}
\tR^{- \nu} \,  | a |_{m , W^{2,\infty}, 7}  \, | b |_{m' , W^{2,\infty}, 7} \, \norm{u}_s 
\end{align}
In conclusion \eqref{est.aR.comp} follows from \eqref{est.aR.comp1}, \eqref{est.aR.comp2}.
 \end{proof}
 
In the following we shall  use a  well-known cancellation which is a direct consequence of Proposition \ref{teoremadicomposizione}: if 
$a\in \Gamma^{m}_{W^{2,\infty}}$,  $b\in \Gamma^{m'}_{W^{2,\infty}}$, with $m, m' \in \R$, then  
\begin{equation}\label{autoag}
\Opbw{b}\circ \Opbw{a}\circ \Opbw{ b}= \Opbw{ab^2} + R^{-2}(a,b),  
\end{equation}
where $R^{-2}(a,b)$  is a bounded operator  $H^s \to  H^{s-(m+2m')+2}$, $\, \forall s\in \R$, satisfying, for any 
$ u \in H^s  $,
\begin{equation}\label{restoR0}
\|R^{-2}(a,b)\|_{s-(m+m') + 2} \lesssim  |a|_{m,W^{2,\infty},8} \ 
|b|_{m',W^{2,\infty},8}^2
 \  \| u \|_s.
\end{equation}

 In the next lemma we prove a simplified version of the strong Garding inequality adapted to our setting. 
  \begin{lemma}[\bf Strong Garding's inequality]
\label{garding}
Let $\tR \geq 1$, $a(x) \in W^{3,\infty}$ and $a(x) \geq 0$. Let $\psi(\xi) \in \wt\Gamma^{m}_0$, $m > 0$,  a real valued Fourier multiplier  with $\textup{supp } \psi \subseteq  [\tR, +\infty)$. Then there is $C >0$ such that 
\begin{equation}
\label{gard1}
\la \Opbw{a(x) \psi^2(\xi)} u , u \ra \geq - C \frac{\norm{a}_{W^{3,\infty}}}{\tR^2} \norm{u}_{m}^2  \ .
\end{equation}
\end{lemma}
\begin{proof}
Arguing as in 
Lemma \ref{lem:cutoff} one shows that, for any $n \in \N_0$, 
\begin{equation}\label{psi.m+1}
|\psi|_{m+1, L^\infty, n} \leq C_n \frac{1}{\tR} |\psi|_{m, L^\infty, n} \ .
\end{equation}
We apply now the composition formula \eqref{autoag} regarding $\psi(\xi)$ as a symbol in $ \wt\Gamma^{m+1}_0$: 
\be \label{positive1}
\Opbw{\psi} \circ \Opbw{a} \circ \Opbw{\psi} = \Opbw{a \psi^2} + \fR_1
\ee
with $\fR_1\colon H^{m} \to H^{-m} $ fulfilling, by \eqref{restoR0},
\be\label{R1.estA}
\norm{\fR_1 u}_{-m} {\lesssim }
\norm{a}_{W^{2, \infty}} \, |\psi|_{m+1, L^\infty, 8}^2 \norm{u}_m
\stackrel{\eqref{psi.m+1}}{\lesssim} \frac{1}{\tR^2 }\norm{a}_{W^{2,\infty}}\norm{u}_m \ .
\ee
Then observe that $\Opbw{\psi} = \Opw{\psi} = \psi(D)$ and $\Opbw{a} = 
\Opw{a} + \Opw{a_\chi - a}$, where $a_\chi$ is the cut-offed symbol defined  in \eqref{regula12}, so 
\be \label{positive2}
\Opbw{\psi} \circ \Opbw{a} \circ \Opbw{\psi} = 
\psi(D)  \circ \Opw{a} \circ \psi(D) + \fR_2
\ee
where $\fR_2:= \psi(D)  \circ \Opw{a_\chi -a} \circ \psi(D)  \ . $
Now we prove that $\fR_2$ is bounded $H^{m}\to H^{-m}$.
First note that, 
 by the definitions \eqref{regula12} and \eqref{BWnon}, for any $ v\in H^{-1}$, 
\begin{align*}
\norm{\Opw{a_\chi -a} v}_{1}^2 & \lesssim  \sum_{j} \la j \ra^2 \abs{\sum_k \widehat a_{j-k} \, \Big(1- \chi\big(k-j, \frac{j+k}{2} \big)  \Big) \,  v_k}^2 \\
& \lesssim \sum_{j}  \left|\sum_k  \la j-k \ra^2  \, |\widehat a_{j-k}|  \, \Big(1- \chi\big(k-j, \frac{j+k}{2} \big) \Big)\, \frac{1}{\la k \ra}| v_k| \right|^2\\
& \lesssim \sum_{j}  \abs{\sum_k  \la j-k \ra^2 | \widehat a_{j-k} | \, \frac{1}{\la k \ra}|  v_k| }^2\\
& \lesssim 
\norm{a}_{3}^2 \norm{v}_{-1}^2 \lesssim 
\norm{a}_{W^{3,\infty}}^2 \, \norm{v}_{-1}^2
\end{align*}
where to pass from the first to the second line we used that, on the support  of $1- \chi\big(k-j, \frac{j+k}{2} \big)$,  one has 
$$\la k \ra , \la j \ra \lesssim \la j - k \ra + \la j+k \ra \lesssim \la j - k \ra  \ , $$
and to pass from the third to the last line we used Young's inequality for convolution of sequences.\\
Thus we get, for any $u \in H^m$,
\begin{align}
\norm{\fR_2 u}_{-m}
&  \lesssim 
|\psi|_{m+1,L^\infty,0} \, \norm{\Opw{a_\chi -a} \circ \psi(D)  u}_{1}\notag\\
& \lesssim |\psi|_{m+1,L^\infty,0} \, \norm{a}_{W^{3,\infty}} \, \norm{\psi(D)  u}_{-1} \notag\\
& \lesssim |\psi|_{m+1,L^\infty,0}^2 \, \norm{a}_{W^{3,\infty}} \, \norm{u}_{m}  
\stackrel{\eqref{psi.m+1}}{\lesssim} \frac{1}{\tR^2} \norm{a}_{W^{3,\infty}} \norm{u}_m \ . \label{R2.estA}
\end{align}
In conclusion, combining \eqref{positive1} and \eqref{positive2} and  since $\Opw{a} = a \geq 0$ and $\psi(D)$ is self-adjoint,   we have that
$$
0 \leq \la \psi(D) \circ a \circ \psi(D) u, u \ra = \la \Opbw{a \psi^2} u, u \ra + \la (\fR_1 - \fR_2) u, u \ra
$$
and \eqref{gard1} follows by \eqref{R1.estA} and \eqref{R2.estA}.
\end{proof}

\section{Flows and conjugations}\label{sec:flows}
In this section we collect some  results about the conjugation of 
paradifferential operators and smoothing remainders under flows, following  \cite{BD,BFP,BMM, MMS}. 

\smallskip

\noindent{\bf Conjugation by a flow generated by a real symbol of order one.} \label{sec:flow-1}
Given a function  $\beta\in \wt \mF_2^\R$ gauge invariant, i.e. $\beta(\tg_\theta U;\cdot) = \beta(U;\cdot)$ for any $\theta \in \T$,  consider the flow  $\Phi^\tau(u)$, $\tau\in[-1,1]$
defined by \eqref{flusso1para}.
It is standard (see e.g.  Lemma 3.22 in \cite{BD}) that, for any  
	$U \in B_{s_0,\R}(r)$ with $s_0>0$ sufficiently large and $r>0$ sufficiently small, 
	the operator 
	$ \Phi^\tau (U)\in\cL\pare{H^s(\T, \C^2)} $ for any $s \in \R$ with the quantitative estimate:  there is a constant $C(s)>0$ such that	for any $ W\in H^{s}(\T, \C^2)$,
	$\norm{ \Phi^\tau (U) W }_{s}+\norm{ \Phi^\tau (U)^{-1} W }_{s}
		 \leq   C \pare{ s }  \| W \|_{s}.$ 
\smallskip 
Following \cite{BD}, we define  the  path of diffeomorphism of $ \T $ via 
\begin{equation}
	\label{eq:diffeo}
	\Psi \pare{ U, \tau ;x } \defeq  x + \tau \beta\pare{U; x}
	 \quad \text{with inverse}
	 \quad 
	 \Psi^{-1}\pare{ U, \tau;y }  \defeq y + \breve{\beta}\pare{U, \tau;y} 
	, \quad 
	  \   \breve{\beta}\in  \cF_{\geq 2}^\R\bra{r} 
\end{equation}
and set 
$\Psi(U;x)\defeq\Psi (U, 1;x)$.

\begin{proposition}[Conjugations for a transport flow]
	\label{prop:Egorov}
	Let $m\in \R$, $\varrho>0$,and let $\Phi(U)$ be the flow generated by \eqref{flusso1para}. 
	\begin{enumerate}
		
		\item \label{item:Egorovi} {\bf Space conjugation of a para-differential operator:} Let $a \in  \Sigma\Gamma^{m}_2\bra{r}$ be a real symbol and
		$a^{\pare{m}}\pare{U;x, \xi} \defeq \left.  a\pare{U; y, \xi \ \partial_y\Psi^{-1}\pare{U;y}}\right|_{y=\Psi\pare{U;x}}
			\in  \Sigma\Gamma^{m}_2\bra{r}$.
		Then
		\begin{equation}
			\label{eq:conj_generic_symbol}
			\begin{aligned}
				\Phi(U) \circ \vOpbw{a\pare{U;x,\xi}} \circ \Phi(U)^{- 1} = & \  \vOpbw{a^{\pare{m}}\pare{U;x, \xi} 
					+ a^{\pare{m-2}}_{\geq 4 }  \pare{U;x, \xi}
				} + R_{\geq 4}(U)
				\\
				= & \ \vOpbw{a\pare{U;x,\xi} 
					+ a^{\pare{m}}_{\geq 4 }\pare{U;x, \xi}} + R_{\geq 4}(U),
			\end{aligned} 
		\end{equation} 
		where 
		 $a^{\pare{m-2}}_{\geq 4 }\pare{U;x,\xi}$ and {$ a^{\pare{m}}_{\geq 4}\pare{U;x,\xi}$} are  non-homogeneous real symbols  in $ \Gamma^{m-2}_{\geq 4}\bra{r} $respectively  $\Gamma^m_{\geq 4}\bra{r}$, whereas $ R_{\geq 4}(U)$ is a real-to-real  matrix of smoothing operators in $ \cR^{-\vr+m }_{\geq 4}\bra{r} $.
		In addition if $ a\pare{U;x,\xi} = V\pare{U;x} \xi $ for some $ V\in \wt \cF^\R_2\bra{r} $  then in \eqref{eq:conj_generic_symbol} $ a^{\pare{ m-2 }}_{\geq 4 }\equiv 0 $   and $ a^{\pare{ m }}_{\geq 4} \pare{U;x, \xi} = {V}'_{\geq 4} \pare{U;x}  \xi $ for a suitable function $ {V}'_{\geq 4}\in \cF_{\geq 4}^\R\bra{r} $.

\item {\bf Space conjugation of a Fourier multiplier} Let $\omega(\xi) \in \wt\Gamma^\alpha_0$ be a real Fourier multiplier. 
		Then
		\begin{equation}
			\label{eq:conj_fou}
			\begin{aligned}
				\Phi(U) \circ \vOpbw{\im \omega} \circ \Phi(U)^{- 1} 
				= & \ \vOpbw{\im \big( \omega  
					+ a^{\pare{\alpha}}_{2 }\pare{U;x, \xi} + a^{\pare{\alpha}}_{\geq 4 }\pare{U;x, \xi} + a_{\geq 4}^{(\alpha -2)}(U; x, \xi)\big)} \\
					& + R_2(U) + R_{\geq 4}(U),
			\end{aligned} 
		\end{equation} 
		where
  
  \noindent $ \bullet $ $a^{\pare{\alpha }}_{2 }\pare{U;x, \xi}$ is a real, zero-average, gauge invariant symbol in $\wt\Gamma^\alpha_2$;
  
			\noindent $ \bullet $
		 $a^{\pare{\alpha}}_{\geq 4 }\pare{U;x, \xi}$ is  a real non-homogeneous symbol in $\Gamma^\alpha_{\geq 4}[r]$ and $a_{\geq 4}^{(\alpha-2)}(U; x, \xi)$ is a non-homogeneous symbol in $\Gamma^{\alpha-2}_{\geq 4}[r]$;
		
   \noindent $ \bullet $ $R_2(U)$ is a real-to-real, gauge invariant  matrix of smoothing operators in $\wt\cR^{-\varrho+m}_2$, and $R_{\geq 4}(U)$ is a real-to-real matrix of non-homogeneous smoothing operators in $\cR^{-\varrho+m}_{\geq 4}$.
		
		\item{\bf Space conjugation of a smoothing remainder:} If $ R_2(U)$ is a real-to-real matrix of  smoothing operators in $ \widetilde \cR^{-\vr}_2\bra{r} $ then 
		\begin{equation*}
			\Phi \pare{U} \circ R_2 \pare{U} \circ \Phi\pare{U}^{-1} = R_2(U)+ R_{\geq 4 }\pare{U},
		\end{equation*}
		where $ R_{\geq 4}(U)$ is a real-to-real matrix of smoothing operators in $\cR^{-\vr+1}_{\geq 4}\bra{r} $. 
		
		\item {\bf Conjugation of $\partial_t$:} If $ U $ is a solution of \eqref{eq.U} then
		\begin{equation}\label{detconju}		
			(\partial_t \Phi\pare{U})\,  \Phi\pare{U}^{-1} =  \ii \  \vOpbw{  2 \beta  \pare{ -\ii   \vOmega (D)U,U; x }\ \xi+ \ii \ \fV_{\geq 4}\pare{ U;x } \xi } + R_{\geq 4} \pare{U}, 
		\end{equation}
		where  $ \vOmega (D)$  is the matrix of  real Fourier multipliers  in \eqref{U.start}, $ \fV_{\geq 4}\pare{U;x}$  is a real function in $\cF_{\geq 4}^\R\bra{r}$ and $ R_{\geq 4}\pare{U}$ is a real-to-real matrix of smoothing operators in $  \cR^{-\vr}_{\geq 4}\bra{r}$.

	\end{enumerate} 
 \end{proposition}
	\begin{proof}
 During the proof we shall denote $b:= \frac{\beta}{1+\tau \beta_x}$.
 
\noindent 	1.  Follows by Lemmas A4 and A5 in \cite{BFP}. 
			
\noindent 2. 	We first define the operator 
			$
			P^\tau(U):= \Phi^\tau \pare{U} \circ \vOpbw{\im\omega}  \circ \big(\Phi^\tau\pare{U}\big)^{-1}.
			$ 
			Note that $P^\tau(U)$ is gauge invariant being composition of gauge invariant operators. 
By Theorem 3.27 in \cite{BD} (actually adapting that result when the function $\beta$ is $2$-homogeneous rather than $1$-homogeneous), we have for any $ \tau\in[0,1]  $
\begin{equation}\label{P.tau}
\begin{aligned}
P^\tau(U) = &\vOpbw{\im \omega_\Phi^{(\alpha)} + \im \omega^{(\alpha-2)}} + R(U,\tau) \\
=&\vOpbw{\im \omega+ \im \omega_2^{(\alpha)}+ \im a_{\geq 4}^{(\alpha)}+\im \omega_2^{(\alpha-2)}+\im  a_{\geq 4}^{(\alpha-2)}} + R_2(U,\tau)+ R_{\geq 4}(U,\tau)
 \end{aligned}
\end{equation} 
 where $\omega_\Phi^{(m)} = \omega+ \omega_2^{(\alpha)}+ a_{\geq 4}^{(\alpha)}$ is a real symbol in $ \Sigma\Gamma^\alpha_0[r]$,
  $\omega^{(\alpha-2)}=\omega_2^{(\alpha-2)}+ a_{\geq 4}^{(\alpha-2)}$ is a symbol in $ \Sigma\Gamma^{\alpha-2}_2[r]$ and $R = R_2+ R_{\geq 4} \in \Sigma\cR^{-\varrho+\alpha}_2[r]$.
 
To identify the  quadratic component of $P^1(U)$  we use the Taylor expansion 
$
P^1(U) = P^0(U) + \pa_\tau P^\tau(U)\vert_{\tau =0}
+\int_0^1(1-\tau) \pa_\tau^2 P^\tau(U) \di \tau 
$
and  exploit  that $P^\tau(U)$  fulfills the  Heisenberg equation 
			$
			\partial_\tau P^\tau(U)= [G(U,\tau),{P^\tau(U)}] , \quad P^0(U) = \vOpbw{\im\omega}$ .	
Using that $G(U,0) = \vOpbw{\im b(U)\xi}$ and the  paradifferential structure  of $P^\tau(U)$ in 
\eqref{P.tau}, we obtain 
$$
P^1(U) = \vOpbw{\im \omega} + [\vOpbw{\im b(U)\xi}, \vOpbw{\im \omega}] 
+ M_{\geq 4}(U)
$$
with $M_{\geq 4}(U)$ a $\alpha$-operator in $\cM_{\geq 4}^{\alpha}[r]$.
Now we use the composition Theorem \ref{teoremadicomposizione} (with $\vr\leadsto \vr+1$) and formula  \eqref{espansione2} to expand the commutator as 
\begin{equation}\label{P12}
P^1(U) = \vOpbw{\im \omega + \im a_2^{(\alpha)}} + R_2(U) + M_{\geq 4}(U)
\end{equation}
with 
$a_2^{(\alpha)}(U;x, \xi) $ the real, zero-average symbol 
			\begin{equation}\label{corr.w}
			a^{\pare{m}}_{2 }\pare{U;x, \xi} := \sum_{k=1}^{\varrho+1} \frac{(-1)^k -1}{2^k k!} (D_x^k \beta)\, (\pa_\xi^k \omega )\,  \im \xi  \in \wt \Gamma^\alpha_2 \ ,
			\end{equation}
and $R_2(U) \in \wt \cR^{-\varrho+\alpha}_2$.
Identifying the quadratic  components of $P^1(U)$ in \eqref{P.tau} and \eqref{P12} we get that 
$$
\Opbw{\omega_2^{(\alpha)} + \omega_2^{(\alpha-2)}}= \Opbw{a_2^{(\alpha)}} + \tilde R_2(U)
$$
and therefore we get the thesis.
Since $\beta(U)$ is gauge invariant (fulfills the first of \eqref{sym.gauge}), so is   $a_2^{(\alpha)}$ in \eqref{corr.w}.
Finally,  since $P^1(U)$ is gauge invariant, also $R_2(U)$ in \eqref{P12} is gauge invariant by difference. 

\noindent 3.		It follows as in \cite[Remark at pag. 89]{BD}  (see also \cite[Proposition A.2]{MMS} for details).
			
\noindent 4.
		{	Differentiating 
            $$
\begin{cases} 
	\partial_\tau \Phi^\tau(U(t))= G(U(t))\Phi^\tau(U(t))\\
	\Phi^0(U(t))=\uno,
\end{cases}$$
with respect to time, we get that 
  $\pa_t \Phi^\tau(U(t))$ fulfills the variational equation 
\be 
\begin{cases}
\partial_\tau  \, \left(\pa_t \Phi^\tau (U(t))\right) =G(U(t)) \, \left( \pa_t \Phi^\tau(U(t))\right) 
+ \left( \pa_t G(U(t))\right) \,  \Phi^\tau(U(t)) \\
\pa_t \Phi^0(U(t))=0 
\end{cases}  \ , 
\ee
whose solution is given by the  Duhamel formula 
\begin{align}
\left(\pa_t \Phi^\tau(U(t))\right)=& \Phi^\tau(U(t)) \int_0^\tau  \Phi^{\tau_1}(U(t))^{-1} \left(  \pa_t G(U(t))\right)   \Phi^{\tau_1} (U(t))\, {\rm d}\tau_1  \ .
\end{align}
Evaluating at $\tau =1$,   applying $\Phi^1(U)^{-1}$ to the right 
and using that in our case $\pa_t G(U(t)) = \vOpbw{ \partial_t b(U,\tau_1;x) \im \xi } $
yields
			\be \label{conju}
			\big(\pa_t \Phi^1\pare{U} \big) \Phi^1\pare{U}^{-1}= \int_0^1  \Phi^1 \pare{U} \left[\Phi^{\tau_1} \pare{U}\right]^{-1} \vOpbw{ \partial_t b(U,\tau_1;x) \im \xi }   \Phi^{\tau_1} \pare{U}[\Phi^1 \pare{U}]^{-1}\, \di \tau_1.
			\ee
}
			We claim that  
			\be\label{betamaggiore}
			\partial_t b(U,\tau;x)= \beta(-\ii  \vOmega(D)U;x)+ \fV_{\geq 4}(U,\tau;x), \quad \fV_{\geq 4}\in \mF_{\geq 4}^{\R}[r] \ .
			\ee
			Differentiating  $b(U(t),\tau;x)$ with respect to $ t $ and using that, by equation \eqref{eq.U},  $\pa_t \beta(U) = 2 \beta(\pa_t U, U) = 2\beta(X(U), U)$ 
   with $X(U) = - \im \vOmega(D) U + X_3(U)$ we get  
			\begin{align*}
& 				\pa_t  b(U,\tau;x)
				= 2\beta\pare{-\ii  \vOmega(D) U, U  ; x}\\
				& + 
				\underbrace{
				2\beta\pare{M_{\mathtt{NLS}}\pare{U}U, U  ; x}  - 2\tau \left[ \frac{\beta\pare{U ; x}\beta_x\pare{X\pare{U},U; x}}{(1+\tau \beta_x\pare{U; x })^2}
				+
				\frac{\beta_x\pare{U; x}\beta\pare{X\pare{U}, U; x}}{(1+\tau \beta_x\pare{U; x })}\right]}_{=:\fV_{\geq 4}(U,\tau;x)} \ .
			\end{align*}
			Then \eqref{betamaggiore} follows  using Lemma \ref{nuovetto}--$1$ for each internal composition, getting that 
			$
			\fV_{\geq 4}(U,\tau;x)$
			is a function in $\mF_{\geq 4}^{\R}[r]$.

	\end{proof}

\paragraph{Conjugation by flows generated by linear smoothing operators.} 
In this section we study the conjugation rules for a flow 
$\Upsilon(U)\defeq \Upsilon^{\tau}(U)\vert_{\tau =1}$ generated by 
\begin{align}\label{flusso smoothing}
	\partial_\tau \Upsilon(U) = Q(U) \circ \Upsilon^\tau(U),  \ \ 
	\Phi^0_{Q}(U) = \uno \,,
\,
\end{align}
with $Q(U)$ a matrix of smoothing operators in $\wt\cR^{-\varrho}_2$.
We denote the inverse of $\Phi_{Q}(u)$ as 
$\Upsilon(U)^{- 1} = \left. \Upsilon^\tau(U)\right|_{ \tau = - 1}$. 

\smallskip

The following result is a small variation of \cite[Proposition A.5]{MMS} and we omit the proof.
\begin{proposition}[Conjugation by flows generated by smoothing operators]\label{prop:Egorov_smoothing}
	Let $m\in \R$, $\varrho, \vr', r >0$. Let   $Q(U)$ be a matrix of smoothing operators in $\widetilde{\mathcal{R}}^{-\varrho}_2$ and $\Upsilon (U)$ be the flow generated by $ Q(U) $ as in \eqref{flusso smoothing}. 
	Then the following holds:
	
	\begin{itemize}
		\item[$i)$] {\bf Space conjugation:} \label{item:Egorov_smoothingi} If $ a\in  \Sigma\Gamma^m_2[r] $,   then 
		\begin{align}\label{eq:conj_symbol_smoothing1}
			\Upsilon(U)\circ \vOpbw{a\pare{U;x,\x}} \circ  \Upsilon(U)^{-1}
			-  \vOpbw{a \pare{U;x,\x}}    &\in  \cR^{-\vr + \max\{m,0\}}_{\geq 4 }[r],\\
			\label{eq:conj_symbol_smoothing2}
			\Upsilon(U)\circ  (-\im \vOmega(D) )  \circ \Upsilon(U)^{-1}- \pare{- \im \vOmega(D)+ [Q(U), -\im \vOmega(D)] } & \in  \cR^{-\vr + \alpha}_{\geq 4 }[r] \ . 
		\end{align}
These matrices of operators are real-to-real provided $Q(U)$ is. 
		\item[$ii)$] {\bf Conjugation of smoothing operators:} \label{item:Egorov_smoothingii}If $ R(U)$ is a real-to-real matrix of smoothing operators in $\Sigma\cR^{-\vr'}_2[r] $, then 
		\be \label{bottadevita}
		\Upsilon(U)\circ R\pare{U}  \circ  \Upsilon(U)^{-1} - R\pare{U} \in  \cR^{-\min\{\vr,\vr'\}}_{\geq 4 }[r] \ 
		\ee
		and it is real-to-real.
		\item [$iii)$]  \label{item:Egorov_smoothing_iii} {\bf Conjugation of $\partial_t$:} If $ U $ is a  solution of \eqref{eq.U} then
		\begin{equation}\label{claimed}
			(\partial_t   \Upsilon\pare{U})\circ\Upsilon\pare{U}^{- 1} -  2  Q\pare{ -\im \vOmega(D) U, U} \in  \cR^{-\vr+1}_{\geq  4}\bra{r} \ 
		\end{equation}
		and it is real-to-real.
	\end{itemize}
\end{proposition}

\footnotesize{
\noindent {\bf Acknowledgments.}
We thank R. Grande for useful discussions, and the anonymous referee for the careful and in depth reading of the manuscript.
A. Maspero is supported by the European Union  ERC CONSOLIDATOR GRANT 2023 GUnDHam, Project Number: 101124921
and  by  PRIN 2022 (2022HSSYPN)   ``TESEO - Turbulent Effects vs Stability in Equations from Oceanography'', and GNAMPA.
F.\ Murgante is supported by the ERC STARTING GRANT 2021  HamDyWWa, Project Number: 101039762.
 Views and opinions expressed are however those of the authors only and do not necessarily reflect those of the European Union or the European Research Council. Neither the European Union nor the granting authority can be held responsible for them.
}


\vspace{-1em}
{\footnotesize 
\begin{thebibliography}{12} 



\bibitem{BGMR1}
Bambusi D., Gr{\'e}bert B., Maspero A., Robert D.
\newblock{\em Reducibility of the quantum harmonic oscillator in d-dimensions with
  polynomial time-dependent perturbation.}
\newblock {Anal. PDE}, 11(3):775--799, 2018.

\bibitem{BGMR2}
Bambusi D., Gr{\'e}bert B., Maspero A., Robert D.
\newblock {\em Growth of Sobolev norms for abstract linear Schr{\"o}dinger
  equations}.
\newblock {J. Eur. Math. Soc.}, 23(2): 557--583,  2021.


\bibitem{BGMRV}
Bambusi D., Grébert B., Maspero A., Robert D., Villegas-Blas C.  
\newblock {\em Longtime dynamics for the Landau Hamiltonian with a time dependent magnetic field.}  
\newblock EMS Surveys  Math. Sc.,  12(1): 155--184, 2025. 


\bibitem{BL}
Bambusi D., Langella B.
\newblock{\em Growth of Sobolev norms in quasi integrable quantum systems.}
\newblock  Ann.  Sci. Ec. Norm. Sup., 58(4): 997-1035, 2025.

\bibitem{BLM}
Bambusi D., Langella B., Montalto R.
\newblock{\em Growth of Sobolev norms for unbounded perturbations of the Schr\"odinger equation on flat tori.}
\newblock J.  Diff.  Eq.,  318: 344--358, 2022.

%
\bibitem{BBHM} Baldi P., Berti M., Haus E., Montalto R., 
\newblock{\em Time quasi-periodic gravity water waves
in finite depth}. 
\newblock Invent. Math. 214(2): 739-911, 2018.
%
\bibitem{BE}
 Biasi A., Evnin O.
\newblock{\em 
Turbulent cascades in a truncation of the cubic Szegő equation and related systems.}
\newblock Anal. PDE, 
 15(1),  217--243, 2022.

\bibitem{BCGS}
Berti M., Cuccagna S., Gancedo F., Scrobogna S.  
\newblock {\em Paralinearization and extended lifespan for solutions of the $\alpha$-SQG sharp front equation.}  
\newblock Adv. Math., 460:110034, 2025.



\bibitem{BD}
Berti M.,  Delort J.-M., 
\newblock 
{\it Almost Global Solutions of Capillary-gravity Water Waves Equations on the Circle}. 
\newblock UMI Lecture Notes 2018, 
ISBN 978-3-319-99486-4.  



\bibitem{BFF}  Berti M., Feola R., Franzoi L., 
{\it Quadratic life span of periodic gravity-capillary water waves}.
Water Waves 3(1): 85-115, 2021.


\bibitem{BFP}  Berti M., Feola R., Pusateri F., 
{\it Birkhoff Normal Form and Long Time Existence for Periodic Gravity Water Waves}.
Commun.  Pure Appl. Math., 
76(7), 1416--1494, 2023.


%
 \bibitem{BFM2}
Berti M., Franzoi L., Maspero A.,
 {\it  Pure gravity traveling quasi-periodic water waves with constant vorticity}.
 Commun. Pure Appl.  Math., 77(2), 990-1064, 2024.
%
%
%

\bibitem{BMM}
Berti M., Maspero A., Murgante F.,
\newblock{\it Local well posedness of the Euler-Korteweg
equations on $\mathbb{T}^{d}$}. 
\newblock{J. Dyn. Diff. Equat.}, 
33: 1475-1513,  2021.

\bibitem{BMM2}
Berti M., Maspero A., Murgante F.,
\newblock{\it Hamiltonian Birkhoff normal form for gravity-capillary water waves with constant vorticity: almost global existence.
}
\newblock{ Ann.  PDE}, 10(22), 2024. 

\bibitem{bou96}
Bourgain, J.,
\newblock{\it On the growth in time of higher Sobolev norms of smooth solutions of Hamiltonian PDE.}
\newblock{Int. Math. Res. Not.} 6:277–304,  1996.

\bibitem{bou99}
Bourgain J.
\newblock {\em Growth of {S}obolev norms in linear {S}chr\"odinger equations with
  quasi-periodic potential.}
\newblock {Comm. Math. Phys.}, 204(1):207--247, 1999.

\bibitem{bourgain99}
Bourgain J.
\newblock On growth of {S}obolev norms in linear {S}chr\"odinger equations with smooth time dependent potential.
\newblock {\em J. Anal. Math.}, 77:315--348, 1999.

\bibitem{bou2000}
 Bourgain J. 
 \newblock{\em Problems in Hamiltonian PDE’s.}
 \newblock Geom. Funct. Anal., Special Volume, Part I:32–56, 2000.







\bibitem{CKSTT}
Colliander J., Keel M., Staffilani G., Takaoka H., Tao T.
\newblock {\em Transfer of energy to high frequencies in the cubic defocusing
  nonlinear {S}chr\"odinger equation.}
\newblock {Invent. Math.}, 181(1):39--113, 2010.

\bibitem{CE}
Constantin A., Escher J. 
\newblock{\em Wave breaking for nonlinear nonlocal shallow water equations.}  
\newblock Acta Math. 181 229–43, 1998.

\bibitem{Christ}
 Christ M. 
 \newblock{\em Illposedness of a Schrödinger equation with derivative nonlinearity.}
 \newblock Preprint, \url{https://math.berkeley.edu/~mchrist/Papers/dnls.ps}

\bibitem{Delort10}
Delort J.M.
\newblock {\em Growth of {S}obolev norms of solutions of linear {S}chr\"odinger
  equations on some compact manifolds.}
\newblock {\em Int. Math. Res. Not.}, (12):2305--2328, 2010.


\bibitem{del}
Delort J.M.
\newblock{\em Growth of {S}obolev norms for solutions of time dependent
  {S}chr\"odinger operators with harmonic oscillator potential.}
\newblock {Comm. Par. Diff.  Eq. }, 39(1):1--33, 2014.

\bibitem{DH}
Deng, Y., Hani, Z. 
\newblock{\em Full derivation of the wave kinetic equation.}
\newblock Invent. math. 233, 543–724, 2023.

\bibitem{EK2022}
Elgindi T., Shikh Khalil  K. 
\newblock{\em Strong Ill-Posedness in $L^\infty$ for the Riesz Transform Problem.}
\newblock{Anal.  PDE}, 18(3): 715--741, 2025.

\bibitem{FaouRaph}
Faou E.,  Raphael P.
\newblock{\em On weakly turbulent solutions to the perturbed linear harmonic
  oscillator.}
\newblock{Am.  J.   Math. }, 145(5):  1465--1507, 2023.



\bibitem{FG} 
Feola R. , Giuliani F.
\newblock {\it Quasi-periodic traveling waves on an infinitely deep fluid under gravity}. 
 \newblock 
 Mem.  Amer.  Math.  Soc. , Vol. 295, 2024.



\bibitem{FGI} 
Feola R.,  Grebert B., Iandoli F.
\newblock{\em Long time solutions for quasi-linear Hamiltonian perturbations of Schrödinger and Klein-Gordon equations on tori}
\newblock{Anal. \& PDE}, 16(5):1133--1203,  2023.


\bibitem{FI} 
Feola R. , Iandoli F.
\newblock{\em Long time existence for fully nonlinear NLS with small Cauchy data on the circle.}
\newblock{Annali SNS},  22(1), 2021.

\bibitem{Gallone}
Gallone M.,  Marian M.,  Ponno A., Ruffo S.
\newblock{\it Burgers turbulence in the Fermi-Pasta-Ulam-Tsingou chain.}
\newblock Phys. Rev. Lett. 129,  114101, 2022.

\bibitem{gerard_grellier0}
Gérard P., Grellier S. 
\newblock {\it The cubic Szeg\H{o} equation. }
\newblock Ann. Sci. Éc. Norm. Supér. 43(4), 761–810, 2010

\bibitem{gerard_grellier}
Gérard P., Grellier S.
\newblock {\it The cubic {S}zeg{\H o} equation and {H}ankel operators.}
\newblock {Ast\'erisque}, (389):vi+112, 2017.



\bibitem{gerard_grellier3}
Gérard P., Grellier S.
 \newblock{\em Effective integrable dynamics for a certain nonlinear wave equation.} 
 \newblock Anal.\& PDE, 5, 1139--1155, 2012.

 
\bibitem{GL}
Gérard P.,  Lenzmann E.
\newblock{\em The Calogero--Moser Derivative Nonlinear Schr\"odinger Equation.}
\newblock{Commun. Pure Appl. Math.}, 2024. doi:10.1002/cpa.22203

\bibitem{GLPR}
Gerard, P., Lenzmann, E., Pocovnicu, O.,  Raphael, P. 
\newblock{ \em A two-soliton with transient turbulent regime for the cubic half-wave equation on the real line.}
\newblock{Ann. PDE}, 4(1), Article 7, 2018. 


\bibitem{GG}
Giuliani F., Guardia M.
\newblock{\em Sobolev norms explosion for the cubic NLS on irrational tori}.
\newblock{Nonlinear  Anal.},  220,  112865, 2022.


\bibitem{GG2}
Giuliani F., Guardia M.
\newblock{\em Arnold diffusion in Hamiltonian systems on infinite lattices}.
\newblock{Commun. Pure Appl. Math.}, 77(8):  3333--3426, 2024.


\bibitem{Giu}
Giuliani F.  
\newblock{\em Sobolev instability in the cubic NLS equation with convolution potentials on irrational tori.}  
\newblock J. Diff. Eq., 416(1), 1--27, 2025.



\bibitem{GHMMZ}
Guardia M., Haus E.,  Hani Z., Maspero A.,  Procesi M. 
\newblock {\em Strong nonlinear instability and growth of
Sobolev norms near quasiperiodic finite-gap tori for the 2D cubic NLS equation.}
\newblock J. Eur. Math. Soc., 25(4) 1497--1551, 2022.


\bibitem{guardia_haus_procesi16}
Guardia M., Haus E.,   Procesi M.
\newblock{\em Growth of {S}obolev norms for the analytic {NLS} on {$\Bbb{T}^2$}.}
\newblock {Adv. Math.}, 301:615--692, 2016.

\bibitem{guardia_kaloshin}
Guardia M.,  Kaloshin V. 
\newblock{\em Growth of Sobolev norms in the cubic defocusing nonlinear Schr\"odinger 
equation. }
\newblock{J. Eur. Math. Soc.}, 17(1):71--149, 2015.


\bibitem{hani14}
Hani Z.
\newblock{\em Long-time instability and unbounded {S}obolev orbits for some
  periodic nonlinear {S}chr\"odinger equations.}
\newblock {Arch. Rational Mech. Anal.}, 211(3):929--964, 2014.


\bibitem{Hasselmann1}
 Hasselmann K.
 \newblock{\em On the nonlinear energy transfer in a gravity wave spectrum. Part 1.}
 \newblock{J. Fluid Mech.} 12, 481--500, 1962.

 \bibitem{Hasselmann2}
 Hasselmann K.
  \newblock{\em On the nonlinear energy transfer in a gravity wave spectrum. Part 2.}
  \newblock{J. Fluid Mech.} 15, 273--281, 1963.

\bibitem{haus_procesi15}
 Haus E.,   Procesi M.
\newblock {\em Growth of {S}obolev norms for the quintic {NLS} on {$\T^2$}.}
\newblock {Anal. PDE}, 8(4):883--922, 2015.

\bibitem{hani15}
Hani Z., Pausader B., Tzvetkov N., and Visciglia N.
\newblock{\em Modified scattering for the cubic {S}chr\"odinger equation on product
  spaces and applications.}
\newblock {Forum Math. Pi}, 3:e4, 63, 2015.


\bibitem{HausMaspero}
Haus E.,   Maspero A.
\newblock{\em Growth of Sobolev norms in time dependent semiclassical anharmonic oscillators}. 
\newblock{J. Funct. Anal.} , 278(2), 108316, 2020.



\bibitem{H1}
Hur, V. 
\newblock{\it Wave breaking in the Whitham equation.}
\newblock Adv. Math. 317, 410–437, 2017.

 
\bibitem{H2}
Hur V., Tao L.
\newblock{\it  Wave breaking for the Whitham equation with fractional dispersion.} 
\newblock Nonlinearity 27(12), 2937–2949, 2014.




\bibitem{KS}
Klein C., Saut J.-C.
\newblock{\it  A numerical approach to blow-up issues for dispersive perturbations of Burgers’ equation.}
\newblock Phys. D 295(296), 46–65, 2015.

\bibitem{KSW}
Klein C., Saut J.-C., Wang, Y.
\newblock{\it On the modified fractional Korteweg-de Vries and related equations.}
\newblock Nonlinearity, 35(3), 1170, 2022.


\bibitem{Kuk96}
Kuksin, S.
\newblock{ \it Growth and oscillations of solutions of nonlinear Schrödinger equation.}
\newblock Commun. Math. Phys. 178, 265–280, 1996.

\bibitem{Kuk97}
Kuksin,S. 
\newblock{\it On turbulence in nonlinear Schrödinger equations.}
\newblock Geom. Funct. Anal. 7,783–822, 1997.

\bibitem{LZZ}
Liang Z.,  Zhao Z. and Zhou Q.
\newblock{\em  1-d quantum harmonic oscillator with time quasi-periodic quadratic perturbation: reducibility and growth of Sobolev norms.}
\newblock{ J. Math. Pures Appl.} 146(1): 158--182 (2021).

\bibitem{LLZ2}
Liang Z., Luo J.,  Zhao Z.
\newblock{\em Symplectic Normal Form and Growth of Sobolev Norm.}
\newblock J. Diff. Eq.,  449,  113702, 2025.

\bibitem{LH}
Longuet-Higgins M. and Gill A.
\newblock{\em Resonant interactions between planetary waves.}
\newblock{Proc. R. Soc. Lond. A} 299, 120--144, 1967.


\bibitem{LLZ}
 Luo J.,  Liang Z. and  Zhao Z.
 \newblock{\em Growth of Sobolev Norms in 1-d Quantum Harmonic Oscillator with Polynomial Time Quasi-periodic Perturbation. }
\newblock Commun. Math. Phys. 392,  1-–23, 2022.




\bibitem{Mas19}
Maspero A.
\newblock {\em Lower bounds on the growth of {S}obolev norms in some linear time
  dependent {S}chr\"{o}dinger equations.}
\newblock {Math. Res. Lett.}, 26(4):1197--1215, 2019.

\bibitem{Mas21}
Maspero A.
\newblock {\em Growth of Sobolev norms in linear Schr\"odinger equations as a dispersive phenomenon.}
\newblock  Adv.  Math., 411(A), 2022.

\bibitem{Mas22}
Maspero A.
\newblock {\em Generic transporters for the linear time dependent quantum Harmonic oscillator on $\R$. }
\newblock Int.
Math. Res. Not. , rnac174, 2022.



\bibitem{MaRo}
Maspero A., Robert D.
\newblock On time dependent {S}chr{\"o}dinger equations: {G}lobal
  well-posedness and growth of {S}obolev norms.
\newblock {\em J.  Fun. Anal.}, 273(2):721 -- 781, 2017.

\bibitem{Metivier}
Metivier G. 
\newblock{\em Para-differential Calculus and Applications to the Cauchy Problem for Nonlinear
Systems.}
 Università di Pisa,
  ffcel-00287554f, 2007.

\bibitem{MMS}
Montalto, R., Murgante, F., Scrobogna, S.
\newblock {\it Quadratic lifespan for the sublinear
$\alpha$-SQG sharp front problem.}
\newblock { J. Dyn. Diff. Equat.}, (2024). https://doi.org/10.1007/s10884-024-10400-8


		

\bibitem{Mourre}
Mourre E.
\newblock {\em Absence of singular continuous spectrum for certain selfadjoint
  operators.}
\newblock {Commun. Math. Phys.}, 78(3):391--408, 1980/81.



\bibitem{OP}
Oh, SJ., Pasqualotto, F. 
\newblock {\em Gradient Blow-Up for Dispersive and Dissipative Perturbations of the Burgers Equation.}
\newblock Arch. Rational Mech. Anal. 248, 54, 2024.


\bibitem{PTV}
Planchon, F., Tzvetkov, N., Visciglia, N.
\newblock{\em On the growth of Sobolev norms for NLS on 2- and
3-dimensional manifolds.}
\newblock{Anal. \& PDE} 10, 1123--1147, 2017.

\bibitem{S}
Saut J.-C., Wang Y.
\newblock{\it  Global dynamics of small solutions to the modified fractional Korteweg–de Vries and fractional cubic nonlinear Schrödinger equations.}
\newblock Commun. Part. Differ. Eq. 46(10), 1851–1891, 2021.



 \bibitem{SS}
 Sigal I., Soffer A.
 \newblock{\em Local decay and velocity bounds for quantum propagation.}
 \newblock {Preprint}, 1988.
 \url{http://www.math.toronto.edu/sigal/publications/SigSofVelBnd.pdf}

\bibitem{Sohinger}
Sohinger, V.
\newblock{\em Bounds on the growth of high {S}obolev norms of solutions to
  nonlinear {S}chr\"odinger equations on {$S^1$}.}
\newblock {Differ. Integral  Equ.}, 24(7-8):653--718, 2011.

\bibitem{Staffilani}
 Staffilani G.
 \newblock{ \em On the growth of high Sobolev norms of solutions for KdV and Schr\"odinger
equations}.
\newblock Duke Math. J.,  86:109--142, 1997. 

\bibitem{SS}
Sukhatme J. and  Smith  L. 
\newblock{\em  Local and nonlocal dispersive turbulence.}
\newblock{Physics of Fluids, 21(5)}, 056603--, 2009.


\bibitem{Y}
Yang R.
\newblock{\it Shock formation for the Burgers–Hilbert equation.}
\newblock SIAM J. Math.  Anal.  53(5), 2021.

\bibitem{ZakD}
Zakharov V.E. and  Dyachenko A.I.
\newblock{\it  Is free-surface hydrodynamics an integrable system?}
\newblock  Physics Letters A 190, 
144-148, 1994.


\end{thebibliography}}
\end{document}